\documentclass[11pt, a4paper]{amsart}

\title[Classification of $q$-pure $q$-weight maps]
{Classification of $q$-pure $q$-weight maps over finite dimensional Hilbert
spaces}
\author{Christopher Jankowski}
\address{Christopher Jankowski, School of Mathematics,
Georgia Institute of Technology,
686 Cherry St, Atlanta, GA 30332.}
\email{cjankowski3@gatech.edu}

\author{Daniel Markiewicz}
\address{Daniel Markiewicz,
Department of Mathematics,
Ben-Gurion University of the Negev,
P.O.B. 653, Beersheva 84105,
Israel.}
\email{danielm@math.bgu.ac.il}

\author{Robert T. Powers}
\address{Robert T. Powers,
Department of Mathematics,
University of Pennsylvania,
David Rittenhouse Lab.,
209 South 33rd Street ,
Philadelphia, PA 19104-6395.}
\email{rpowers@math.upenn.edu}
\keywords{$E_0$-semigroups, $CP$-Flows, $q$-pure, $q$-positive}
\subjclass[2010]{Primary 46L55, 46L57}

\date{\today}

\usepackage[a4paper, margin=2.35cm]{geometry}
\usepackage{amsmath, amssymb}
\usepackage[all]{xy}
\usepackage{verbatim}
\usepackage{enumerate}
\usepackage{multirow}

\newtheorem{thm}{Theorem}[section]
\newtheorem{lem}[thm]{Lemma}

\newtheorem{defn}[thm]{Definition}

\numberwithin{equation}{section}
\begin{document}

\begin{abstract}
An $E_0$-semigroup of $B(H)$ is a one parameter strongly continuous
semigroup of $*$-endomorphisms of $B(H)$ that preserve the identity.  Every
$E_0$-semigroup that possesses a strongly continuous intertwining semigroup
of isometries is cocycle conjugate to an $E_0$-semigroup induced by the
Bhat induction of a $CP$-flow over a separable Hilbert space $K$.  We say
an $E_0$-semigroup $\alpha$ is $q$-pure if the $CP$-subordinates $\beta$ of
norm one (i.e.  $\Vert\beta_t(I)\Vert = 1$ and $\alpha_t-\beta_t$ is
completely positive for all $t \geq 0$) are totally ordered in the sense that
if $\beta$ and $\gamma$ are two $CP$-subordinates of $\alpha$ of norm one,
then $\beta \geq \gamma$ or $\gamma \geq \beta$.  This paper shows how to
construct and classify all $q$-pure $E_0$-semigroups induced by $CP$-flows
over a finite-dimensional Hilbert space $K$ up to cocycle conjugacy.
\end{abstract}

\maketitle
\begin{center}  Introduction \end{center}

An $E_0$-semigroup of $B(H)$ is a one parameter strongly
continuous semigroup of $*$-endomorphisms of $B(H)$ (the set of
all bounded operators on a Hilbert space $H$) that preserve the
identity, (i.e. $\alpha_t(I) = I$ for all $t \geq 0)$.  Two
$E_0$-semigroups $\alpha$ and $\beta$ of $B(H_1)$ and $B(H_2)$ are
said to be conjugate if there is a unitary operator $U$ from $H_1$
onto $H_2$ so that $\beta_t(UAU^{-1}) =  U\alpha_t(A)U^{-1}$ for
all $A \in B(H_1)$ and $t \geq 0$.  If $\alpha$ is an
$E_0$-semigroup of $B(H)$ we say $U = \{ U(t):t \geq 0\}$ is a
cocycle for $\alpha$ if the operators $U(t)$ are strongly
continuous in $t$ and satisfy the cocycle relation $U(t+s) =
U(t)\alpha_t(U(s))$ for $t,s \geq 0$.  The cocycle is said to be
unitary if the $U(t)$ are unitary operators.  Two $E_0$-semigroups
$\alpha$ and $\beta$ are said to be cocycle conjugate if there is
a unitary cocycle $U(t)$ for $\alpha$ so that the $E_0$-semigroup
$\gamma$ given by $\gamma_t(A) = U(t)\alpha_t(A)U(t)^{-1}$ for $A
\in B(H)$ and $t \geq 0$ is conjugate with $\beta$.  The main
problem in the theory of $E_0$-semigroups is to classify them up
to cocycle conjugacy.  For a discussion of $E_0$-semigroups we
refer to the book of Arveson \cite{arv-monograph}.

An $E_0$-semigroup $\alpha$ of $B(H)$ is said to be spatial if there is a
strongly continuous one parameter semigroup of isometries $U(t)$ which
intertwine $\alpha$ so that $U(t)A = \alpha_t(A)U(t)$ for $A \in B(H)$ and
$t \geq 0$.  For spatial $E_0$-semigroups there is an integer valued index
$n = 0,1,\cdots $ and $\infty$ first suggested by Powers \cite{P1} and
later correctly defined by Arveson \cite{arv-monograph}.  Arveson showed that
the index is additive under taking tensor products so if $\alpha$ and $\beta$
are spatial $E_0$-semigroups of index $n$ and $m$ then the tensor product
$\alpha_t\otimes \beta_t$ is spatial and of index $n + m$.  If an
$E_0$-semigroup can be reconstructed from its intertwining semigroups it is
said to be completely spatial.  These are the $E_0$-semigroups of
type~I$_n$ where $n$ is the index.  Arveson showed that the index $n$ is a
complete cocycle conjugacy invariant for the $E_0$-semigroups of type~I
(i.e. if $\alpha$ and $\beta$ are of type~I$_n$ then $\alpha$ and $\beta$
are cocycle conjugate.)  Like the theory of factors the $E_0$-semigroups of
type~I are well understood.

An $E_0$-semigroup that is spatial but not completely spatial is said to be
of type~II and $E_0$-semigroups which are not spatial are said to be of
type~III.  In this paper we focus on what we call the $q$-pure
$E_0$-semigroups of type~II.  If $\alpha$ is an $E_0$-semigroup of $B(H)$
we say $C = \{ C(t):t \geq 0\}$ is a local cocycle for $\alpha$ if $C$
satisfies the cocycle condition $C(t+s) = C(t)\alpha_t(C(s))$ for $t,s \geq
0$ and $C(t)$ commutes with $\alpha_t(A)$ for $A \in B(H)$ (i.e., $C(t) \in
\alpha_t(B(H))^{\prime})$.  Note if $C$ is a positive local cocycle for
$\alpha$ and $s \leq 0$ then $D(t) = e^{st}C(t)$ for $t > 0$ is also a
positive local cocycle for $\alpha$.  We can exclude these trivial subordinates
by only considering local cocycle of norm one (i.e. $\Vert C(t)\Vert =
1$ for $t \geq 0$).  If $C_1$ and $C_2$ are local cocycles for $\alpha$
then we say $C_1 \geq C_2$ if $C_1(t) \geq C_2(t)$ for all $t \geq 0$.  The
positive local cocycles and their order structure is a cocycle conjugacy
invariant for $\alpha$.  We say an $E_0$-semigroup is $q$-pure if the
positive local cocycles of norm one are totally ordered.  This means if $C_1$
and $C_2$ are positive local cocycles of norm one then either $C_1(t) \geq
C_2(t)$ or $C_2(t) \geq C_1(t)$ for all $t \geq 0$.  If $\alpha$ is a $q$-pure
spatial $E_0$-semigroup then $\alpha$ is of index zero which means that either
$\alpha$ is of type~I$_o$ so $\alpha_t(A) = U(t)AU(t)^{-1 }$ for $U$ a
strongly continuous one parameter unitary group or $\alpha$ is an
$E_0$-semigroup of type~II$_o$.

In this paper we begin the classification of the $q$-pure spatial
$E_0$-semigroups.  We classify all the $q$-pure spatial $E_0$-semigroups
that come from boundary weight maps over finite dimensional spaces.  We
find that all such $q$-pure $q$-weight maps are cocycle conjugate to
$q$-weight maps of range rank one.  These $q$-pure range rank one $q$-weight
maps were first constructed by Powers in \cite{powers-holyoke} and then by
Jankowski \cite{jankowski1} in the case of weight maps given by a pair
$(\phi ,\nu )$ where $\phi$ is a completely positive map of the
$(n \times n)$-matrices into themselves and $\nu$ is a pure weight on
$B(L^2(0,\infty ))$.  The surprising result of Jankowski was that the
$q$-weight map one constructs can be $q$-pure when $\phi$ has a one dimensional
range.  This result runs counter to intuition in that a map $\phi (A) =
tr(A)I$ (which is in a sense the least pure completely positive map being
the average of all pure maps) yields a $q$-pure weight map.  In \cite{jmp1} we
classified all the range rank one $q$-weight maps over a finite dimensional
Hilbert space $K$ up to cocycle conjugacy and in this paper we show that in
the case of $q$-pure $q$-weight maps over a finite dimensional Hilbert space
are all cocycle conjugate to $q$-pure range rank one $q$-weight maps.

What is surprising about this result is that a few years ago these problems
seemed intractable.  Even in the case of a two dimensional Hilbert space
$\mathbb C^2$ it was more of a hope than a belief that $q$-pure weight maps
were cocycle conjugate to range rank one $q$-weight maps.  It took months of
analysis to produce a single example of an non Shur $q$-pure $q$-weight map in
the case of a Hilbert space $K$ of dimension two.  It seemed then that the
structure of these boundary weight maps was so complex any sort of
classification was out of the question.  This paper shows how to construct
and classify all $q$-pure $E_0$-semigroups coming from boundary weight maps
over finite dimensional Hilbert spaces up to cocycle conjugacy.  Whether this
analysis carries over to the infinite dimensional case is the next burning
question.

\section{Boundary weight maps}

Each $q$-weight map over $K$ uniquely defines a $CP$-flow over
$K$.  A $CP$-flow $\alpha_t$ over $K$ is a strongly continuous one
parameter semigroup of completely positive contractions of $B(H) =
B(K\otimes L^2(0,\infty ))$ into itself that is intertwined by translation.
Specifically, if $U(t)$ is the translation isometry given by
$(U(t)F)(x) = F(x-t)$ for $t \geq 0$ where $F \in H$ is
represented by a $K$-valued function $F(x)$ then to say $\alpha_t$
is intertwined by $U(t)$ means that $U(t)A = \alpha_t(A)U(t)$ for
$A \in B(H)$ and $t \geq 0$.  A $CP$-flow $\alpha_t$ is unital if
$\alpha_t(I) = I$ for $t \geq 0$.  By the Bhat induction theorem
\cite{bhat-dilation} each unital $CP$-flow gives rise to a spatial
$E_0$-semigroup, where spatial means the $E_0$-semigroup is of
type~I or II.  It follows then that each unital $q$-weight map
over $K$ gives rise to a $E_0$-semigroup which is determined up to
cocycle conjugacy.  Two such $E_0$-semigroups are cocycle
conjugate if the $q$-weight maps are cocycle conjugate which is
explained below.  All this is explained in \cite{powers-CPflows}
in excruciating detail.  Fortunately, for our purposes we will
simply use the results and work directly with the $q$-weights.
With this said we begin.

In this paper all Hilbert spaces $H$ are assumed to be separable.  Finite
dimensional Hilbert spaces of dimension $n$ are denoted by $\mathbb C^n$.  The
inner product $(f,g)$ with $f,g \in H$ is linear in $g$ and conjugate
linear in $f$.

Of utmost importance the concept of a completely positive map.  If $H$ and
$K$ are Hilbert spaces a mapping $\phi$ of $B(H)$ into $B(K)$ is said to be
completely positive if for all $n \in \mathbb N$  if $A_i \in B(H)$ and $f
_i \in K$ for $i = 1,\cdots ,n$ then
\begin{equation*}
\sum_{i,j=1}^n (f_i,\phi (A^*_iA_j)f_j) \geq 0.
\end{equation*}
An alternative definition is as follows.  If $\phi$ is a linear mapping of
$B(H)$ into $B(K)$ we say $\phi$ is positive if $\phi (A)$ is positive if
$A$ is positive.  We denote by $\phi_n = \iota_n \otimes \phi$ the mapping of
$(n \times n)$-matrices with entries in $B(H)$ into $(n \times n)$-matrices
with entries in $B(K)$ given by
\begin{equation*}
\phi_n\left(\left[\begin{matrix} A_{11}&A_{12}&\cdots &A_{1n}
\\
A_{21}&A_{22}&\cdots &A_{2n}
\\
\cdots&\cdots&\cdots&\cdots
\\
A_{n1}&A_{n2}&\cdots &A_{nn}
\end{matrix} \right]\right)=
\left[\begin{matrix} \phi (A_{11})&\phi (A_{12})&\cdots &\phi (A_{1n})
\\
\phi (A_{21})&\phi (A_{22})&\cdots &\phi (A_{2n})
\\
\cdots&\cdots&\cdots&\cdots
\\
\phi (A_{n1})&\phi (A_{n2})&\cdots &\phi (A_{nn})
\end{matrix} \right].
\end{equation*}
A mapping $\phi$ is $n$-positive if $\phi_n(A)$ is positive if $A$ is
positive.  A mapping $\phi$ is completely positive if $\phi_n$ is positive
for all $n = 1,2,\cdots$.

An important estimate for completely positive maps is that if $\phi$
is a completely positive mapping of $B(H)$ into $B(K)$ then
\begin{equation*}
\phi (A)^*\phi (A) \leq \phi (A^*A)\Vert\phi (I)\Vert  \leq \Vert A\Vert^2
\Vert\phi (I)\Vert\phi(I)
\end{equation*}
for $A \in B(H)$.  We will refer to the inequality on the left as the
Schwarz inequality for completely positive maps.  This estimate is
derived by observing that
\begin{equation*}
\phi_2(
\left[\begin{matrix}
I&A
\\
A^*&A^*A
\end{matrix} \right])=
\left[\begin{matrix}
\phi (I)&\phi (A)
\\
\phi (A^*)&\phi (A^*A)
\end{matrix} \right] \leq
\left[\begin{matrix}
\Vert\phi (I)\Vert I&\phi (A)
\\
\phi (A^*)&\phi (A^*A)
\end{matrix} \right]
\end{equation*}
so the matrix of operators on the right is positive.  Taking the
operator norm of the Schwarz inequality gives
$\Vert\phi (A)\Vert \leq \Vert A\Vert\medspace \Vert\phi (I)\Vert$ so the
norm of a completely positive map $\phi$ as a map is the norm
$\Vert\phi (I) \Vert$.

Every normal completely positive map $\phi$ from $B(H)$ to $B(K)$ can be
expressed in the form
$$
\phi (A) = \sum_{k\in J} S_kAS_k^*
$$
where $J$ is a countable index set and the $S_k$ are bounded linear
operators from $H$ to $K$ which are linearly independent, by which we
mean if $c_k \in \mathbb C$ for $k \in J$ so that
$$
\sum_{k\in J} \vert c_k\vert^2 < \infty
$$
and
$$
S = \sum_{k\in J} c_kS_k,
$$
then $S = 0$ if and only if $c_k = 0$ for all $k \in J$.  If one has a
second decomposition of $\phi$ in the above manner and $J^{\prime}$ is the
index set for the second decomposition then there is a one to one mapping
of $J$ onto $J^{\prime}$.  The number of elements of $J$ is called the
index of $\phi$.  If the index set has one element we say $\phi$ is pure.
If $\phi$ is a completely positive map of $B(H)$ into $B(K)$ of the above
form and $\psi$ is a map of $B(H)$ into $B(K)$ of the form $\psi (A) =
SAS^*$ where $S$ is a linear operator from $H$ to $K$ then $\phi - \psi$ is
completely positive if and only if there are complex numbers $c_k$ for $k
\in J$ so that
$$
S = \sum_{k\in J} c_kS_k\qquad \text{and} \qquad \sum_{k\in J} \vert
c_k\vert ^2 \leq 1.
$$

We mention that given a map $\phi$ of $B(H)$ into $B(K)$ we will assume it
is normal unless stated otherwise.

We introduce the notation we will be using through out this paper.  First
let $K$ be a Hilbert space and let $H = K \otimes L^2(0,\infty )$, so $H$
can be thought of as Lebesgue measurable $K$-valued functions $F(x)$ for $x
\in [0,\infty )$ with inner product
$$
(F,G) = \int_0^\infty (F(x),G(x)) \,dx.
$$
We define $\Lambda$ as the mapping from $B(K)$ to $B(H)$ given by
$$
(\Lambda (A)F)(x) = e^{-x}AF(x)
$$
for $F \in H,\medspace A \in B(K)$ and $x \in [0,\infty )$.  We often use
the operator $\Lambda (I)$ which will we often simply write as $\Lambda$.
We hope the reader is not confused as $\Lambda$ can refer to the mapping
$\Lambda$ and the operator $\Lambda (I)$ but the difference can be inferred
from the context.  In almost all cases where we use $\Lambda$ to denote
$\Lambda (I)$ it occurs in the form $I - \Lambda$ in which case $\Lambda$
is clearly $\Lambda (I)$ since $I$ is the unit operator so when we write
the expression $I - \Lambda$ it always means $I - \Lambda (I)$.  We will
denote the identity mapping by $\iota$ so when we write $\iota + \omega
\Lambda$ then $\Lambda$ denotes the mapping $\Lambda$.

Given a Hilbert space $K$ we define $\mathfrak A (K)$ as the set of operators
in $B(H)$ where $H = K \otimes L^2(0,\infty )$ of the form
$$
A = (I - \Lambda )^{\tfrac{1}{2}} B(I - \Lambda )^{\tfrac{1}{2}}
$$
with $B \in B(H)$.  The norm on $\mathfrak A (K)$ is given by
$$
\Vert A\Vert_+ =\Vert (I-\Lambda )^{-\tfrac{1}{2}} A(I-\Lambda )^{-\tfrac{1}{2}}
\Vert
$$
where $\Vert\cdot\Vert$ is the usual norm on $B(H)$.  We denote by $\mathfrak A
(K)_*$ the linear functionals $\eta$ on $\mathfrak A (K)$ so that the
functional
$$
\rho (A) = \eta ((I-\Lambda )^{\tfrac{1}{2}}  A(I-\Lambda )^{\tfrac{1}{2}})
$$
for $A \in B(H)$ is $\sigma$-weakly continuous on $B(H)$.  Such a linear
functional $\eta$ is called a boundary weight or $b$-weight.  Now each
element of $\mathfrak A (K)_*$ is automatically bounded in the
$\Vert\cdot\Vert_+$ norm given by
$$
\Vert\eta\Vert_1 = \sup\{\vert\eta (A)\vert :A \in \mathfrak A (K),\medspace
\Vert A\Vert_+ \leq 1\} .
$$
When we say an element $\eta \in \mathfrak A (K)_*$  is bounded we mean it is
bounded with respect to the ordinary Hilbert space norm.  So $\eta$ is
bounded if
$$
\Vert\eta\Vert = \sup\{\vert\eta (A)\vert :A \in \mathfrak A (K),\medspace
\Vert A\Vert \leq 1\} < \infty .
$$

Recall $H = K \otimes L^2(0,\infty )$ so we can think of $H$ as $K$-valued
functions of $x$.  We define $L_+^2(0,\infty )$ as the set of Lebesgue
measurable functions $f$ so that
$$
\Vert f\Vert_+^2 = \int_0^\infty (1 - e^{-x})\vert f(x)\vert^2 \,dx <
\infty
$$
and $\Vert\cdot\Vert_+$ denotes the norm on $L_+(0,\infty )$.  We denote by
$H _+ = K \otimes L_+^2(0,\infty )$ where we can think of elements $F,G \in
H_+$ as $K$-valued functions where the inner product is given by
$$
(F,G)_+ = \int_0^\infty (1 - e^{-x})(F(x),G(x)) \,dx.
$$
Note each $F \in H_+$ can be written as $F = (I - \Lambda )^{-\tfrac{1}{2}}G$
with $G \in H$.

Note each $\eta \in \mathfrak A (K)_*$ can be expressed in the form
$$
\eta (A) = \sum_{k\in J} (F_k,AG_k)
$$
where
$$
\sum_{k\in J} \Vert F_k\Vert_+^2 = \sum_{k\in J} \Vert (I-\Lambda
)^{-\tfrac{1}{2} }F_k\Vert^2 < \infty
$$
and
$$
\sum_{k\in J} \Vert G_k\Vert_+^2 = \sum_{k\in J} \Vert (I-\Lambda
)^{-\tfrac{1}{2} }G_k\Vert^2 < \infty
$$
and $J$ is a countable index set.  Note if $\eta \in \mathfrak A (K)_*$ and
$\eta$ is bounded then $\eta$ can be expressed in the form
$$
\eta (A) = \sum_{k\in J} (F_k,AG_k)
$$
where
$$
\sum_{k\in J} \Vert F_k\Vert^2 < \infty\qquad \text{and} \qquad \sum_{k\in
J} \Vert G_k\Vert^2 < \infty
$$
so $\eta$ can be thought of as an element of the predual of $B(K\otimes
L^2(0, \infty ))$ (so $\eta \in B(K\otimes L^2(0,\infty ))_*)$.

We denote by $B(K) \otimes \mathfrak A (K)_*$ the set of all linear mappings
$\phi$ of $\mathfrak A (K)$ into $B(K)$ so that for all $f,g \in K$ the linear
functional $(f,\phi (A)g)$ is in $\mathfrak A (K)_*$.  Again we say such a
$\phi$ is bounded if
$$
\Vert\phi\Vert = \sup\{\Vert\phi (A)\Vert :A \in \mathfrak A (K),\medspace
\Vert A\Vert \leq 1\} < \infty .
$$

Next we introduce the cut off notation.  Recall $H = K \otimes L^2(0,\infty)$.
We define for $0 \leq a < b \leq \infty$ the projection $E(a,b)$ on
$H$ as the hermitian projection onto functions $F \in H$ with support in
the closed interval $[a,b]$ or $[a,\infty )$ in the case where $b =
\infty$.   If $A \in B(K)\otimes\mathfrak A (K)$ or $A \in B(K \otimes
L^2(0,\infty ))$ and $t > 0$ we denote by $A \vert_t  = E(t,\infty
)AE(t,\infty )$.  If $\eta$ is a $b$-weight we denote by $\eta \vert_t $
the functional
$$
\eta \vert_t (A) = \eta (E(t,\infty )AE(t,\infty ))
$$
for all $A \in B(K \otimes L^2(0,\infty ))$ and $t > 0$.  Note that for all
$t > 0$ if $A \in B(K \otimes L^2(0,\infty ))$ then we have $A\vert_t \in
 \mathfrak A (K)$ and if $\eta \in \mathfrak A (K)_*$ then $\eta \vert_t \in
B(K\otimes L^2(0,\infty ))_*$.

Finally we will denote the identity mapping $A \rightarrow A$ of $B(K)$
into itself by $\iota$ so $\iota (A) = A$ for all $A \in B(K)$.

Armed with this notation we can now define a $q$-weight map over $K$.

\begin{defn}
Suppose $K$ is a separable Hilbert space.  A $q$-weight map over $K$ is a
completely positive element $\omega \in B(K) \otimes \mathfrak A (K)_*$ so
that for each $t > 0$ the mapping $(\iota + \omega \vert_t \Lambda )$ of
$B(K)$ into itself is invertible (i.e.  this mapping has both a right and
left inverse) and the mapping
$$
\pi_t^\# (A) = (\iota + \omega \vert_t \Lambda )^{-1}\omega \vert_t (A)
$$
is a completely positive contractive normal linear mapping of $B(H) =
B(K\otimes L^2(0,\infty ))$ into $B(K)$.  The $q$-weight map $\omega$ over $K$
is unital if $\omega (I-\Lambda ) = I$.  The mappings $\pi_t^\#$ are called
the generalized boundary representation of $\omega$.  If $\eta$ is an
element of $B(K) \otimes \mathfrak A (K)_*$ we denote the fact that $\eta$ is
completely positive by writing $\eta \geq 0$ and we denote the fact that
the generalize boundary representation $\psi_t^\#$ constructed from $\eta$ as
shown above is completely positive for all $t > 0$ by writing $\eta \geq_q 0$.
\end{defn}

Given an $\omega \in B(K) \otimes \mathfrak A (K)_*$ to check that $\omega$ is
a $q$-weight map it is only necessary to check that the generalized boundary
representation $\pi_{t_k}^\#$ of $\omega$ is completely positive and
$\pi_{t_k}^\# (I) \leq I$ for any sequence of $t_k > 0$ so that $t_k
\rightarrow 0+$ as $k \rightarrow \infty$.  This is because if $\pi_t^\#$
is a completely positive contraction then $\pi_s^\#$ is a completely
positive contraction for $s \geq t$ so it is only necessary to check the
condition for small $t$.  We caution the reader that in checking that a
weight map is a $q$-weight map that knowing that the limit
of $\pi_t^\# (I)$ as $t \rightarrow 0+$ is $B$ and $B \leq I$ does not
imply $\pi _t^\# (I) \leq I$ for $t > 0$.  The important result of
\cite{powers-CPflows} is that every unital $q$-weight map over $K$ uniquely
defines a spatial $E_0$-semigroup of $B(H)$ and up to cocycle conjugacy every
spatial $E_0$-semigroup of $B(H)$ comes from a unital $q$-weight map over a
Hilbert space $K$.

\begin{defn}
Suppose $K$ is a separable Hilbert space and
$\omega$ and $\eta$ are $q$-weight maps over $K$ with generalized boundary
representations $\pi_t^\#$ and $\psi_t^\#$, respectively.  We say $\eta$ is a
$q$-subordinate of $\omega$ (denoted $\omega \geq_q \eta )$ if for all $t >
0$ we have $\pi_t^\# - \psi_t^\#$ is completely positive.  We say a
$q$-weight map $\omega$ is $q$-pure if the $q$-subordinates of $\omega$ are
totally ordered so if $\omega \geq_q \eta_1 \geq_q 0$ and $\omega \geq_q
\eta_2 \geq _q 0$ then either $\eta_1 \geq_q \eta_2$ or $\eta_2 \geq_q
\eta_1$.
\end{defn}

Again to check that $\eta$ is a $q$-subordinate of $\omega$ it is only
necessary to check that $\pi_{t_k}^\# - \psi_{t_k}^\#$ is completely
positive for any sequence $t_k > 0$ so that $t_k \rightarrow 0+$ as $k
\rightarrow \infty$ in that if $\pi_t^\# \geq \psi_t^\#$ then $\pi_s^\#
\geq \psi_s^\#$ for all $s \geq t$.  Again the fact that $\pi_t^\# -
\psi_t^\#$ is completely positive in the limit does not imply that $\eta$
is a $q$-subordinate of $\omega $.

If $\omega$ is a unital $q$-weight map and $\eta_1$ and $\eta_2$ are
$q$-subordinates of $\omega$ then $\eta_1$ and $\eta_2$ are uniquely
associated with a positive contractive local cocycles $C_1$ and $C_2$ of
norm one of the $E_0$-semigroup $\alpha$ induced by $\omega$ so $0 \leq
C_1(t) \leq I,\medspace 0 \leq C_2(t) \leq I$ and $C_1(t) \geq C_2(t)$ for
all $t > 0$ if and only if $\eta_1 \geq_q \eta_2$.  Conversely, if $C$ is a
positive contractive local cocycle of norm one for the $E_0$-semigroup
$\alpha$ associated with $\omega$ then there is a $q$-subordinate $\eta$ of
$\omega$ uniquely associated with $C$ so there is a one to one order
preserving mapping from contractive local cocycles of norm one for $\alpha$
onto the $q$-subordinates of $\omega $.

The next theorem summarizes results in \cite{powers-CPflows} and shows how to
compute the index of the $E_0$-semigroup induced by a $q$-weight map.

\begin{thm}
Suppose $\omega$ is a $q$-weight map over a Hilbert space $K$ and
$\pi_t^\#$ is the generalized boundary representation of $\omega$. Then for
each $s > 0$ the map $A \rightarrow \pi_t^\# \vert_s (A)$ is non decreasing
in $t$ for $0 < t < s$.  There is a completely positive $\sigma$-weakly
continuous mapping $\pi_o^\#$ of $B(K \otimes L^2(0,\infty ))$ into $B(K)$
so that for each $s > 0$ the map $A \rightarrow \pi_t^\# \vert_s (A)$ for
$0 < t < s$ converges in the $\sigma$-strong operator topology to $\pi_o^\#
\vert_s (A)$ as $t \rightarrow 0+$.  The mapping $\pi _o^\#$ is called the
normal spine of $\omega$.  If $\omega$ is unital the index of the
$E_0$-semigroup induced by $\omega$ is the rank of $\pi_o^\#$ as a
completely positive map.  We say the index of $\omega$ is the index of the
normal spine of $\omega$.  In particular we say $\omega$ is of index zero
if normal spine of $\omega$ is zero.
\end{thm}

It follows that to understand all $E_0$-semigroups of type~II$_o$ one only
has to understand all unital $q$-weight maps over a Hilbert space $K$ of index
zero.  The next theorem shows how to determine whether two $q$-weight maps of
index zero induce $E_0$-semigroups that are cocycle conjugate.  First we
define $q$-corners between two $q$-weight maps.

\begin{defn}
Suppose $\omega_1$ and $\omega_2$ are
$q$-weight maps over the Hilbert spaces $K_1$ and $K_2$, respectively.  Let $K$
be the direct sum of $K_1$ and $K _2$ so $K = K_1 \oplus K_2$.  Note every
operator in $B(K)$ can be uniquely written in matrix form
$$
A = \left[\begin{matrix} A_{11}&A_{12}
\\
A_{21}&A_{22}
\end{matrix} \right]
$$
where $A_{ij}$ is a bounded linear operator from $K_j$ to $K_i$ for $i,j =
1,2$.  Similarly every operator in $B(K \otimes L^2(0,\infty ))$ and every
operator in $\mathfrak A (K)$ can be written in matrix form where $A_{ij}$ is a
bounded linear operator from $K_j\otimes L^2(0,\infty )$ to $K_i\otimes
L^2(0,\infty )$.  Let $E_i$ be the hermitian projection of $K = K_1 \oplus
K_2$ onto $K_i$ for $i = 1,2$.  Now consider the $b$-weight map on
$\mathfrak A (K)$ given by

$$
\omega ( \left[\begin{matrix} A_{11}&A_{12}
\\
A_{21}&A_{22}
\end{matrix} \right]) =
\left[\begin{matrix} \omega_1(A_{11})&\gamma (A_{12})
\\
\gamma^*(A_{21})&\omega_2(A_{22})
\end{matrix} \right]
$$
where $\gamma \in B(K) \otimes \mathfrak A (K)_*$ and
$$
\gamma (A) = E_1\gamma ((E_1\otimes I)A(E_2\otimes I))E_2
$$
for $A \in \mathfrak A (K)$ so $\gamma$ only depends on the $A_{12}$ entry of
$A$ and $\gamma (A)$ is a matrix with all zero entries except for the upper
right entry.  We mean by $\gamma^*$ the mapping given by $\gamma^*(A) =
(\gamma (A^*))^*$.  We say $\gamma$ is a corner from $\omega_1$ to
$\omega_2$ if $\omega$ given above is completely positive.  We say $\gamma$
is a $q$-corner from $\omega_1$ to $\omega_2$ if $\omega$ is $q$-positive,
$\omega \geq_q 0$.

We say $\gamma$ is a maximal corner from $\omega_1$ to $\omega_2$ if
$\eta$ is a subordinate $\eta$ of $\omega$ of the form
$$
\eta ( \left[\begin{matrix} A_{11}&A_{12}
\\
A_{21}&A_{22}
\end{matrix} \right]) =
\left[\begin{matrix} \omega_1^{\prime}(A_{11})&\gamma (A_{12})
\\
\gamma^*(A_{21})&\omega_2(A_{22})
\end{matrix} \right]
$$
then $\omega_1^{\prime} = \omega_1$.  We say $\gamma$ is a maximal
$q$-corner from $\omega_1$ to $\omega_2$ if $\eta$ is a $q$-subordinate of
$\omega$ of the form
$$
\eta ( \left[\begin{matrix} A_{11}&A_{12}
\\
A_{21}&A_{22}
\end{matrix} \right]) =
\left[\begin{matrix} \omega_1^{\prime}(A_{11})&\gamma (A_{12})
\\
\gamma^*(A_{21})&\omega_2 (A_{22})
\end{matrix} \right]
$$
then $\omega_1^{\prime} = \omega_1$.

We say $\gamma$ is a hyper maximal corner from $\omega_1$ to $\omega_2$
if $\eta$ is a subordinate $\eta$ of $\omega$ of the form
$$
\eta ( \left[\begin{matrix} A_{11}&A_{12}
\\
A_{21}&A_{22}
\end{matrix} \right] ) =
\left[\begin{matrix} \omega_1^{\prime}(A_{11})&\gamma (A_{12})
\\
\gamma^*(A_{21})&\omega_2^{\prime}(A_{22})
\end{matrix} \right]
$$
then $\omega_1^{\prime} = \omega_1$ and $\omega_2^{\prime} = \omega_2$.  We
say $\gamma$ is a hyper maximal $q$-corner from $\omega_1$ to $\omega_2$ if
$\eta$ is a $q$-subordinate of $\omega$ of the form
$$
\eta ( \left[\begin{matrix} A_{11}&A_{12}
\\
A_{21}&A_{22}
\end{matrix} \right]) =
\left[\begin{matrix} \omega_1^{\prime}(A_{11})&\gamma (A_{12})
\\
\gamma^*(A_{21})&\omega_2^{\prime}(A_{22})
\end{matrix} \right]
$$
then $\omega_1^{\prime} = \omega_1$ and $\omega_2^{\prime} = \omega_2$.
\end{defn}

\begin{thm}
Suppose $\omega_1$ and $\omega_2$ are unital
$q$-weight maps over $K_1$ and $K_2$ of index zero, respectively.  Then the
$E_0$-semigroups induced by $\omega_1$ and $\omega_2$ are cocycle conjugate
if and only if there is a hyper maximal $q$-corner from $\omega_1$ to
$\omega_2$.
\end{thm}

Note that if $\omega$ is a $q$-weight map over $K$ then $\omega$ is a hyper
maximal $q$-corner from $\omega$ to $\omega$ so every $q$-weight map over $K$
is cocycle conjugate to itself.  Technically cocycle conjugacy refers to
$E_0$-semigroups which correspond to unital $q$-weight maps but we extend the
notion of cocycle conjugacy to arbitrary $q$-weight maps by saying the
$q$-weight maps $\omega_1$ and $\omega_2$ over $K_1$ and $K_2$, respectively,
are cocycle conjugate if there is a hyper maximal $q$-corner from
$\omega_1$ to $\omega_2$.  An important word of caution is that in the non
unital case we do not know that cocycle conjugacy is an equivalence relation.

\section{$Q$-pure $q$-weight maps}

In this section we discuss the notion of $q$-pure $q$-weight maps over a
Hilbert space $K$.  We recall from the last section.

\begin{defn}
We say a $q$-weight map is $q$-pure if its
$q$-subordinates are totally ordered.
\end{defn}

We believe the next theorem is true for $q$-pure $q$-weight maps over $K$ where
$K$ is infinite dimensional but so far we only have a proof for the case
when $K$ is finite dimensional (i.e. $K = \mathbb C^p$ for $p$ a positive
integer.)

\begin{thm}
Suppose $\omega$ is $q$-pure $q$-weight map over $\mathbb
C ^p$ and $\rho$ is a faithful normal state on $B(\mathbb C^p)$.  Then the
mapping $\eta (\rho ) \rightarrow \eta (\rho )(I-\Lambda )$ is a one to one
mapping of the $q$-subordinates $\eta$ of $\omega$ onto the interval $[0,\omega
(\rho )(I-\Lambda )]$.  Furthermore, if $\eta$ is a $q$-subordinate of
$\omega$ then the range of $\eta$ is contained in the range of $\omega$.
\end{thm}

\begin{proof}  Assume the hypothesis. Suppose $\eta$ and $\nu$ are $q$-weight
maps so that $\omega \geq_q \eta \geq_q \nu \geq_q 0$.  Since $\eta \geq \nu$
we have $\eta (\rho )(I-\Lambda ) \geq \nu (\rho )(I-\Lambda )$ and if
$(\eta-\nu )(\rho)(I-\Lambda ) = 0$ it follows that $\eta = \nu$.  Hence, we see
the mapping $\eta \rightarrow \eta (\rho )(I-\Lambda )$ is one to one.  Next we
show there are no gaps, which is to say that for every $s \in [0,\omega (\rho )
(I-\Lambda)]$, there is a $q$-subordinate $\nu$ of $\omega$ so that $\nu (\rho )
(I-\Lambda) = s$.  If $s = \omega (\rho )(I-\Lambda )$ then $\nu = \omega$
provides the example and if $s = 0$ then $\nu = 0$ provides the example.  We
consider the case where $s \in (0,\omega (\rho )(I-\Lambda ))$.  Let $\pi_t^\#$
be the generalized boundary representation of $\omega$ and let $\eta_{(t,
\lambda )} = \lambda (\iota - \lambda\pi_t^\#\Lambda )^{-1}\pi_t^\#$ for
$\lambda \in [0,1]$ and $t > 0$.  Note that $\eta_{(t,\lambda )}$ is a
$q$-subordinate of $\omega \vert_t $ and as $\lambda$ goes from 1 to 0,
$\eta_{(t,\lambda )}(\rho)(I-\Lambda )$ goes continuously from $\omega \vert_t
(\rho )(I-\Lambda )$ to zero.  Since $s < \omega (\rho )(I-\Lambda )$ and
$\omega \vert_t (\rho)(I-\Lambda ) \rightarrow \omega (\rho )(I-\Lambda )$ as
$t \rightarrow 0+$ we have that there is a $t_o$ so that $\omega \vert_t (\rho )
(I-\Lambda ) \geq s$ for $t \leq t_o$.  Hence, for each $t \in (0,t_o]$ there
is a $\lambda = \lambda_t$ so that $\eta_{(t,\lambda )}(\rho )(I-\Lambda ) = s$.
To simplify notation we denote this $\eta_{(t,\lambda )}$ by $\eta_t$.
Calculating $\eta_t$ in terms of $\omega$ we find

$$
\eta_t = \lambda_t(\iota + (1-\lambda_t)\omega \vert_t \Lambda )^{-1
}\omega \vert_t
$$
for $t > 0$.

Let $\{ t_k:k = 1,2,\cdots \}$ be decreasing sequence tending
to zero (so $t_k \rightarrow 0$ as $k \rightarrow \infty$) and let $\sigma$
be an ultrafilter.  To further simplify notation we let $\eta_k =
\eta_{t_k}$.  Since $\eta_k \in \mathfrak A (K)_*$ and we have the bound
$\Vert\eta_k(I-\Lambda )\Vert \leq 1$ the $\eta_k$ are in a compact set in
the weak topology so $\lim_\sigma \eta_k = \nu$ exists.

We claim $\nu$ is the desired $q$-subordinate of $\omega$.  Since for $k$
sufficiently large we have $\eta_k(\rho )(I-\Lambda ) = s$ it follows that
$\nu (\rho )(I-\Lambda ) = s$.  To show $\nu$ is a $q$-subordinate of $\omega$
we must show that for each $t \in (0,t_o)$ that $(\iota+\nu \vert_t \Lambda
)^{-1}$ exists and $\pi_t^\# \geq (\iota+\nu \vert_t \Lambda )^{-1}\nu \vert_t
0$.  Let $\phi_{kt}^\#$ be the generalized boundary representation of
$\eta_k$.  Note for each $t > 0$ and $k$ sufficiently large we have
$\pi_t^\# \geq \phi_{kt}^\# \geq 0,\medspace \pi_t^\#\Lambda \geq
\phi_{kt}^\#\Lambda$ and $I \geq \pi_t^\#\Lambda (I) \geq
\phi_{kt}^\#\Lambda (I)$.  Now we have
$$
(\iota + \omega \vert_t \Lambda )^{-1} = \iota - \pi_t^\#\Lambda\qquad
\text{and } \qquad (\iota + \eta_k\vert_t )^{-1} = \iota -
\phi_{kt}^\#\Lambda.
$$
Since the mappings above are mappings of finite dimensional linear spaces
into themselves the weak limit of such maps is also the strong limit and,
therefore, the limit of the composition of two such maps is the composition
of the limits so, for example,
$$
\lim_\sigma \phi_k\psi_k = \lim_\sigma\phi_k \lim_\sigma\psi_k
$$
and we will freely use this in our computations.  Note that
$$
(\iota + \eta_k \vert_t \Lambda )^{-1} = \iota - \phi_{kt}^\# \Lambda.
$$
Then if
$$
\psi_t = \lim_\sigma (\iota - \phi_{kt}^\# \Lambda)
$$
we see that
$$
(\iota+\nu \vert_t \Lambda )\psi_t = \lim_\sigma (\iota+\eta_k \vert_t \Lambda
)(\iota+\eta_k \vert_t \Lambda )^{-1} = \lim_\sigma \iota = \iota
$$
and
$$
\psi_t(\iota+\nu \vert_t \Lambda ) = \lim_\sigma (\iota+\eta_k \vert_t \Lambda
)^{-1}(\iota+\eta_k \vert_t \Lambda ) = \lim_\sigma \iota = \iota
$$
and we conclude that $(\iota+\nu \vert_t \Lambda )^{-1}$ exists and
\begin{align*}
(\iota +\nu \vert_t \Lambda )^{-1}\nu \vert_t  &= \lim_\sigma (\iota -
\phi_{kt}^\#\Lambda ) (\iota - \phi_{kt}^\#\Lambda )^{-1} \phi_{kt}^\#
\\
&= \lim_\sigma \phi_{kt}^\#
\end{align*}
and since $\pi_t^\# \geq \phi_{kt}^\# \geq 0$ for $k$ sufficiently large we
conclude that the generalized boundary representation $\phi_t^\#$ for $\nu$
exist and satisfies
$$
\pi_t^\# \geq \phi_t^\# \geq 0
$$
and, hence, $\nu$ is $q$-positive and $\nu$ is a $q$-subordinate of $\omega
$.

To complete the proof we show that the range of any $q$-subordinate is
contained in the range of $\omega$.  Recall that
$$
\eta_{(t,\lambda )} = \lambda (\iota - \lambda\pi_t^\#\Lambda )^{-1}\pi_t^\#
=\lambda\pi_t^\# + \lambda^2\pi_t^\#\Lambda\pi_t^\# + \cdots
$$
so we see the range of $\eta_{(t,\lambda )}$ is contained in the range of
$\pi _t^\#$ which is contained in the range of $\omega$.  Since the range
of $\eta_{(t,\lambda )}$ is contained in the range of $\omega$ for all $t >
0$ and $\lambda \in [0,1]$ and any $q$-subordinate $\nu$ of $\omega$ is the
limit of $\eta_{(t,\lambda )}$ it follows that the range of $\nu$ is
contained in the range of $\omega.$ \end{proof}

Consider $\omega$ a $q$-weight map over $\mathbb C$ so $\omega \in \mathfrak A
(\mathbb C ) _*$ and $\omega (I - \Lambda ) \leq 1$.  In Theorem 3.9 of
\cite{powers-holyoke} is was shown how to find all $q$-subordinates of $\omega$.

\begin{thm}
Suppose $\omega$ is a $q$-weight map over $\mathbb C$ and
$\rho$ is a positive normal functional on $B(L^2(0,\infty ))$ so $\rho \in
B(L^2(0,\infty ))_* \subset \mathfrak A (\mathbb C )_*$ and $\rho (I) < \infty$
and $\omega \geq \rho$ so $\omega (A) \geq \rho (A)$ for all positive $A
\in \mathfrak A (\mathbb C )$.  Then $\eta = \lambda (1 + \rho (\Lambda
))^{-1}(\omega - \rho )$ for $0 \leq \lambda \leq 1$ is a $q$-subordinate
of $\omega$.  Conversely, suppose $\eta$ is a non zero $q$-subordinate of
$\omega$ then there is a positive normal functional $\rho \in
B(L^2(0,\infty ))_*$ (so $\rho (I) < \infty )$ and a real number $\lambda
\in (0,1]$ so that $\eta = \lambda (1 + \rho (\Lambda ))^{-1}(\omega -
\rho )$.  Furthermore, if $\omega$ is unbounded then $\rho$ and $\lambda$
are unique.
\end{thm}

One sees that an unbounded $q$-weight map $\omega$ over $\mathbb C$ is $q$-pure
if and only if $\rho$ is a subordinate of $\omega$ (so $\omega (A) \geq \rho
(A) \geq 0$ for all positive $A \in \mathfrak A (\mathbb C ))$ and $\rho$ is
bounded then $\rho = 0$.  This notion of purity comes up again and again
as a condition for the $q$-purity of a $q$-weight map.  For this reason we will
give this notion a name.

\begin{defn}
Suppose $K$ is a separable Hilbert space and $\mu$ is a $b$-weight over $K$
(i.e. $\mu \in \mathfrak A (K)_*$) then $\mu$ is strictly infinite if $\mu$ has
no bounded suborinates.  Similarly if $\phi$ is a $B(K)$ valued completely
positive $b$-weight map on $\mathfrak A (\mathbb C^p)$ (so $\phi \in B(K)
\otimes \mathfrak A (\mathbb C^p)_*$ and $\phi \geq 0)$ we say $\phi$ is
strictly infinite if $\phi$ has no bounded non zero subordinates so if $\eta$
is a non zero subordinate of $\phi$ (i.e. $\eta$ is another $B(K)$ valued
completely positive $b$-weight map on $\mathfrak A (\mathbb C^p)$ and $\phi
\geq \eta )$ then $\eta$ is unbounded.
\end{defn}

In this terminology we see that an unbounded $q$-weight map $\omega$ over
$\mathbb C$ is $q$-pure if and only if it is strictly infinite.  In \cite{jmp1}
we discussed range rank one $q$-weight maps over $\mathbb C^p$ and our results
are summarized in the following theorems.

\begin{thm}
Suppose $\omega$ is a $q$-weight map of range rank one
over $K$ of index zero where $K$ is a separable Hilbert space so $\omega$
can be expressed in the form $\omega (\rho )(A) = \rho (T)\mu (A)$ for
$\rho \in B (K )_*$ and $A \in \mathfrak A (\mathbb C^p)$
where $T$ is a positive operator of norm one and $\mu$ is a positive element of
$\mathfrak A (K \otimes L^2(0,\infty ))_*$ and $\mu (I -\Lambda (T)) \leq 1$
and $\mu (I) = \infty$. Then $\omega$ is $q$-pure if and only if the following
three conditions are met.
\begin{enumerate}[(i)]
\item $T$
is a projection.
\item $\mu$ is strictly infinite.
\item If $e \in B (K )$ is an hermitian rank one
projection with $T \geq e$ then $\mu (\Lambda (e)) = \infty$.
\end{enumerate}
\end{thm}

\begin{thm}
Suppose $\omega$ and $\eta$ are $q$-pure range
rank one $q$-weight maps over $K_1$ and $K_2,respectively$, where $K_1$ and
$K_2$ are separable Hilbert spaces so $\omega (\rho )(A) = \rho (T_1)\mu (A)$
for $A \in \mathfrak A (K)$ and $\rho \in B(K_1)_*$ and $\eta (\rho )(A) =
\rho (T_2)\nu (A)$ for $A \in \mathfrak A (K_2)$ and $\rho \in B(K_2)_*$ and
where $T_1$ and $T_2$ are hermitian projections and $\mu$ and $\nu$ are
$q$-pure $q$-weight maps.  Then $\omega$ and $\eta$ are cocycle conjugate (so
there is a hyper maximal $q$-corner from $\omega$ to $\eta$ ) if and only if
there is a partial isometry $U$ from $K _1$ to $K_2$ and a $\lambda > 0$ so that
$U^*U=T_1$, $UU^* = T_2$, and $\mu$ and $\nu$ can be expressed in the
form
$$
\mu (A) = \sum_{k\in J} (f_k,Af_k)\qquad \text{and} \qquad \nu (B) =
\sum_{k \in J} (g_k,Bg_k)
$$
for $A \in \mathfrak A (K_1)$ and $B \in \mathfrak A (K_2)$ with $g_k = \lambda
(U \otimes I)f_k + h_k$ where $h_k \in K_2\otimes L^2(0,\infty )$ for $k \in
J$ and
$$
\sum_{k\in J} \Vert h_k\Vert^2 < \infty .
$$
\end{thm}

In \cite{jmp1} the above theorem was proved in the finite dimensional case.
In this paper we will only make use of this theorem in the case
where $K_1$ and $K_2$ are finite dimensional.

\section{Completely positive and conditionally positive maps}

This section we discuss completely positive maps with special emphasis on
completely positive maps of the $(p \times p)$-matrices $B(\mathbb C ^p)$ into
themselves.  We denote by $\mathcal{S}(\mathbb C^p)$ the space of all hermitian
linear maps of $B(\mathbb C^p)$ into itself.
If $\phi \in \mathcal{S} (\mathbb C^p)$
we denote by $\phi_{ij}$ the $(i,j)$-entry of the matrix $\phi (A)$.  Given
such a mapping we define the super matrix associated with $\phi$ as
$$
S_{injm} = \phi_{ij}(e_{nm})
$$
where $e_{nm}$ are the complete set of matrix units for $B(\mathbb C^p)$ i.e.
$(e_{nm}x)_k = \delta_{nk}x_m$ for $n,m,k \in \{ 1,\cdots ,p\}$ for
$x \in \mathbb C^p$ and $\delta_{mk}$ is the Kronecker delta which equals one
for $m=k$ and zero otherwise.  Note the super matrix $S$ is a
$(p^2 \times p^2)$-matrix.  A very useful result of Choi is that $\phi$ is
completely positive if and only if the super matrix $S$ is positive.  Also the
rank of a completely positive map of $B(\mathbb C^p)$ into itself is equal to
the rank of its super matrix.

We consider ways of representing completely positive maps $\phi$ of $B(\mathbb
C^p)$ into itself.  Since the combined action of $B(\mathbb C^p)$ on $B(\mathbb
C^p)$ by both right and left multiplication is irreducible it follows that
every linear mapping $L$ of $B(\mathbb C^p)$ into itself can be written in the
form
$$
L(A) = \sum_{i=1}^m B_iAC_i
$$
for $A \in B(\mathbb C^p)$ where the $B_i,C_i \in B(\mathbb C^p)$ for $i =
1,\cdots ,m$.  If we require $L$ to be hermitian so that $L(A^*) = L(A)^*$
for $A \in B(\mathbb C^p)$ then we have
$$
L(A) = \tfrac{1}{2} \sum_{i=1}^m B_iAC_i + C_i^*AB_i^*
$$
and since
$$
BAC + C^*AB^* = \tfrac{1}{2} ((B+C^*)A(B+C^*)^* - (B-C^*)A(B-C^*)^*)
$$
we see that every hermitian $L$ can be written in the form
$$
L(A) = \sum_{i=1}^m r_i X_iAX_i^*
$$
for $A \in B(\mathbb C^p)$, where $r_1, \ldots, r_m \in \mathbb{R}$.  Now let
$G$ be the unitary group $SU(p)$ and let $\mu$ be Haar measure on $G$.  We note
that for $A \in B(\mathbb C^p)$ we have
$$
\int_G U_gAU_g^* d\mu (g)  = tr(A)I
$$
where $tr$ is the trace normalized so that $tr(I) = 1$.  Now given an
hermitian linear map $L$ of $B(\mathbb C^p)$ into itself we define $\Theta_L$
as
$$
\Theta_L(A) = \int_G L(U_gA)U_g^* d\mu (g)
$$
for $A \in B(\mathbb C^p)$.  If $L$ is of the form given above then we have
$$
\Theta_L(A) = \int_G \sum_{i=1}^m r_i X_iU_gAX_i^*U_g^* d\mu (g) = \sum_{i=1}^m
r_i tr(X_i^*A)X_i
$$
for $A \in B(\mathbb C^p)$.  Considering $B(\mathbb C^p)$ as a Hilbert space
$B(\mathbb C ^p)\otimes B(\mathbb C^p)$ with inner product $(A,B)_1 = tr(A^*B)$
we see that $\Theta_L$ is an hermitian linear operator and as such it can
be diagonalized so that we can write $L$ in the form
\begin{equation}\label{3.1}
L(A) = \sum_{i=1}^m \lambda_iS_iAS_i^*
\end{equation}
for $A \in B(\mathbb C^p)$ where $S_i \in B(\mathbb C^p)$ and the $\lambda_i\in
\mathbb R$ and $tr(S_i^*S_j) = \delta_{ij}$ for $i,j = 1,\cdots ,m$ and $m
\leq p^2$.  The numbers $\lambda_i$ are the eigenvalues of $\Theta_L$ and
the $S_i$ are the associated normalized eigenvectors.

We show that $L$ is completely positive if and only if $\lambda_i \geq 0$
for $i = 1,\cdots ,m$.  First we note that if $\lambda_i \geq 0$ in
equation (3.1) then $L$ is the sum of completely positive terms and, thus,
$L$ is completely positive.  To prove the converse suppose $\lambda_q < 0$
in the sum (3.1).  Let $h_i$ be an orthonormal basis for $\mathbb C^p$ and let
$A_i$ be the operator on $\mathbb C^p$ given by $A_if = h_1(h_i,S_qf)$ for $f
\in \mathbb C^p$.  We compute the sum
$$
\sum_{i,j=1}^p (h_i,L(A_i^*A_j)h_j) = \sum_{k=1}^m \sum_{i,j=1}^p
\lambda_k(h _i,S_kA_i^*A_jS_k^*h_j).
$$
Now we have
$$
\sum_{j=1}^p A_jS_k^*h_j = \sum_{j=1}^p h_1(h_j,S_qS_k^*h_j) =  p \cdot
tr(S_qS_k ^*) h_1 = p \delta_{qk} h_1.
$$
Combining this result with the above equation we find
$$
\sum_{i,j=1}^p (h_i,L(A_i^*A_j)h_j)  = p^2 \lambda_q.
$$
We see that if $\lambda_q < 0$ then $L$ is not completely positive.  Hence,
we have shown that $L$ is completely positive if and only if $\lambda_i
\geq 0$ for each term in equation (3.1).

The norm on $\mathcal{S} (\mathbb C^p)$ is given by
$$
\Vert\phi (A)\Vert = \sup \{\Vert\phi (A)\Vert :\Vert A\Vert \leq 1\} .
$$
For a completely positive map $\phi$ we have $\Vert\phi\Vert = \Vert\phi
(I) \Vert$.  In general it is hard to compute the norm so we will often use
a softer norm, the Hilbert Schmidt norm, which is the Hilbert Schmidt norm of
the super matrix, so
$$
\Vert\phi\Vert_{H.S.}^2 = \frac {1} {p^2} \sum_{j,k,r,s=1}^p
\vert\phi_{jkrs }\vert^2.
$$
One checks that if $\phi$ is of the form given in (3.1) then
$$
\Vert\phi\Vert_{H.S.}^2 = \sum_{i=1}^p \lambda_i^2.
$$
The topology given by the Hilbert Schmidt norm and the norm is equivalent
since
\begin{equation}
p^{-3/2}\Vert\phi\Vert  \leq \Vert\phi\Vert_{H.S} \leq
\Vert\phi\Vert
\end{equation}
and note that for the identity $\iota (\iota (A) = A$ for $A \in B(\mathbb
C^p))$ we have $\Vert\iota\Vert_{H.S} = \Vert\iota\Vert$ and if $\phi (A) =
tr(A)e$ for $A \in B(\mathbb C^p)$ where $e$ is a rank one projection we have
$p^{-3/2}\Vert\phi\Vert = \Vert\phi\Vert_{H.S}$.  To prove the right
inequality we note that if $\{ U_i:i = 1,\cdots ,p^2\}$ is a orthonormal
basis for $B( \mathbb C^p)$ of unitaries so $tr(U_i^*U_j) = \delta_{ij}$ then
$\Vert\phi (U_i)\Vert_{H.S.  }^2 \leq \Vert\phi\Vert^2$ so we have
$$
\Vert\phi\Vert_{H.S.}^2 = \frac {1} {p^2}  \sum_{i=1}^{p^2} \Vert\phi (U_i)
\Vert_{H.S.}^2\leq \frac {1} {p^2}  \sum_{i=1}^{p^2} \Vert\phi\Vert^2 =
\Vert \phi\Vert^2
$$
and to prove the left hand inequality we note that if $A_o \in B(\mathbb C^p)$
and $\Vert A_o\Vert = 1$ and $\Vert\phi (A_o)\Vert = \Vert\phi\Vert$ then
there are unit vectors in $\mathbb C^p$ so that $(f,\phi (A_o)g) =
\Vert\phi\Vert$.  Now suppose $e_1$ and $e_2$ are the one dimensional
hermitian projections so that $e_1f = f$ and $e_2g = g$ and we see that if
$\psi (A) = e_1\phi (A)e_2$ then $\Vert\psi\Vert = \Vert\phi\Vert$ and
$\Vert \psi\Vert_{H.S.} \leq \Vert\phi\Vert_{H.S.}$ so we see that if we
want to construct a $\phi$ with $\Vert\phi\Vert = 1$ of smallest Hilbert
Schmidt norm we should consider only $\phi$ that have one dimensional range
that is a multiple of a rank one operator and since the Hilbert Schmidt
norm of $\phi$ and $\phi\cdot U$ where $U$ is unitary are the same it is
enough to consider a mapping $\phi$ of the form $\phi (A) = tr(XA)e$ for $A
\in B(\mathbb C^p)$ and we have $\Vert\phi\Vert = tr((X^*X)^{\tfrac{1}{2}})$ and
\begin{align*}
\Vert\phi\Vert_{H.S.}^2&= \frac {1} {p^2} \sum_{i,j=1}^p p \cdot tr (\phi
(e_{ij})^*\phi (e_{ij})) = \frac {1} {p^2} \sum_{i,j=1}^p \vert tr(Xe_{ij})
\vert^2
\\
&= \frac {1} {p^4}  \sum_{i,j=1}^p \vert x_{ij}\vert^2 = \frac {1} {p^3}
\Vert X\Vert_{H.S.}^2.
\end{align*}
Now if $(X^*X)^{\tfrac{1}{2}}$ has eigenvalues $\lambda_i$ for $i = 1,\cdots ,p$
then
$$
\Vert\phi\Vert = \frac {1} {p} \sum_{i=1}^p \lambda_i\qquad \text{and}
\qquad \Vert X\Vert_{H.S.}^2 = \frac {1} {p}  \sum_{i=1}^p \lambda_i^2
$$
so to minimize $\Vert\phi\Vert_{H.S.}$ subject to the fact that
$\Vert\phi\Vert$ is given one sets all the $\lambda_i$ equal so $\lambda_i
= \Vert\phi\Vert$ and so $\Vert X\Vert_{H.S.}= \Vert\phi\Vert$.  Then we
have
$$
\Vert\phi\Vert_{H.S.}^2 = \frac {1} {p^3} \Vert X\Vert_{H.S.}^2 = \frac {1}
{p^3}\Vert\phi\Vert^2
$$
which establishes the left hand inequality of (3.2).

Next we consider the norm of the product of elements in $\mathcal{S} (\mathbb
C^p)$.  Clearly we have $\Vert\phi\psi\Vert \leq
\Vert\phi\Vert\cdot\Vert\psi\Vert $.  To figure out the effect of the
product on the Hilbert Schmidt norm consider $\phi \in \mathcal{S}(\mathbb C^p)$
and $A \in B(\mathbb C^p)$ so we have
\begin{align*}
\Vert\phi (A)\Vert_{H.S.}^2= tr(\phi (A)^*\phi (A)) &\leq \Vert\phi
(A)\Vert ^2 \leq \Vert\phi\Vert^2\Vert A\Vert^2 = \Vert\phi\Vert^2\Vert
A^*A\Vert
\\
&\leq p\Vert\phi\Vert^2tr(A^*A) = p\Vert\phi\Vert^2\Vert A\Vert_{H.S.}^2.
\end{align*}

If $\phi \in \mathcal{S} (\mathbb C^p)$ we denote by
$$
B_{ij} = \phi (e_{ij})\qquad \text{and} \qquad (f_i,\phi (A)f_j) = p \cdot
tr(\Omega _{ij}A)
$$
where $f_i$ and $e_{ij}$ are the standard basis and matrix units for $\mathbb
C ^p$ and $B(\mathbb C^p)$.  Then we note
$$
\Vert\phi\Vert_{H.S.}^2 = \frac {1} {p} \sum_{i,j=1}^p \Vert
B_{ij}\Vert_{H.S.  }^2 = \frac {1} {p} \sum_{i,j=1}^p
\Vert\Omega_{ij}\Vert^2.
$$

Then we note that the Hilbert Schmidt norm of the product $\psi\phi$ of two
elements in $\mathcal{S} (\mathbb C^p)$ is given by
\begin{align*}
\Vert\psi\phi\Vert_{H.S.}^2 &= \frac {1} {p} \sum_{i,j=1}^p \Vert\psi
(B_{ij })\Vert_{H.S.}^2
\\
&\leq \Vert\psi\Vert^2 \sum_{i,j=1}^p \Vert B_{ij}\Vert_{H.S.}^2=
p\Vert\psi \Vert^2\Vert\phi\Vert_{H.S.}^2.
\end{align*}

Hence, we have $\Vert\psi\phi\Vert_{H.S.} \leq \sqrt{p} \Vert\psi\Vert\cdot
\Vert\phi\Vert_{H.S.}$for $\psi ,\medspace \phi \in \mathcal{S} (\mathbb C^p)$.
An example where the inequality is an equality is when
$$
\psi (A) = p\cdot tr(Ae_{11})e_{11}\qquad \text{and} \qquad \phi (A) =
tr(A)I
$$
for $A \in B(\mathbb C^p)$.

Next we will prove the similar inequality, namely, that
$$\Vert\phi\psi\Vert_{H.S.} \leq \sqrt{p} \Vert\psi\Vert\cdot\Vert\phi\Vert_
{H.S.}$$ for $\psi ,\phi \in \mathcal{S} (\mathbb C^p)$.  Now for $\psi
\in \mathcal{S} (\mathbb C^p)$ we have
$$
\psi (A) = \sum_{k=1}^m \lambda_kS_kAS_k^*
$$
with $S_k \in B(\mathbb C^p)$ and $tr(S_k^*S_n) = \delta_{kn}$ and the
$\lambda _k$ real.  We denote by $\tilde \psi \in \mathcal{S} (\mathbb C^p)$
the mapping
$$
\tilde \psi (A) = \sum_{k=1}^m \lambda_kS_k^*AS_k
$$
for $A \in B(\mathbb C^p)$.  Note that for $A,B \in B(\mathbb C^p)$ we have
$$
tr(A^*\psi (B)) = tr(\tilde \psi (A)^*B).
$$
Then for $\psi ,\medspace \phi \in \mathcal{S} (\mathbb C^p)$ we have
$$
\Vert\phi\psi\Vert_{H.S.}^2 = \frac {1} {p} \sum_{i,j=1}^p \Vert\tilde \psi
(\Omega_{ij})\Vert_{H.S.}^2.
$$
So we need to estimate $\Vert\tilde \psi (A)\Vert_{H.S.}$ in terms of the
Hilbert Schmidt norm of $A$.  Now we have $\Vert\tilde \psi
(A)\Vert_{H.S.}^2\leq p\Vert\tilde \psi \Vert^2\Vert A\Vert_{H.S.}^2$ but
the norm of $\psi$ and $\tilde \psi$ need not be equal as when $\psi (A) =
p\cdot tr(Ae_{11})I$ then $\Vert\psi\Vert = 1$ but $\Vert\tilde \psi\Vert =
p$.  Now we have
\begin{align*}
\Vert\tilde \psi (A)\Vert_{H.S.} & = \sup\{ Re(tr(B^*\tilde \psi (A))):B
\in B(\mathbb C^p),\medspace \Vert B\Vert_{H.S.}\leq 1\}
\\
&= \sup\{ Re(tr(\psi (B)^*A)):B \in B(\mathbb C^p),\medspace \Vert
B\Vert_{H.S.  }\leq 1\}
\\
&\leq \sup\{\Vert\psi (B)\Vert_{H.S.}\Vert A\Vert_{H.S.}:B \in B(\mathbb
C^p), \Vert B\Vert_{H.S.}\leq 1\}
\\
&\leq \sup\{ p^{\tfrac{1}{2}}\Vert\psi\Vert\cdot\Vert B\Vert_{H.S.}\Vert
A\Vert_{H.S.  }:B \in B(\mathbb C^p),\medspace \Vert B\Vert_{H.S.}\leq 1\}
\\
&= p^{\tfrac{1}{2}} \Vert\psi\Vert\cdot\Vert A\Vert_{H.S}
\end{align*}
and, hence, we have
\begin{align*}
\Vert\phi\psi\Vert_{H.S.}^2 &= \frac {1} {p} \sum_{i,j=1}^p \Vert\tilde
\psi (\Omega_{ij})\Vert_{H.S.}^2
\\
&\leq \sum_{i,j=1}^p \Vert\psi\Vert^2\Vert\Omega_{ij}\Vert_{H.S.}^2 =
p\Vert \psi\Vert^2\Vert\phi\Vert_{H.S.}^2.
\end{align*}
Hence, we have shown that
\begin{equation} 
\Vert\psi\phi\Vert_{H.S.} \leq
\sqrt{p}\Vert\psi\Vert\cdot\Vert\phi\Vert_{H.S.  }\qquad \text{and} \qquad
\Vert\phi\psi\Vert_{H.S.} \leq \sqrt{p}\Vert\psi\Vert
\cdot\Vert\phi\Vert_{H.S.}
\end{equation}
for $\psi ,\medspace \phi \in \mathcal{S} (\mathbb C^p)$.

Next we turn to conditionally positive maps.  A mapping of $B(K)$ into
itself is said to be conditionally positive if $\phi$ is hermitian (so
$\phi (A^*) = \phi (A)^*$ for $A \in B(K))$ and if $A_i \in B(K)$ and $f _i
\in K$ for $i = 1,\cdots ,n$ and
$$
\sum_{i=1}^n A_if_i = 0\qquad \text{then} \qquad \sum_{i,j=1}^n (f_i,\phi
(A^*_iA_j)f_j) \geq 0.
$$
We really should call such maps conditionally completely positive maps but we
will use the shorter term, conditionally positive, and hope the reader will
remember we mean the longer expression.  We say a linear mapping $\phi$ of
$B(K)$ into itself is conditionally negative if $-\phi$ is conditionally
positive.  We say such a map $\phi$ is conditionally zero if both $\phi$ and
$-\phi$ are conditionally positive.

If $K = \mathbb C^p$ then one discovers that a mapping $\phi$ is conditionally
positive if and only if the super matrix $S$ associated with $\phi$ is
hermitian and
$$
(F,SF) \geq 0\qquad \text{for} \qquad F \in \mathbb C^p \otimes \mathbb C^p
\qquad \text{with} \qquad \sum_{i=1}^p f_{ii} = 0.
$$
Note this is not the same as saying that the super matrix is conditionally
positive which would be

$$
(F,SF) \geq 0\qquad \text{for} \qquad F \in \mathbb C^p \otimes \mathbb C^p
\qquad \text{with} \qquad \sum_{i,j=1}^p f_{ij} = 0.
$$
We consider the Shur product $A \circ B$ of two matrices in $B(\mathbb C^p)$
given by
$$
(A \circ B)_{ij} = a_{ij}b_{ij}
$$
for $i,j \in \{ 1,\cdots ,p\}$.  We note that the mapping $\psi_A$ given by
$$
\psi_A(B) = A \circ B
$$
is completely positive if and only if $A$ is positive and $\psi_A$ is
conditionally positive if and only if $A$ is conditionally positive by
which we mean $(f,Af) \geq 0$ for all $f \in \mathbb C_o^p$ where
$$
\mathbb C_o^p = \{ f \in \mathbb C^p:\sum_{i=1}^p f_i = 0 \}.
$$

Next we give a brief discussion of ways to represent conditionally positive
maps of $B(\mathbb C^p)$ into itself.  We will use a slightly different
representation.  Looking at equation (3.1) we note that each of the
matrices $S_i$ can expressed as
$$
S_i = s_iI + S_i^{\prime}
$$
where $tr(S_i^{\prime}) = 0$.   One can then write equation (3.1) with
terms that have trace zero.  The resulting $S_i^{\prime}$ will not
necessarily be orthonormal but by choosing a new basis one can write
equation (3.1) in the form
\begin{equation} 
L(A) = sA+YA+AY^*+K_L(A)\qquad \text{where} \qquad K_L(A) = \sum_{i=1}^m
\lambda _iX_iAX_i^*
\end{equation}
where $s$ is real the $\lambda_i$ are real and $Y$ and the $X_i$ are of
trace zero.  The $X_i$ can be chosen so that tr(X$_i^*X_j) = \delta_{ij}$
for $i,j = 1,\cdots ,m$.  We call $K_L$ the internal part of $L$ and $s$ the
coefficient of the identical part of $L$.  We note that $s$ and $Y$ are
uniquely determined since
\begin{equation*}
\Theta_L(I) = \int L(U_g)U_g^* d\mu (g) = sI + Y
\end{equation*}
so $s$ and $Y$ are uniquely determined.  Note $s$ is determined
since $Y$ is of trace zero.

Once $s$ and $Y$ are determined one can form the map
\begin{equation*}
K_L(A) = L(A) - sA - YA - AY^* = \sum_{i=1}^m \lambda_iX_iAX_i^*
\end{equation*}
so $K$ the internal part of $L$ and $s$ the coefficient of the identical part
of $L$ are uniquely determined.

Next we note that $L$ is conditionally positive if and only if
$\lambda_i \geq 0$ for each $i = 1,\cdots ,m$ (i.e. $L$ is
conditionally positive if and only if $K_L$ is completely
positive). Note if $\lambda_i \geq 0$ in (3.4) then $K_L$ is a
completely positive map and one checks that the first three terms
of (3.4) are conditionally zero.  Hence, if the $\lambda_i \geq 0$
the mapping $L$ of (3.4) is the sum of a completely positive map
and a conditionally zero map so it is conditionally positive.

Next we show that if one of the $\lambda_i$ is negative then $L$ is not
conditionally positive.  Suppose then that at least one of the $\lambda_i$,
say $\lambda_q$, in (3.4) is negative.  Let $h_i$ be an orthonormal basis
for $\mathbb C ^p$ and let $A_i$ be the operator given by $A_if =
h_1(h_i,X_qf)$ for $f \in \mathbb C ^p$.  We have
\begin{equation}
\sum_{i=1}^p A_ih_i = \sum_{i=1}^p h_1(h_i,X_qh_i) = p \cdot tr(X_q) h_1 = 0.
\end{equation}

Now if $L$ is conditionally positive we have
$$
\sum_{i,j=1}^p (h_i,L(A_i^*A_j)h_j)  = \sum_{k=1}^m \sum_{i,j=1}^p \lambda
_k(h_i,X_kA_i^*A_jX_k^*h_j) \geq 0.
$$
where in the computation to the right of the equal sign we used the fact that
condition (3.5) causes the first three terms in equation (3.4) to vanish.  Now
we have
$$
\sum_{j=1}^p A_jX^*_kh_j = \sum_{j=1}^p h_1(h_j,X_qX_k^*h_j) =  p
tr(X_qX_k^*) h_1 = p \delta_{kq} h_1.
$$
Combining this result with the above equation we find
$$
\sum_{i,j=1}^p (h_i,L(A_i^*A_j)h_j)  = p^2 \lambda_q.
$$
We see that if $\lambda_q < 0$ then $L$ is not conditionally positive.

We note that if $L$ is conditionally zero then the $\lambda_i$ in the sum
(3.4) are both positive and negative so we see that if $L$ is conditionally
zero then $L$ is of the form
$$
L(A) = sA + YA + AY^*
$$
for $A \in B(\mathbb C^p)$ with $s$ real and $Y \in B(\mathbb C^p)$ of trace
zero.  Note if $L$ is conditionally negative and $L$ is completely positive
then $L$ is of the above form with $Y = 0$ and $s \geq 0$.  To see this,
note that if $L$ satisfies these conditions then $L$ is conditionally zero,
so $L$ is of the above form, and one checks that if $L(e) \geq 0$ for all
rank one projections then $s \geq 0$ and $Y$ is a multiple of the identity.
Since $Y$ has trace zero we have $Y = 0$ so $L(A) = sA$ for $A \in
B(\mathbb C ^p)$ with $s \geq 0$.

Given an element $\phi \in \mathcal{S} (\mathbb C^p)$ we will want to represent
it the form $\phi = s\iota + \psi$ where $\psi$ is small.  Now in general it
is difficult to compute the norm of an element of $\mathcal{S} (\mathbb C^p)$
so we will use the Hilbert Schmidt norm.  We point out an advantage of equation
(3.4) for representing elements of $\mathcal{S} (\mathbb C^p)$ is that if
$\phi$ is of the form (3.4) and you wish to write $\phi = s\iota + \psi$ where
$\psi$ has the smallest Hilbert Schmidt norm then $s$ is precisely the $s$ in
equation (3.4) so $\psi$ is simply the expression in equation (3.4) with
the $sA$ term omitted.

Note if $\phi$ is given by equation (3.4) then one computes
$$
\Vert\phi\Vert_{H.S.}^2 = s^2 + 2tr(Y^*Y) + \sum_{i=1}^m \lambda_i^2
$$
and if we write $\phi = s\iota + \psi$ then
$$
\Vert\psi\Vert_{H.S.}^2 = 2tr(Y^*Y) + \sum_{i=1}^m \lambda_i^2.
$$

Next we consider the question of when a map $\phi \in \mathcal{S} (C^p)$
expressed in the form of equation (3.4) is completely positive.  We claim
the map
$$
L(A) = sA + YA + AY^* + \sum_{i=1}^m \lambda_iX_iAX_i^*
$$
for $A \in B(\mathbb C^p)$ is completely positive if and only if the
$\lambda_i$ are positive and there are complex numbers $c_i$ so that
\begin{equation} 
Y = \sum_{i=1}^m c_iX_i\qquad \text{and} \qquad \sum_{i=1}^m \vert c_i\vert
^2/\lambda_i \leq s.
\end{equation}

To see this assume the conditions on $Y$ and $s$ are satisfied.  Let $r$ be
the square root of the second sum above so $0 \leq r^2 \leq s$.  Then we
have
$$
L(A) = (s-r^2)A + (rI+Y/r)A(rI+Y^*/r) + \sum_{i=1}^m \lambda_iX_iAX_i^* - r
^{-2}YAY^*.
$$
Now
$$
Y/r = \sum_{i=1}^m (\lambda_i^{-\tfrac{1}{2}}c_i/r)\lambda_i^{\tfrac{1}{2}} X_i
= \sum _{i=1}^m b_i\lambda_i^{\tfrac{1}{2}} X_i
$$
where $b_i = \lambda_i^{-\tfrac{1}{2}}c_i/r$ for $i = 1,\cdots ,m$.   The last
two terms in the expression for $L$ are completely positive if and only if
$$
\sum_{i=1}^m \vert b_i\vert^2 \leq 1
$$
and we have
$$
\sum_{i=1}^m \vert b_i\vert^2 = r^{-2} \sum_{i=1}^m \vert
c_i\vert^2/\lambda _i  = r^{-2}\cdot r^2 = 1
$$
so $L$ is the sum of three completely positive maps so $L$ is completely
positive.

In the other direction let $h_i$ be an orthonormal basis for $\mathbb C^p$ and
let $A_i$ be the operator given by $A_if = (h_i,(zI + Z)f)h_1$ for $f \in
\mathbb C ^p$ where $z \in \mathbb C$ and $Z \in B(\mathbb C^p)$ with $tr(Z) =
0$ of our choosing.  We compute
\begin{align*}
\sum_{i,j=1}^p (h_i,L(A_i^*A_j)h_j) &= \sum_{i,j=1}^p s(h_i,A_i^*A_jh_j)
\\
&  + \sum_{i,j=1}^p (h_i,YA_i^*A_jh_j) + \sum_{i,j=1}^p
(h_i,A_i^*A_jY^*h_j)
\\
&\qquad + \sum_{k=1}^m \sum_{i,j=1}^p
\lambda_k(h_i,X_k A_i^*A_jX_k^*h_j).
\end{align*}
Since $Z$ has trace zero we have
$$
\sum_{j=1}^p A_jh_j= h_1(h_j(zI+Z)h_j) = pzh_1
$$
and using this we find
\begin{align*}
\sum_{i,j=1}^p (h_i,L(A_i^*A_j)h_j)  &= sp^2 \vert z\vert^2 +
2p^2Re(\overline {z}tr(Y^*Z))
\\
&\qquad + \sum_{k=1}^m \sum_{i,j=1}^p \lambda_k(h_i,X_kA_i^*A_jX_k^*h_j).
\end{align*}
We compute
$$
\sum_{i=1}^p A_iX_k^*h_i = \sum_{i=1}^p h_1(h_i,(zI + Z)X_k^*h_i)  = p \cdot
tr(X_k^*Z)h_1
$$
so we find
\begin{align*}
\frac {1} {p^2} \sum_{i,j=1}^p (h_i,L(A_i^*A_j)h_j) &= s\vert z\vert^2 +
2Re(\overline {z}tr(Y^*Z))
\\
&\qquad + \sum_{k=1}^m \lambda_ktr(Z^*X_k)tr(X_k^*Z).
\end{align*}
Since this expression must be positive for all complex $z$ we find
$$
\vert tr(Y^*Z)\vert^2 \leq s \sum_{k=1}^m \lambda_ktr(Z^*X_k)tr(X_k^*Z).
$$
Now considering $B(\mathbb C^p)$ as a Hilbert space with inner product $(A,B)
= tr(A^*B)$ we see that the matrices $X_k$ are orthogonal vectors so we can
uniquely express $Y$ in the form
$$
Y = \sum_{k=1}^m c_kX_k  + W
$$
where $W$ is orthogonal to $X_k$ for $k = 1,\cdots ,m$. Note that since $Y$
and the $X_k$ have trace zero then $W$ is of trace zero.  Setting $Z = W$
in the above inequality we find $tr(W^*W) = 0$ so $W = 0$.  Then setting
$$
Z = \sum_{k=1}^m \lambda_k^{-1}c_kX_k
$$
we find
$$
( \sum_{k=1}^m \vert c_k\vert^2/\lambda_k )^2 \leq s \sum_{k=1}^m \vert c_k
\vert^2/\lambda_k
$$
and so we have the desired inequality
$$
\sum_{k=1}^m \vert c_k\vert^2/\lambda_k \leq s.
$$

A useful inequality concerning $Y$ for completely positive maps of the form
(3.4) is the following.  Suppose
$$
L(A) = sA + YA + AY^* + \rho (A)\qquad \text{where} \qquad \rho (A) =
K_L(A) = \sum_{i=1}^m \lambda_iX_iAX_i^*
$$
for $A \in B(\mathbb C^p)$ and $L$ is completely positive.  Then the following
$(2 \times 2)$-matrix with entries in $B(\mathbb C^p)$ is positive so
\begin{equation}
\left[\begin{matrix} sI&Y^*
\\
Y&\rho (I)
\end{matrix} \right]
\geq 0.
\end{equation}
To see this we note since $L$ is completely positive there are complex
numbers $c_k$ for $k = 1,\cdots ,m$ so that
$$
Y = \sum_{i=1}^m c_iX_i\qquad \text{and} \qquad \sum_{i=1}^m \vert c_i\vert
^2/\lambda_i = r \leq s.
$$
Then we have
$$
\sum_{i=1}^m \left[\begin{matrix} \overline{c_i}\lambda_i^{-\tfrac{1}{2}} I
\\
\lambda_i^{\tfrac{1}{2}} X_i
\end{matrix} \right]
\left[\begin{matrix} c_i\lambda_i^{-\tfrac{1}{2}}I&\lambda_i^{\tfrac{1}{2}}
X_i^*
\end{matrix} \right] =
\left[\begin{matrix} rI&Y^*
\\
Y&\rho (I)
\end{matrix} \right]
\geq 0
$$
and since $s \geq r > 0$ the result follows.

Another useful fact about completely positive map in $\mathcal{S} (\mathbb C^p)$
is this.  Suppose $\phi \in \mathcal{S} (\mathbb C^p)$ is completely positive
and $\Vert \phi (I)\Vert < 1$.  Then $\iota - \phi$ is invertible and its
inverse is completely positive.  This is seen as follows.  Since $\phi$ is
completely positive $\Vert\phi\Vert = \Vert \phi (I)\Vert < 1$ and, hence,
the series
$$
1 + \Vert\phi\Vert + \Vert\phi\Vert^2 + \cdots
$$
converges and since $\Vert\phi^n\Vert \leq \Vert\phi\Vert^n$ it follows
that
$$
(\iota - \phi )^{-1} = \iota + \phi + \phi^2 + \cdots
$$
where the series converges in norm.  Since $\phi^n$ is completely positive
for each $n$ it follows that $(\iota - \phi )^{-1}$ is completely positive.

We now prove a technical lemma which is of crucial importance in the next
section.
\begin{lem} 
Suppose $\epsilon > 0$.  Then there is a
$\delta > 0$ so that if $\xi_1$ and $\xi_2$ are completely positive maps in
$\mathcal{S} (\mathbb C^p)$ so that
$$
\Vert\xi_1(I) + \xi_2(I)\Vert < 1\qquad \text{and} \qquad \Vert\iota -
(\iota - \xi_2)^{-1}\xi_1\Vert < \delta
$$
then
$$
\iota - \xi_2 = \kappa (\iota + \eta )
$$
with $\kappa > 0$ and $\Vert\eta\Vert < \epsilon$.
\end{lem}
\begin{proof}  Assume $0 < \epsilon < 1$ and let
$$
\delta = \min(\frac {\epsilon} {4p^2},\medspace \frac {0.1} {\sqrt{p}})
$$
and suppose $\xi_1$ and $\xi_2$ satisfy the hypothesis of the lemma.  Let
$$
\iota - \xi  = b^{-1}(\iota - \xi_2)\qquad \text{with} \qquad b =
\Vert\iota - \xi_2\Vert_{H.S.}.
$$
Note that $b > 0$ since if $b = 0$ then $\xi_2 = \iota$ which would violate
the assumption $\Vert\xi_1(I) + \xi_2(I)\Vert < 1$.  Note $\Vert\iota
- \xi \Vert_{H.S.} = 1$.  Let
$$
\zeta = \iota - (\iota - \xi_2)^{-1}\xi_1
$$
and by assumption $\Vert\zeta\Vert < \delta$.  Note $\iota - \xi_2 =
b(\iota - \xi )$ and we have
$$
\xi_1 = (\iota - \xi_2)(\iota - \zeta ) = b(\iota - \xi )(\iota - \zeta ).
$$
Since $\xi_1 \geq 0$ and $b > 0$ we have
$$
\iota - \xi - \zeta + \xi\zeta \geq 0.
$$
Let $\nu = (\iota - \xi )\zeta$.  Since $\Vert\iota - \xi\Vert_{H.S.} = 1$
and from our previous estimate of the Hilbert Schmidt norm of the product
of two elements of $\mathcal{S} (\mathbb C^p)$ (see inequality (3.3)) we have
$$
\Vert\nu\Vert_{H.S.} = \Vert (\iota - \xi )\zeta\Vert_{H.S.} \leq
\sqrt{p}\Vert \iota - \xi\Vert_{H.S.}\Vert\zeta\Vert =
\sqrt{p}\Vert\zeta\Vert < \delta\sqrt {p}
$$
and we have $\iota - \xi - \nu \geq 0$.

Now $\iota - \xi = b^{-1}(\iota - \xi_2)$ and since $\xi_2 \geq 0$ and the
identity map $\iota$ is conditionally zero we have $\iota - \xi$ is
conditionally negative.  Hence, we can write this map in the form of
equation (3.4) as
$$
(\iota - \xi )(A) = sA + YA + AY^* - \sum_{i=1}^m \lambda_iX_iAX_i^*
$$
where $Y$ and the $X_i$ have trace zero and $s$ is real and $tr(X_i^*X_j) =
\delta_{ij}$ and $\lambda_i > 0$ for $i,j = 1,\cdots ,m$.  Since the
Hilbert Schmidt norm of $\iota - \xi$ is one we have
$$
\Vert\iota - \xi\Vert_{H.S.}^2 = s^2 + 2tr(Y^*Y) + \sum_{i=1}^m \lambda_i^2 = 1.
$$
Now we can express $\nu$ in the form
$$
\nu (A) = rA + ZA + AZ^* - \sum_{i=1}^q \mu_iS_iAS_i^*
$$
where $Z$ and the $S_i$ have trace zero and $tr(S_i^*S_j) = \delta_{ij}$.
Note we have put a minus sign in front of the sum in anticipation of the
fact that the $\mu_i$ will turn out to be positive.  Since
$\Vert\nu\Vert_{H.S.}< \delta p^{\tfrac{1}{2}}$
$$
\vert r\vert^2 + 2\Vert Z\Vert_{H.S.}^2+ \sum_{i=1} \mu_i^2 < p\delta^2.
$$
Now we have
$$
(\iota-\xi-\nu )(A) = (s-r)A+(Y-Z)A+A(Y-Z)^*-\sum_{i=1}^m \lambda_iX_iAX_i^
* + \sum_{j=1}^q \mu_jS_jAS_j^*
$$
for $A \in B(\mathbb C^p)$.  Note the above map is completely positive.  Now
for the above mapping to be completely positive the two sums term above
must add up to a positive quadratic form.  Note then if any of the $\mu_i$
are negative then the above expression is not completely positive.  Let $A$
and $B$ correspond to the positive quadratic forms
$$
A(Q)\quad \Leftrightarrow  \quad \sum_{j=1}^q \mu_iS_iQS_i^*\qquad \text{and}
\qquad B(Q) \quad \Leftrightarrow \quad \sum_{i=1}^m \lambda_iX_iQX_i^*.
$$
Since if $A \geq B \geq 0$ we have
\begin{align*}
tr(A^2) &= tr(A^{\tfrac{1}{2}} AA^{\tfrac{1}{2}} ) \geq tr(A^{\tfrac{1}{2}}
BA^{\tfrac{1}{2}}) = tr(AB)
\\
&= tr(B^{\tfrac{1}{2}} AB^{\tfrac{1}{2}} )\geq tr(B^{\tfrac{1}{2}}
BB^{\tfrac{1}{2}} ) = tr(B^2)
\end{align*}
it follows that
$$
\sum_{i=1}^m \lambda_i^2 \leq \sum_{j=1}^q \mu_j^2 < p\delta^2.
$$
Since the quadratic form $A - B$ is positive we can write
$$
L(A) =  -\sum_{i=1}^m \lambda_iX_iAX_i^* + \sum_{j=1}^q \mu_jS_jAS_j^*  =
\sum _{k=1}^r \sigma_kC_kAC_k
$$
for $A \in B(\mathbb C^p)$ where the $C_k$ are of trace zero and $tr(C_k^*C_l)
= \delta_{kl}$ and $\sigma_k > 0$ for $k,l = 1,\cdots ,r$.  We have
$$
\sum_{j=1}^r \sigma_j^2 \leq  \sum_{j=1}^q \mu_j^2  <  p\delta^2
$$
and we have
$$
(\iota - \xi - \nu )(A) = (s-r)A + (Y-Z)A + A(Y-Z)^* + \sum_{j=1}^r
\sigma_jC _jAC_j^*
$$
for $A \in B(\mathbb C^p)$.  Since the above map is completely positive we
have from our previous discussion that $Y - Z$ must satisfy the conditions
(3.6) so
$$
Y - Z = \sum_{i=1}^r c_iC_i\qquad \text{and} \qquad \sum_{i=1}^r \vert
c_i\vert ^2/\sigma_i \leq s-r.
$$
Note $s^2 \leq \Vert\iota - \xi\Vert_{H.S.}^2= 1$ and $r^2 \leq
\Vert\nu\Vert _{H.S.}^2< p\delta^2$
$$
s - r < 1 + \delta p^{\tfrac{1}{2}}
$$
and
$$
\sum_{i=1}^r \vert c_i\vert^2/\sigma_i < 1 + \delta p^{\tfrac{1}{2}}.
$$
Now we have
$$
\Vert Y - Z\Vert_{H.S.}^2 = \sum_{i=1}^r \vert c_i\vert^2.
$$
Maximizing the above sum subject the above inequality on the $\vert
c_i\vert ^2$ we see the maximum occurs when $c_i = 0$ except $c_q$ where
$\sigma_q$ is the maximum of the $\sigma_{i^{^{\prime}s.}}$  Hence, we have
$$
\Vert Y - Z\Vert_{H.S.}^2 < (1 + \delta p^{\tfrac{1}{2}} )\sigma_{\max}
$$
where $\sigma_{\max}$ is the largest of the $\sigma_i$.  We have
$$
\sigma_{\max}^2  \leq \sum_{i=1}^r \sigma_i^2  < p\delta^2
$$
we have
$$
\Vert Y - Z\Vert_{H.S.}< \delta p^{\tfrac{1}{2}}\sqrt{1+\delta p^{\tfrac{1}{2}}}
$$
and since the $2\Vert Z\Vert_{H.S.}^2< p\delta^2$ we have
$$
\Vert Y\Vert_{H.S.}\leq \Vert Y - Z\Vert_{H.S.}+ \Vert Z\Vert_{H.S.} <
\delta p^{\tfrac{1}{2}} (\sqrt{1+\delta p^{\tfrac{1}{2}}} + 2^{-\tfrac{1}{2}}).
$$
Using the fact that $\delta\leq p^{-\tfrac{1}{2}}/10$ so $\delta
p^{\tfrac{1}{2}} \leq 0.1$ this simplifies to
$$
\Vert Y\Vert_{H.S.}< \delta p^{\tfrac{1}{2}} (\sqrt{1.1} + \sqrt{\tfrac{1}{2}})
< 2\delta p^{\tfrac{1}{2}}.
$$
We recall that
$$
\Vert\iota - \xi\Vert_{H.S.}^2 = s^2 + 2tr(Y^*Y) + \sum_{i=1}^m \lambda_i^2
= 1
$$
and
$$
\sum_{i=1}^m \lambda_i^2 < p\delta^2\qquad \text{and} \qquad  2tr(Y^*Y) <
8p \delta^2.
$$
We find
$$
1 \geq s^2 > 1 - 9p\delta^2.
$$
We assure the reader that $s > 0$ since $s - r > 0$ and $\vert r\vert <
\delta p^{\tfrac{1}{2}} \leq 0.1$ and $s^2 > 0.9$.  Now we have
$$
\Vert (\iota - \xi ) - s\iota\Vert_{H.S.}^2= 2tr(Y^*Y) + \sum_{i=1}^m
\lambda _i^2 < 8p\delta^2 + p\delta^2 = 9p\delta^2
$$
and, hence, we have from (3.2) that
$$
\Vert (\iota - \xi ) - s\iota\Vert \leq p^{3/2}\Vert (\iota - \xi ) -
s\iota \Vert_{H.S.}< 3\delta p^2.
$$
Now let
$$
\kappa = b s\qquad \text{and} \qquad \eta = s^{-1}(\iota - s\iota - \xi ).
$$
Then we have
$$
\iota - \xi_2 = b(\iota - \xi ) = \kappa (\iota + \eta )
$$
and
$$
\Vert\eta\Vert < 3p^2\delta/s
$$
and since $s > 0$ and $s^2 > 0.9$ we have $3/s < 4$ so $\Vert\eta\Vert <
4\delta p^2 \leq \epsilon .$ \end{proof}

We prove an easy lemma so we can refer to it later.

\begin{lem} 
Suppose $\phi$ is a completely positive map of
$B(\mathbb C ^p)$ into itself and $T$ is an hermitian operator.  Then there is
an hermitian linear map $\psi$ of $B(\mathbb C^p)$ into itself so that $\psi
(I) = T$ and $\psi + \phi$ is conditionally zero.  Furthermore if $\psi
^{\prime}$ is an hermitian linear map of $B(\mathbb C^p)$ into itself so that
$\psi ^{\prime}+ \phi$ is conditionally negative and $\psi ^{\prime} \geq
\psi$ then $\psi ^{\prime} = \psi + s\iota$ where $s \geq 0$ and $\iota$ is
the identity map.
\end{lem}
\begin{proof}  Assume the hypothesis and notation of the lemma.  Let $\psi$
be given by
$$
\psi (A) = YA + AY^*- \phi (A)
$$
for $A \in B(K)$ where $Y = B + iC$ and $B = {\tfrac{1}{2}} (T + \phi (I))$ and
$C = C^*$ can be freely chosen.  Notice we can add a real multiple of the
identity to $C$ and $\psi$ is unchanged.  Now suppose $\psi ^{\prime}$
satisfies the hypothesis of the theorem.  Then $\psi ^{\prime} = \psi +
\eta$ where $\eta$ is completely positive and $\psi ^{\prime} + \phi$ is
conditionally negative so we have the map
$$
A \rightarrow YA + AY^* + \eta (A)
$$
is conditionally negative.  Since the map $A \rightarrow YA + AY^*$  is
conditionally zero we have the map $\eta$ is conditionally negative.
Hence, $\eta$ is conditionally zero and completely positive so
$$
\eta (A) = sA
$$
for $A \in B(\mathbb C^p)$ with $s \geq 0.$ \end{proof}

An important theorem of Evans and Lewis \cite{evans-lewis} is that a mapping
$\phi$ is the generator of a semigroup of completely positive maps if and only
if $\phi$ is conditionally completely positive.  We use this result in the
follows lemma which we will need later.

\begin{lem} 
Suppose $\phi$ is a conditionally negative map of $B(
\mathbb C^p)$ into itself and $\phi (I) \geq sI$ with $s > 0$.  Then $\phi$ is
invertible and its inverse $\phi^{-1}$ is completely positive.  If $\phi
^{\prime}$ is also a conditionally negative map of $B(\mathbb C^p)$ into
itself and $\phi \leq \phi^{\prime}$ then $\phi^{\prime}$ is also
invertible and $\phi^{-1} \geq \phi^{\prime -1}$.
\end{lem}
\begin{proof}  Assume the hypothesis and notation of the lemma.  Choose
$s^{\prime}$ so that $0 < s^{\prime} < s$.  Since the identity map $\iota$
is conditionally zero $\phi - s^{\prime}\iota$ is conditionally negative
and the result of Evans and Lewis mentioned above states that the exponential
of a conditionally positive linear mapping is completely positive.  Hence,
$\Psi_t =$ $\exp(-t(\phi-s^{\prime}\iota ))$ is a completely positive map
for $t > 0.  $ Since $\Psi_t$ is completely positive we have
$\Vert\Psi_t\Vert = \Vert\Psi_t(I)\Vert$.  We have
\begin{align*}
\Psi_t(I) &= I - t(\phi-s^{\prime}\iota )(I) + t^2/2! (\phi-s^{\prime}\iota
)^2(I) \cdots
\\
&\leq I - t(s-s^{\prime})I + t^2/2! (\phi-s^{\prime}\iota )^2(I) \cdots
\\
&\leq I - t(s-s^{\prime})I + I((1/2!)t^2\Vert\phi-s^{\prime}\iota\Vert^2+(1
/3!)t^3\Vert\phi-s^{\prime}\iota\Vert^3+\cdots )
\\
&= (1 -t(s-s^{\prime}) + r(t))I
\end{align*}
where
$$
r(t) = (e^{t\Vert\phi-s^{\prime}\iota\Vert} - 1 - t\Vert\phi-s^{\prime}\iota
\Vert ).
$$
Since $r(t)/t^2\rightarrow \tfrac{1}{2}\Vert\phi-s^{\prime}\iota\Vert^2$ as $t
\rightarrow 0$ there is a $\delta > 0$ so that $r(t) < t(s-s^{\prime})$
for $0 < t < \delta$.  Hence, $\Psi (t) \leq I$ for $0 < t < \delta$ and,
hence, $\Vert\Psi (t)\Vert \leq 1$ for $0 < t < \delta$ and since $\Psi
(t)$ is a semigroup we have $\Vert\Psi (t)\Vert \leq 1$ for all $t$.
Since
$$
e^{-t\phi} = e^{-s^{\prime}t}\Psi (t)
$$
we have $\Vert e^{-t\phi}\Vert \leq e^{-s^{\prime}t}$.  Hence, we have
$$
\phi^{-1} = \int_0^\infty e^{-t\phi} \,dt
$$
where our estimate on the norm of $e^{-t\phi}$ insures that the integral
exists and is equal to $\phi^{-1}$.  Since $\phi^{-1}$ is the integral of
completely positive maps $\phi^{-1}$ is completely positive.

Now suppose $\phi^{\prime}$ has the properties stated in the lemma.  Since
$\phi^{\prime} \geq \phi$ we have $\phi^{\prime}(I) \geq \phi (I) \geq sI$
so $\phi^{\prime-1}$ exists and is completely positive.  Consider the
equation
$$
\frac {d} {dt} e^{-t\phi}e^{t\phi^{\prime}}= e^{-t\phi}(\phi ^{\prime}-\phi
)e^{t\phi^{\prime}}
$$
and integrating we find
$$
e^{-t\phi}e^{t\phi^{\prime}}- I = \int_0^t e^{-s\phi}(\phi^{\prime}-\phi
)e^{s\phi^{\prime}}ds
$$
and multiplying by $e^{-t\phi^{\prime}}$ on the right we have
$$
e^{-t\phi} - e^{-t\phi^{\prime}} = \int_0^t e^{-s\phi}(\phi^{\prime}-\phi
)e^{(s-t)\phi^{\prime}}ds.
$$
Since $e^{-s\phi},\medspace \phi^{\prime}-\phi$ and $e^{(s-t)\phi^{\prime} }$
are completely positive for $0 \leq s \leq t$ the expression on the right hand
side of the above equation is completely positive and, hence,
$$
\phi^{-1} - \phi^{\prime-1} = \int_0^\infty (e^{-t\phi} - e^{-t\phi
^{\prime} }) \,dt
$$
is completely positive so $\phi^{-1} \geq \phi^{\prime-1}.$ \end{proof}

\section{Index zero $q$-weight map of full range rank}

We begin this section by introducing what will call the skeleton of a
$q$-weight map where by the skeleton we mean the mapping $\phi_t = \omega
\vert_t \Lambda$.  This means we focus only on $\omega \vert_t \Lambda$ and
 ignore any further details contained in $\omega$.  The notion of the
skeleton is more psychological than mathematical.  The idea is that
$q$-weight map are fairly complicated objects and in focusing on the skeleton
we concentrate on a less complicated object.  We believe a complete
understanding to skeletons would go a very long way toward understanding
$q$-weight maps.  We have discovered a great deal of information is contained
in the skeleton.  For example, in an earlier paper we showed that every
range rank two $q$-weight map is never $q$-pure and some subordinates can be
constructed with no further knowledge than the skeleton.  We will say that
a $q$-weight map has a pure skeleton if the subordinates of the $q$-weight map
than are constructed from only a knowledge of the skeleton are totally ordered.
This is a somewhat vague statement but it will become precise in specific
situations.

We begin with the general properties of skeletons.  First we note that
$\omega \vert_t (I)$ can be computed from the skeleton.  This is because
$$
\frac {d} {dt} \omega \vert_t (I) = e^t \frac {d} {dt}\phi_t(I)
$$ 
and, therefore, $\omega \vert_t (I)$ can be calculated from $\phi_t(I)$ by
integration
\begin{equation}\label{skellie}
\omega \vert_t (I) = e^t\phi_t(I) + \int_t^\infty e^s\phi_s(I) \,ds.
\end{equation}  
The important point is when we refer to $\omega \vert_t (I)$ we are
referring to something that is computable from the skeleton so, for
example, the fact that $\omega$ is unital is computable from its skeleton
since $\omega$ is unital if and only if $\omega \vert_t (I - \Lambda ) =
\omega \vert_t (I) - \phi_t(I) \rightarrow I$ as $t \rightarrow 0+$.

\begin{thm}\label{skeleton}
Suppose $\omega$ is a $q$-weight map over $\mathbb C^p$
and $\phi _t = \omega \vert_t \Lambda$ for $t > 0$.  Then $\phi_t$ has the
following properties.
\begin{enumerate}[(i)]
\item $\phi_t$ is completely positive.

\item $\phi_t$ is non increasing so $\phi_t \geq \phi_s$ for $0 < t
\leq s$.

\item $(\iota + \phi_t)^{-1}$ exists and  $(\iota +
\phi_t)^{-1}\phi_t$ is a completely positive contractive map.

\item $(\iota + \phi_t)^{-1}(\phi_s - \phi_r)$ is completely positive
for $0 < t \leq s \leq r$.

\item  $(\iota + \phi_t)^{-1}\omega \vert_t (I) \leq I$.

\item $\iota - (\iota + \phi_t)^{-1}(\iota + \phi_s)$ is completely
positive for $0 < t < s$.

\item  $(\iota + \phi_t)^{-1}(\iota + \phi_s)$ is conditionally
negative for $0 < t < s$.
\end{enumerate}
\end{thm}
\begin{proof}  Assume the hypothesis of the theorem.  Conditions (i) and
(iii) follow from the definition of a $q$-weight map.  Condition (ii) follows
from the fact that for $0 < t \leq s$ we have
$$
(\phi_t - \phi_s)(A) = \omega \vert_t \Lambda (A) - \omega \vert_s \Lambda
(A) = \omega (E(t,s)\Lambda (A)E(t,s))
$$
where $E(t,s)$ is the orthogonal projection in $B(\mathbb C^p\otimes
L^2(0,\infty ))$ onto functions with support in the interval $[t,s]$.

Condition (iv) follows from the fact that the generalized boundary
representation $\pi_t^\#$ is completely positive and
$$
(\iota + \phi_t)^{-1}(\phi_s - \phi_r)(A) = \pi_t^\# (E(s,r)\Lambda
(A)E(s,r))
$$
and condition (v) is just the statement that $\pi_t^\#$ is completely
contractive.

Condition (vi) follows from the fact that $(\iota + \phi_t)^{-1}\phi_s$ is
non increasing in $s$ for $0 < t < s$ and $(\iota + \phi_t)^{-1}(\iota +
\phi_t) = \iota$ and since the identity map is conditionally zero
condition (vii) follows from condition (vi).  \end{proof}

\begin{defn}
If $\phi=\{\phi_t\}_{t \geq 0}$ is a one parameter family of maps
satisfying the conditions of Theorem~\ref{skeleton}, where the formula for
$\omega_t(I)$ in condition (v) is replaced by the right-hand side
of Equation~\eqref{skellie}, we say $\phi$ is an admissible
skeleton over $\mathbb C^p$.
\end{defn}

We believe that if $\phi$ is an admissible skeleton then there is a
$q$-weight map $\omega$ over $\mathbb C^p$ so that $\phi_t = \omega \vert_t
\Lambda$ for $t > 0$.  We believe further that if $\phi$ is a pure
admissible skeleton by which we mean you can not see that the $q$-weight map
$\omega$ associated with $\phi$ is not $q$-pure then there is a $q$-pure
$q$-weight map $\omega$ so that $\phi_t = \omega \vert_t \Lambda$ for $t > 0$.
In the case of a $q$-weight map $\omega$ over $\mathbb C$, one sees that the
$q$-weight map
$$
\omega (A) = (f,Af)\qquad \text{where} \qquad f(x) =
(-e^{x} \frac {d} {dx}\phi_x(I))^{\tfrac{1}{2}}
$$
is a $q$-pure $q$-weight map so that $\phi_t = \omega \vert_t \Lambda$.

The next result except for the last sentence first appeared in \cite{jankowski1}
and then in \cite{jmp1} but because of its importance and the shortness of
the proof we include it here.

\begin{thm}
Suppose $\omega$ is a $q$-weight map over $\mathbb C^p$ of
index zero and $\pi_t^\#$ is the generalized boundary representation of
$\omega$.  Then $\pi_t^\#\Lambda (A) \rightarrow A$ as $t \rightarrow 0+$
for all $A$ in the range of $\omega$.  There is a $\delta > 0$ so that the
range of the skeleton $\phi_t = \omega \vert_t \Lambda$ is the range of
$\omega$ for all $t \in (0,\delta )$.  Furthermore, if $L$ is any limit point
of the $\pi_t^\#\Lambda$ as $t \rightarrow 0+$ then $L$ is a completely
positive idempotent map with range $\mathcal{L}$ the range of $\omega$ and
$\mathcal{L}$ is $*$-algebra which we call the Choi-Effros algebra with
multiplication given by $A\star B = L(AB)$.  Furthermore, if $\omega ^{\prime}$
is a $q$-subordinate of $\omega$ then $\omega ^{\prime}$ is of index zero.
\end{thm}
\begin{proof}  Assume the hypothesis and notation of the theorem.  Note that
the range of $\omega \vert_t$ is non increasing in $t$ so that for $0 < t <
s$ we have $Range (\omega \vert_t ) \supset Range(\omega \vert_s )$.
Since $B( \mathbb C^p)$ is finite dimensional there is an $s > 0$ so that
$Range(\omega \vert_t ) = Range(\omega \vert_s )$ for $0 < t \leq s$.  Now
suppose $A \in Range(\omega )$ so there is an $B \in \mathfrak A (\mathbb C^p)$
so that $\omega \vert_s (B) = A$ and since $\omega \vert_s (B) = \omega
(E(s,\infty )BE(s,\infty ))$ (where $E(s, \infty )$ is the projection in
$\mathfrak A (\mathbb C^p)$ onto functions with support in $[s,\infty ))$ we may
assume $B$ has support in $E(s,\infty )(\mathbb C^p\otimes L^2(0,\infty ))$ so
$B = E(s,\infty )BE(s,\infty )$.  Then we have
\begin{align*}
A - \pi_t^\#\Lambda (A) &= (\iota + \omega \vert_t \Lambda )^{-1}((\iota +
\omega \vert_t \Lambda )(A) - \omega \vert_t \Lambda (A))
\\
&= (\iota + \omega \vert_t \Lambda )^{-1}(A) = (\iota + \omega \vert_t
\Lambda )^{-1}\omega \vert_s (B)
\\
&= (\iota + \omega \vert_t \Lambda )^{-1}\omega \vert_t (B) = \pi_t^\# (B)
\rightarrow \pi_o^\# (B) = 0
\end{align*}
as $t \rightarrow 0+$.  Hence, $\pi_t^\#\Lambda (A) \rightarrow A$ as $t
\rightarrow 0+$.

Clearly the range of $\pi_t^\#\Lambda$ is contained in the range of $\omega$
for all $t > 0$ and since $B(\mathbb C^p)$ is finite dimensional and
$\pi_t^\#\Lambda (A) \rightarrow A$ for all A in the range of $\omega$ is
follows there is a $\delta > 0$ so that the range of $\pi_t^\#\Lambda$ is
equal to the range of $\omega$ for $t \in (0,\delta ).  $ Recall in term of
the skeleton $\phi_t = \omega \vert_t \Lambda$ we have $\pi_t^\#\Lambda =
(\iota + \phi_t)^{-1}\phi_t = \phi_t(\iota + \phi_t)^{-1}$ so the range of
$\pi_t^\#\Lambda$ is the range of $\phi_t$.  Hence the range of $\phi _t$
is the range of $\omega$ for all $t \in (0,\delta )$.

Now if $L$ is any limit point of $\pi_t^\#\Lambda$ as $t \rightarrow 0+$
then $L(A)$ for each $A \in B(\mathbb C^p)$ is the limit of operators in the
range of $\omega$ and since $B(\mathbb C^p)$ is finite dimensional everything
that is the limit of elements in the range of $\omega$ is, in fact, in the
range of $\omega$.  Hence, $L(A)$ is in the range of $\omega$ and,
therefore, $\pi_t^\# (L(A)) \rightarrow L(A)$ as $t \rightarrow 0+$.
Hence, we have $L(L(A)) = L(A)$.

From the result of Choi and Effros \cite{CE} we have
$$
L(AL(B)) = L(L(A)L(B)) = L(L(A)B)
$$
for $A,B \in B(\mathbb C^p)$.  Note $L(A) = A$ for $A$ in the range of $\omega
$.  Choi and Effros define the multiplication as
$$
A\star B = L(AB).
$$
and equipped with this multiplication $\mathcal{L} (\omega )$ is a
$C^*$-algebra.  In the paper of Choi and Effros they assume that $I_o = L(I)$
is the unit $I$ of $B(\mathbb C^p)$ but in our finite dimensional case this
assumption is not needed (see Theorem 5.1 in the next section).

Next suppose $\omega ^{\prime}$ is a $q$-subordinate of $\omega$.  Suppose
$A \in \mathfrak A (\mathbb C^p)$, $A \geq 0$ and $A = E(r,\infty )AE(r,\infty)$
for $r > 0$ where $E(r,\infty )$ is the orthogonal projection in $B(\mathbb
C^p\otimes L^2(0,\infty ))$ onto function with support in $[r,\infty )$.
Suppose $\pi_t^\#$ and $\pi_t^{\prime\#}$ are the generalized boundary
representations of $\omega$ and $\omega ^{\prime}$, respectively.  Since
$\omega$ is of index zero we have $\pi_t^\# (A) \rightarrow 0$ as $t
\rightarrow 0+$ and since $\pi_t^\# \geq \pi_t^{\prime\#}$ we have $\pi
_t^{\prime\#} (A) \rightarrow 0$ as $t \rightarrow 0+$.  Since every $A =
E(r,\infty )AE(r,\infty )$ can be written as a linear combination of
positive elements we have $\pi_t^{\prime}(E(r,\infty )AE(r,\infty ))
\rightarrow 0$ as $t \rightarrow 0+$.  Hence, the normal spine of
$\pi_t^{\prime\#}$ is zero and $\omega ^{\prime}$ is of index zero.
\end{proof}

Now we further specialize our attention to $q$-weight maps of index zero and
full range rank by which we mean for every $A \in B(\mathbb C^p)$ there is an
$B \in \mathfrak A (\mathbb C^p)$ so that $\omega (B) = A$.  It then follows for
such $q$-weight maps that the skeleton $\phi_t$ has range $B(\mathbb C^p)$ for
small $t$ and if $\pi _t^\#$ is the generalized boundary representation of
$\omega$ then $\pi_t^\#\Lambda \rightarrow \iota$ as $t \rightarrow 0+$.
We prove an important theorem.

\begin{thm}
Suppose $\omega$ is a $q$-weight map over $\mathbb C^p$ of
index zero and the range of $\omega$ is $B(\mathbb C^p)$.  For $t > 0$ let
$\phi_t = \omega \vert_t \Lambda$ and let $\pi_t^\#$ be the generalized
boundary representation of $\omega$.  Let
$$
v_t = tr(I + \phi_t(I))\qquad \text{and} \qquad \Theta_t = v_t^{-1}(\iota +
\phi_t)
$$
where $tr$ is the trace on $B(\mathbb C^p)$ normalized so that $tr(I) = 1$.
Then $\Theta _t$ is completely positive and invertible and the inverse
$\Theta_t^{-1}$ is conditionally negative and $\Theta_t$ converges to a
limit $\Theta$ as $t \rightarrow 0+$ and $\Theta$ is completely positive.
The limit $\Theta$ is invertible and the inverse $\Theta^{-1}$ is
conditionally negative and $\Theta_t^{-1} \rightarrow \Theta ^{-1}$ as $t
\rightarrow 0+$.

Furthermore, $\vartheta = \Theta^{-1}\omega$ is a completely positive
$B(\mathbb C ^p)$ valued $b$-weight map on $\mathfrak A (\mathbb C^p)$ (i.e.
$\vartheta \in B(\mathbb C^p) \otimes \mathfrak A (\mathbb C^p)_*)$ with the
property that
$$
\vartheta \vert_t \Lambda  = v_t(\iota + \nu_t)
$$
and $1/v_t \rightarrow 0$ and $\Vert\nu_t\Vert \rightarrow 0$ as $t
\rightarrow 0+$.
\end{thm}
\begin{proof}  Assume the hypothesis and notation of the theorem.  Note
$\Theta _t$ is completely positive and $tr(\Theta_t(I)) = 1$ for all $t >
0$.  We will show that $\Vert\Theta_t - \Theta_s\Vert \rightarrow 0$ as
$t,s \rightarrow 0+$.  Now suppose $0 < \epsilon < 0.1$. and
$$
\epsilon_1 = \frac {\epsilon} {2p+\epsilon}\qquad \text{and} \qquad
\delta_1 = \min(\frac {\epsilon_1} {4p^2},\medspace \frac {0.1}
{\sqrt{p}}).
$$
Since
$$
\pi_t^\#\Lambda = \iota - (\iota+\phi_t)^{-1} \rightarrow \iota
$$
as $t \rightarrow 0+$, there is a $\delta > 0$ so that  
$$
\Vert (\iota + \phi_t)^{-1}\Vert < \delta_1
$$
for $t \in (0,\delta )$.  Now suppose that $0 < t < s < \delta$.  Recall
from the properties of the skeleton we have
$$
\xi_1 = (\iota + \phi_t)^{-1}\phi_s \geq 0\qquad \text{and} \qquad \xi_2 =
(\iota + \phi_t)^{-1}(\phi_t - \phi_s) \geq 0.
$$
Note that
$$
\iota - \xi_2 = (\iota + \phi_t)^{-1}(\iota + \phi_s)
$$
and
$$
\Vert\iota - (\iota - \xi_2)^{-1}\xi_1\Vert = \Vert (\iota +
\phi_s)^{-1}\Vert < \delta_1.
$$
Since $\Vert\xi_1 + \xi_2\Vert = \Vert (\iota + \phi_t)^{-1}\phi_t\Vert =
\Vert \pi_t^\#\Lambda\Vert < 1$ we have by lemma 3.1
$$
(\iota + \phi_t)^{-1}(\iota + \phi_s) = \kappa (\iota + \eta )\qquad
\text{so } \qquad \iota + \phi_s = \kappa (\iota + \phi_t)(\iota + \eta )
$$
where $\kappa > 0$ and $\Vert\eta\Vert < \epsilon_1$.  Now for $\Theta_t$
as defined above we have
$$
\Theta_s = \kappa v_t/v_s\Theta_t(\iota+\eta )
$$
and since $tr(\Theta_s(I)) = tr(\Theta_t(I)) = 1$ we have
$$
1 = \kappa v_t/v_s(1+tr(\Theta_t(\eta (I))))
$$
and since
$$
-\Vert\eta\Vert I \leq \eta (I) \leq \Vert\eta\Vert I
$$
we have
$$
-\epsilon_1 < - \Vert\eta\Vert \leq tr(\Theta_t(\eta (I))) \leq
\Vert\eta\Vert < \epsilon_1
$$
so
$$
\left\vert 1 - \frac {v_s} {\kappa v_t}\right\vert < \epsilon_1\qquad
\text{and} \qquad \left \vert 1 - \frac {\kappa v_t} {v_s}\right\vert <
\frac {\epsilon_1} {1-\epsilon _1}
$$
and
\begin{align*}
\Vert\Theta_s - \Theta_t\Vert  &= \Vert\Theta_t((\kappa v_t/v_s-1)\iota +
\kappa v_t/v_s\eta )\Vert
\\
&\leq \Vert\Theta_t\Vert (\vert\kappa v_t/v_s-1\vert + \kappa
v_t/v_s\Vert\eta \Vert )
\\
&\leq \Vert\Theta_t\Vert (\frac {\epsilon_1} {1-\epsilon_1} + \frac
{\Vert\eta \Vert} {1-\epsilon_1}) < \Vert\Theta_t\Vert (\frac {2\epsilon_1}
{1-\epsilon _1}) \leq \frac {2p\epsilon_1} {1-\epsilon_1} \leq \epsilon
\end{align*}
where we used the estimate $\Vert\Theta_t\Vert = \Vert\Theta_t(I)\Vert \leq
p \cdot tr(\Theta_t(I)) = p$.  Hence, $\Vert\Theta_s - \Theta_t\Vert <
\epsilon$ for $t,s \in (0,\delta )$ and, hence, $\Theta _t$ converges to a
limit $\Theta$ as $t \rightarrow 0+$.  Since $\Theta_t$ is the limit of
completely positive maps we have $\Theta$ is completely positive.  For $0 <
t < s < \delta$ we have
$$
(\iota + \phi_t)^{-1}(\iota + \phi_s) = \kappa (\iota + \eta )
$$
where $\Vert\eta\Vert < \epsilon_1$ which yields
$$
\Theta_t^{-1}\Theta_s = (\kappa v_t/v_s)(\iota + \eta ) = \iota - (\kappa
v_t/v_s-1)\iota + \kappa v_t/v_s\eta.
$$
So
$$
\Vert\iota - \Theta_t^{-1}\Theta_s\Vert \leq \vert\kappa v_t/v_s-1\vert +
\kappa v_t/v_s\Vert\eta\Vert < \frac {2\epsilon_1} {1-\epsilon_1}  \leq
\epsilon < 0.1.
$$
Now if $\Vert\iota - x\Vert < s$ with $0 < s < 1$ then $x$ is invertible
with $\Vert\iota - x^{-1}\Vert < s/(1-s)$ so we have
$$
\Vert\iota - \Theta_s^{-1}\Theta_t\Vert < \frac {0.1} {1-0.1} = 1/9
$$
and letting $t \rightarrow 0+$ we have
$$
\Vert\iota - \Theta_s^{-1}\Theta\Vert \leq 1/9
$$
and $\Theta_s^{-1}\Theta$ is invertible so $\Theta^{-1} =
(\Theta_s^{-1}\Theta )^{-1}\Theta_s^{-1}$.  Hence, $\Theta$ is invertible
and since $\Theta_t \rightarrow \Theta$ as $t \rightarrow 0+$ we have
$\Theta _t^{-1} \rightarrow \Theta^{-1}$ as $t \rightarrow 0+$.  Since
$\Theta$ is the limit of completely positive maps $\Theta$ is completely
positive.

Note that
$$
\Theta_t^{-1}\omega \vert_t  = v_t\pi_t^\#
$$
for $t > 0$ so $\Theta_t^{-1}\omega \vert_t$  is completely positive and
since $\Theta_t^{-1} \rightarrow \Theta^{-1}$ and $\omega \vert_t
\rightarrow \omega$ in $B(\mathbb C^p) \otimes \mathfrak A (\mathbb C^p)_*$ we
have $\vartheta = \Theta^{-1}\omega$ is a completely positive $B(\mathbb C^p)$
valued $b$-weight map on $\mathfrak A (\mathbb C^p)$.

Since $\vartheta = \Theta^{-1}\omega$  we have
$$
\vartheta \vert_t \Lambda = \Theta^{-1}\omega \vert_t \Lambda =
\Theta^{-1}\phi_t.
$$
Now we have
$$
\pi_t^\#\Lambda = (\iota + \phi_t)^{-1}\phi_t  =
v_t^{-1}\Theta_t^{-1}\phi_t
$$
so
$$
\vartheta \vert_t \Lambda = \Theta^{-1}v_t\Theta_t\pi_t^\#\Lambda =
v_t(\Theta^{-1}\Theta_t)\pi_t^\#\Lambda.
$$
Then we have
$$
\vartheta \vert_t \Lambda = v_t(\iota + \nu_t)\qquad \text{where} \qquad
\nu_t = \Theta^{-1}\Theta_t\pi_t^\#\Lambda - \iota
$$
and since $\pi_t^\#\Lambda \rightarrow \iota$ and $\Theta^{-1}\Theta_t
\rightarrow \iota$ as $t \rightarrow 0$ it follows that
$$
\Vert\nu_t\Vert \rightarrow 0
$$
as $t \rightarrow 0+$.  Note that $(\iota + \phi_t)^{-1} = v_t^{-1}\Theta_t
^{-1}$ and since $\Theta_t^{-1} \rightarrow \Theta^{-1}$ and $\Vert (\iota
+ \phi_t)^{-1}\Vert \rightarrow 0$ as $t \rightarrow 0+$ it follows that
$v_t^{-1} \rightarrow 0$ as $t \rightarrow 0+$.

To see that $\Theta^{-1}$ is conditionally negative note that
$$
\pi_t^\#\Lambda = \iota - (\iota + \phi_t)^{-1} = \iota - v_t^{-1}\Theta_t^
{-1}
$$
so if $A_i \in B(\mathbb C^p)$ and $f_i \in \mathbb C^p$ for $i = 1,\cdots ,n$
and
$$
\sum_{i=1}^n A_if_i  = 0
$$
then
$$
\sum_{i,j=1}^n v_t(f_i,\pi_t^\#\Lambda (A_i^*A_j)f_j) = - \sum_{i,j=1}^n (f
_i,\Theta_t^{-1}(A_i^*A_j)f_j) \geq 0
$$
so $\Theta_t^{-1}$ is conditionally negative and its limit $\Theta^{-1}$ is
conditionally negative.   \end{proof}

Now we make an important switch in notation.  The previous theorem proved
statements about the mapping $\Theta$.  Now we will focus are attention on
the mapping $\psi$ which is the inverse of $\Theta$.  The reason for this
change of focus is that it is much easier to characterize subordinates of a
$q$-weight map in terms of the mapping $\psi$.
Express in terms of $\psi$ (the inverse of $\Theta )$ the conclusion of
Theorem 4.4 is that every $q$-weight map $\omega$ over $\mathbb C^p$ of index
zero and range of $B(\mathbb C^p)$ can be written in the form
$\omega = \psi^{-1}\vartheta$ where $\vartheta \in B(\mathbb C^p) \otimes
\mathfrak A (\mathbb C^p)_*$ is completely positive and $\psi \in \mathcal{S}
(\mathbb C^p)$ is conditionally negative and $\psi$ is invertible with
$\psi^{-1}$ completely positive and
$$
\vartheta \vert_t \Lambda = v_t(\iota + \nu_t)
$$
where $v_t \rightarrow \infty$ and $\Vert\nu_t\Vert \rightarrow 0$ as $t
\rightarrow 0+$.  As we will see it is far easier to discover the
properties of $\vartheta$ and $\psi$ then to determine the properties of
$\omega$.  The problem with the above expression for $\vartheta  \vert_t
\Lambda$ is that the definition of $\nu_t$ is ambiguous since we can add or
subtract a multiple of $\iota$ to it.  The method we will use is to choose
$\nu_t$ so as to minimize its Hilbert Schmidt norm.  This then makes the
definition of $\nu_t$ unique.  In Lemma 4.5 we will describe how to express
$\vartheta$ so that $\nu_t$ is easily computable and afterwards we will
adopt the conventions of that lemma.

Note that in the proof of Theorem 4.4 the map $\psi$ has the property that
$tr(\psi^{-1}) = tr(\Theta ) = 1$.  Now in the future we will not require a
particular normalization of $\psi$ and $\vartheta$.  Note if we replace
$\psi$ by $s\psi$ and $\vartheta$ by $s\vartheta$ then $\omega =
\psi^{-1}\vartheta$ remains unchanged.  So from now on $\psi$ and
$\vartheta$ are only unique up to multiplying both by a positive number.
There are advantages and disadvantages to requiring a specific
normalization of $\psi$ and $\vartheta$ and we decided in favor of more
freedom particularly because given $\omega = \psi^{-1}\vartheta$ we will
want to consider other $q$-weight maps of the form $\omega ^{\prime} = \psi
^{\prime- 1}\vartheta$.

Notice in the expression $\vartheta \vert_t \Lambda$ if we change the
normalization of $\vartheta$ so $\vartheta ^{\prime} = s\vartheta$ then we
have
$$
\vartheta ^{\prime} \vert_t \Lambda = v_t^{\prime}(\iota + \nu_t)
$$
where $v_t^{\prime} = sv_t \rightarrow \infty$ and $\Vert\nu_t\Vert
\rightarrow 0$ as $t \rightarrow 0+$.

The reader may well wonder why we use $\psi$ in the pair $(\psi ,\vartheta
)$ when it was the inverse $\psi^{-1} = \Theta$ which we first defined.
This is because the conditions that the $q$-weight map $\omega =
\psi^{-1}\vartheta$ be a $q$-weight map are most easily expressed in terms of
$\psi$.  In the next lemma has very little mathematical content.  It simply
shows that a completely positive $b$-weight map can be expressed in a certain
form.  Yet putting $\vartheta$ in this form proved to be of enormous
importance.  It had taken us almost three months of work using Maple and
Matlab to construct a single example of a $q$-pure non Shur $q$-weight map of
index zero for the simplest case of $p = 2$ but once we struck upon the
idea of constructing examples of the form $\psi^{-1} \vartheta$ with
$\vartheta$ of the form given in the next lemma we found we could construct
barrel loads of examples for any $p$.  Calculation that had taken months
could now be done in hours.  To understand what is going on consider the
case where a $q$-weight map $\omega$ is bounded.  Then one can take the limit
of the generalized boundary representation $\pi_t^\#$ as $t \rightarrow
0+$.  Everything can be calculated from $\pi_o^\#$.  The index of $\omega$
is just the index of $\pi_o^\#$ and all $q$-subordinates of $\omega$ can be
found by taking subordinates of $\pi_o^\#$.  One can easily write down
examples of $\pi _o^\#$ and the associated $\omega$ computed.  In short
$\pi_o^\#$ a completely positive contractive map is easy to write down and
analyze but the $\omega$ is hard to find and hard to analyze.  What the
next lemma does is give us a form for $\vartheta$ much like the form for
$\pi_o^\#$ which is relatively easy to construct.  So even though the
result of the next lemma is not at all deep the form we arrived at made
problems that we had been struggling with for years suddenly tractable.

We should also mention that the next lemma contains a definition.  We will
use the decomposition in the lemma repeatedly in the rest of the paper.

\begin{lem} 
Suppose $\vartheta \in B(\mathbb C^p) \otimes \mathfrak A
(\mathbb C ^p)_*$ is completely positive.  Then $\vartheta$ can be expressed
in the form
$$
\vartheta_{ij}(A) = \sum_{k\in J} ((g_{ik}+h_{ik}),A(g_{jk}+h_{jk}))
$$
where the $g_{ik},h_{ik} \in \mathbb C^p\otimes L_+^2(0,\infty )$ and
$$
(g_{ik})_j(x) = \delta_{ij}g_k(x)\qquad \text{and} \qquad \sum_{i=1}^p
(h_{ik })_i(x) = 0
$$
for $A \in \mathfrak A (\mathbb C^p),\medspace x \geq 0,\medspace i,j \in \{
1,\cdots ,p\}$ and $k \in J$ a countable index set.  Since $\vartheta \in
B(\mathbb C^p) \otimes \mathfrak A (\mathbb C^p)_*$  we have $tr(\vartheta (I -
\Lambda )) < \infty$ where
\begin{align*}
tr(\vartheta (I-\Lambda )) &= \frac {1} {p} \sum_{i=1}^p \vartheta_{ii}
(I-\Lambda)
\\
&=\sum_{k\in J} \int_0^\infty (1-e^{-x})(\vert g_k(x)\vert^2 + \frac {1}
{p} \sum_{i=1}^p \Vert h_{ik}(x)\Vert^2) \,dx.
\end{align*}
We define the completely positive $B(\mathbb C^p)$ valued $b$-weight map
$\rho \in B(\mathbb C^p) \otimes \mathfrak A (\mathbb C ^p)_*$
given by
$$
\rho_{ij}(A) = \sum_{k\in J} (h_{ik},Ah_{jk})
$$
for $A \in \mathfrak A (\mathbb C^p)$ and for $t > 0$ we define
$$
w_t = \sum_{k\in J} (g_k,\Lambda \vert_t g_k), \qquad (R_t)_{ij}(B) =
\rho_{ij}\vert_t (\Lambda (B)),\qquad (Y_t)_{ij} = \sum_{k\in J}
((h_{ik})_j,\Lambda \vert_t g_k)
$$
and
$$
\zeta_t(A) = Y_tA
$$
for $A \in B(\mathbb C^p)$.  Then we have
$$
\vartheta \vert_t \Lambda (A) = w_tA + Y_tA + AY_t^* + R_t(A)
$$
for $A \in B(\mathbb C^p)or$
$$
\vartheta \vert_t \Lambda = w_t\iota + \zeta_t + \zeta_t^* + R_t
$$
and $R_t$ is of the form
$$
R_t(A) = \sum_{i=1}^m \lambda_i(t)X_i(t)AX_i(t)^*
$$
and $Y_t$ and $X_t(i)$ are of trace zero and $tr(X_i(t)^*X_j(t)) = \delta_{ij}$.

From the expression for $tr(\vartheta (I-\Lambda )) < \infty$ above it
follows that $tr(\rho (I-\Lambda)) \leq tr(\vartheta (I-\Lambda )) < \infty$.
\end{lem}
\begin{proof}  Assume $\vartheta \in B(\mathbb C^p) \otimes \mathfrak A
(\mathbb C^p)_ *$ is a completely positive.  From the general theory of
completely positive maps we know that $\vartheta$ can be written in the form
$$
\vartheta_{ij}(A) = \sum_{k\in J} (F_{ik},AF_{jk})
$$
with the $F_{ik} \in \mathbb C^p\otimes L_+^2(0,\infty )$ for $k \in J$ a
countable index set.  Furthermore, we know the $F_{ik}$ can be chosen so
they are linearly independent over $\ell^2(J)$.  Now we simply define
$$
g_k(x) = \frac {1} {p} \sum_{i=1}^p (F_{ik})_i(x)
$$
and then define
$$
(g_{ik})_j(x) = \delta_{ij}g_k(x)\qquad \text{and} \qquad
h_{ik}(x) = F_{ik}(x) - g_{ik}(x)
$$
for $x \geq 0,\medspace i,j \in \{ 1,\cdots ,p\}$ and $k \in J$.

The further results of the lemma follow from the discussion in section 3
and straight forward computation.  \end{proof}

In the above lemma we concluded that for $t > 0$
\begin{equation*}
\vartheta\vert_t\Lambda = w_t\iota + \zeta_t + \zeta_t^* + R_t = w_t(\iota
+ \gamma_t)
\end{equation*}
where
\begin{equation*}
\gamma_t = w_t^{-1}(\zeta_t + \zeta_t^* + R_t)
\end{equation*}
and in Theorem 4.4 we concluded that $\vartheta\vert_t\Lambda = v_t(\iota +
\nu_t)$ and $\Vert\nu_t\Vert \rightarrow 0$ as $t \rightarrow 0+$ where
$v_t = tr(I+\omega \vert_t\Lambda (I))$ and $\nu_t =
\Theta^{-1}\Theta_t\pi_t^\#\Lambda - \iota $.  We remark on the connection
between the two formulae.  From equation 3.4 of section 3 we see that
$R_t$ is the internal part of $\vartheta\vert _t\Lambda$ and $w_t$ is the
coefficient of the identical part of $\vartheta\vert_t\Lambda$.  Since both
are equal to $\vartheta\vert_t\Lambda$ we have
$w_t(\iota + \gamma_t) = v_t(\iota + \nu_t)$ for $t > 0$ and $v_t
\rightarrow \infty$ and $\Vert\nu_t\Vert \rightarrow 0$ as $t \rightarrow
0+$ and $w_t$ is the coefficient of the identical part of $\vartheta
\vert_t\Lambda$ it follows that $w_t/v_t \rightarrow 1$ and
$\Vert\gamma_t\Vert \rightarrow 0$ so the two forms are roughly equivalent in
the limit.

Once we wrote down $\vartheta$ in the form given in the lemma we considered
what would happen if $\rho$ is bounded and we found we could construct
$q$-weight maps fairly easily.  Next we will need a technical lemma for making
estimates.

\begin{lem}
Suppose $\vartheta$ is a completely positive normal
$B( \mathbb C^p)$ valued $b$-weight map of the form given in the previous lemma
and suppose $\rho$ is a bounded $b$-weight map and suppose $w_t \rightarrow
\infty$ as $t \rightarrow 0+$.  Then
$$
\vert (Y_t)_{ij}\vert^2/w_t \rightarrow 0
$$
as $t \rightarrow 0+$ for each $i,j \in \{ 1,\cdots ,p\}$.
\end{lem}
\begin{proof}  Assume the hypothesis and notation of the lemma.  Then from
inequality (3.7) of the last section we have the $(2 \times 2)$-matrix of
matrices in $B(\mathbb C^p)$
$$M(t) = \left[\begin{matrix} w_tI&Y_t^*
\\
Y_t&R_t(I)
\end{matrix} \right]
$$
is positive and non increasing so $M(t) \geq M(s)$ for $0 < t < s$.  Now we
have $w_t \rightarrow \infty$  and $R_t(I) \rightarrow R_o(I)$ as $t
\rightarrow 0+$.  Choose a pair of indices $i,j \in \{ 1,\cdots ,p\}$.  and
let $N(t)$ be the $(2 \times 2)$-matrix
$$
N(t) = \left[\begin{matrix} w_t&(Y_t^*)_{ij}
\\
(Y_t)_{ji}&(R_t)_{jj}
\end{matrix} \right]
= \left[\begin{matrix} a(t)&\overline{b(t)}
\\
b(t)&c(t)
\end{matrix} \right].
$$
Where the second matrix on the right is the definition of $a(t),\medspace b(t)$
and $c(t)$.  Note $N(t)$ is positive and non increasing in $t$ so $N(t) \geq
N(s)$ for $0 < t < s$.  Now suppose $\epsilon > 0$.  Consider the sequence
$0 < t_2 < t_1 < t_o$ and let
$$
a_k = a(t_k) - a(t_{k-1}),\qquad b_k = b(t_k) - b(t_{k-1}),\qquad
c_k = c(t_k)- c(t_{k-1})
$$
for $k = 1,2$.  Since $N(t_k) - N(t_{k-1}) \geq 0$ we have
$$
\vert b_k\vert^2 \leq a_kc_k
$$
and, hence
\begin{align*}
\vert b(t_2)\vert &= \vert b_2 + b_1 + b(t_o)\vert \leq \vert b_2\vert +
\vert b_1\vert + \vert b(t_o)\vert
\\
&\leq \sqrt{a_2c_2 } + \sqrt{a_1c_1 }  + \vert b(t_o)\vert.
\end{align*}
Since $c(t) \rightarrow c(0)$ as $t \rightarrow 0+$ we can choose $t_1$ so
that $c(0) - c(t) < \epsilon^2/4$ for $0 < t < t_1$.  Then
$$
\vert b(t_2)\vert \leq \tfrac{1}{2}\epsilon\sqrt{a(t_2)-a(t_1)}  + \sqrt{a_1c_1
} + \vert b(t_o)\vert
$$
so we have
$$
\frac {\vert b(t_2)\vert} {\sqrt{a(t_2)}}  \leq
\tfrac{1}{2}\epsilon\sqrt{1-a(t_1) /a(t_2)}+a(t_2)^{-\tfrac{1}{2}}(\sqrt{a_1c_1}
+ \vert b(t_o)\vert )
$$
and since $a(t_2) \rightarrow \infty$ as $t_2 \rightarrow 0+$ there is a
$\delta > 0$ so that the right hand side of the above inequality is less
than $\epsilon$ for $0 < t_2 < \delta $.  Hence, we have shown that
$a(t)^{-\tfrac{1}{2}}\vert b(t)\vert \rightarrow 0$ as $t \rightarrow 0+$ and,
hence,
$$
\vert Y_{ij}(t)\vert^2/w_t \rightarrow 0
$$
as $t \rightarrow 0+.$ \end{proof}

\begin{thm}
Suppose $\omega$ is a $q$-weight map over $\mathbb C^p$ of
index zero and the range of $\omega$ is $B(\mathbb C^p)$.  Then $\omega$ is of
the form $\omega = \psi^{-1}\vartheta$ where $\psi$ is an invertible
conditionally negative map of $B(\mathbb C^p)$ into itself with a completely
positive inverse and $\vartheta$ is of the form
$$
\vartheta_{ij}(A) = \sum_{k\in J} ((g_{ik}+h_{ik}),A(g_{jk}+h_{jk}))
$$
where the $g_{ik},h_{ik} \in \mathbb C^p\otimes L_+^2(0,\infty )$ and
$$
(g_{ik})_j(x) = \delta_{ij}g_k(x)\qquad \text{and} \qquad \sum_{i=1}^p
(h_{ik })_i(x) = 0
$$
for $A \in \mathfrak A (\mathbb C^p),\medspace x \geq 0,\medspace i,j \in \{
1,\cdots ,p\}$ and $k \in J$ a countable index set and the $h_{ik} \in \mathbb
C^p \otimes L^2(0,\infty )$ and if
$$
w_t = \sum_{k\in J} (g_k,\Lambda \vert_t g_k)\qquad \rho_{ij}(A) = \sum
_{k\in J} (h_{ik},Ah_{jk})
$$
then $\rho$ is bounded so
$$
\sum_{k\in J} \Vert h_{ik}\Vert^2 < \infty\qquad \text{and} \qquad
\sum_{k\in J} (g_k,(I-\Lambda )g_k) < \infty
$$
and $1/w_t \rightarrow 0$ as $t \rightarrow 0+$ and $\psi$ satisfies the
conditions
$$
\psi (I) \geq \vartheta (I - \Lambda )\qquad \text{and} \qquad \psi + \rho
\Lambda
$$
is conditionally negative.  Furthermore, $\omega$ is unital if and only if
$\psi (I) = \vartheta (I - \Lambda )$.

Conversely, if $\vartheta ,\medspace \rho$ and $\psi$ are as given above
then $\omega$ is a $q$-weight map over $\mathbb C^p$ of index zero and the range
of $\omega$ is all of $B( \mathbb C^p)$.  Furthermore, if $\psi ^{\prime}$ is
a second map satisfying the conditions above and $\omega ^{\prime} = \psi
^{\prime -1}\vartheta$ then $\omega ^{\prime}$ is a $q$-subordinate of
$\omega$ (i.e. $\omega \geq_q \omega ^{\prime})$ if and only if $\psi \leq
\psi ^{\prime}$.
\end{thm}
\begin{proof}  We begin with the proof in the paragraph beginning with the
word, conversely.  Suppose then that $\vartheta ,\medspace \rho$ and $\psi$
have the properties given in the statement of the theorem.  From the form
of $\vartheta$ we see $\vartheta$ is the sum of completely positive maps
and is, therefore, completely positive.  Now we define $w_t,\medspace \rho
,\medspace Y_t, R_t$ and $\zeta_t$ be defined as in Lemma
4.5 and we find
$$
\vartheta \vert_t \Lambda = w_t\iota + \zeta_t + \zeta_t^* + R_t
$$
and we compute
$$
\iota + \omega \vert_t \Lambda = \psi^{-1}(w_t\iota + \zeta_t + \zeta _t^*
+ R_t + \psi ).
$$
Then
$$
\pi_t^\# = Z(t)^{-1}w_t^{-1}\vartheta \vert_t
$$
where
$$
Z_t = (\iota + w_t^{-1}(\zeta_t + \zeta_t^* + R_t + \psi )).
$$

Recall in the discussion after Definition 1.1 to check
that $\pi_t^\#$ is a completely positive contraction for all $t > 0$ is only
necessary to check this condition for $0 < t < \delta$ for some $\delta >0$.
Since $\vartheta \vert_t$  is completely positive and $w_t > 0$ it follows
that $\pi_t^\#$ is completely positive if $Z_t^{-1}$ is completely positive.
Since $R_t$ is uniformly bounded and from Lemma 4.6 we have $\Vert \zeta_t
\Vert^2 / w_t \rightarrow 0$ and $1/w_t \rightarrow 0$ as $t \rightarrow 0+$
we have $Z_t \rightarrow \iota$ as $t \rightarrow 0+$ it follows that there are
real numbers $s,\delta > 0$ so that $Z_t(I) \geq sI$ for $0 < t < \delta$. If
$Z_t(I) \geq sI$ it follows from Lemma 3.3 that $Z_t^{-1}$ is completely
positive provided $Z(t)$ is conditionally negative.  We have $Z(t)$ is
conditionally negative if and only if $R_t + \psi$ is conditionally negative.
Since $R_t$ is non increasing as a completely positive map and
$R_t \rightarrow R_o$ as $t \rightarrow 0+$ we have
$$
R_t + \psi \leq R_o + \psi
$$
and by assumption $R_o + \psi$ is conditionally negative.  Hence,
$Z_t^{-1}$ is completely positive and $\pi_t^\#$ is completely positive for
$0 < t < \delta$.

So all that remains to check is that $\pi_t^\# (I) \leq I$ for $0 < t < \delta$.
Now we have
$$
\pi_t^\# (\Lambda ) = (\iota + w_t^{-1}(\zeta_t + \zeta_t^* + R_t + \psi ))
^{-1}(\iota + w_t^{-1}(\zeta_t + \zeta_t^* + R_t))(I)
$$
so it follows that
$$
I - \pi_t^\# (\Lambda ) = (\iota + w_t^{-1}(\zeta_t + \zeta_t^* + R_t +
\psi ))^{-1}w_t^{-1}\psi (I) = Z_t^{-1}w_t^{-1}\psi(I)
$$
and since
$$
I - \pi_t^\# (I) = I - \pi_t^\# (\Lambda ) - \pi_t^\# (I - \Lambda )
$$
we have
$$
I - \pi_t^\# (I) = Z_t^{-1}w_t^{-1}(\psi (I)-\vartheta \vert_t (I-\Lambda
))
$$
and since $Z_t^{-1}$ is completely positive and $\vartheta \vert_t (I -
\Lambda )$ is non increasing in $t$ and $\psi (I) \geq \vartheta (I -
\Lambda )$ it follows that $I \geq \pi_t^\# (I)$ so $\pi_t^\#$ is
completely positive and completely contractive for $t > 0$.  Hence,
$\omega$ is a $q$-weight map over $\mathbb C^p$.  The fact that the range of
$\omega$ is $B( \mathbb C^p)$ is apparent since $\pi_t^\#\Lambda \rightarrow
\iota$ as $t \rightarrow 0+$.  Since $w_t \rightarrow \infty$ as $t
\rightarrow 0+$ it is apparent that $\pi_t^\# (E(s,\infty )) \rightarrow 0$
as $t \rightarrow 0+$ where $E(s,\infty )$ is the projection in $B(\mathbb
C^p\otimes L^2(0,\infty ))$ onto functions with support to the right of
$s$.  Hence, the normal spine of $\pi_t^\#$ is zero so $\omega$ is of index
zero.

Finally since $\psi\omega = \vartheta$ we have $\psi\omega (I - \Lambda ) =
\vartheta (I - \Lambda )$ so if $\omega$ is unital we have $\psi (I) =
\vartheta (I - \Lambda )$ and, conversely, if $\psi (I) = \vartheta (I -
\Lambda )$ then $\omega (I - \Lambda ) = \psi^{-1}\vartheta (I - \Lambda )
= I$ so $\omega$ is unital.

Now suppose $\psi ^{\prime}$ is second invertible hermitian linear
mapping of $B(\mathbb C^p)$ into itself, $\psi ^{\prime}(I) \geq
\vartheta (I - \Lambda )$ and $\psi ^{\prime} + \rho \Lambda $
is conditionally negative.  Let $\omega ^{\prime} = \psi ^{\prime
-1}\vartheta$.  We show $\omega \geq_q \omega ^{\prime}$ if and
only if $\psi \leq \psi ^{\prime}$.

First suppose $\psi \leq \psi ^{\prime}$.  Then if $\pi_t^\#$ and
$\pi_t^{\prime\#}$ are the generalized boundary representations of $\omega$
and $\omega ^{\prime}$, respectively we have
$$
\pi_t^\# = (\iota + w_t^{-1}(\zeta_t + \zeta_t^* + R_t + \psi
))^{-1}w_t^{-1 }\vartheta \vert_t  = Z_t^{-1}w_t^{-1}\vartheta \vert_t
$$
and
$$
\pi_t^{\prime \#} = (\iota + w_t^{-1}(\zeta_t + \zeta_t^* + R_t + \psi
^{\prime})) ^{-1}w_t^{-1}\vartheta \vert_t  = Z_t^{\prime
-1}w_t^{-1}\vartheta  \vert_t
$$
and
$$
\pi_t^\# - \pi_t^{\prime \#} = (Z_t^{-1} - Z_t^{\prime
-1})w_t^{-1}\vartheta  \vert_t.
$$
So $\omega \geq_q \omega ^{\prime}$ if $Z_t^{-1} - Z_t^{\prime -1}$ is
completely positive.  We have
\begin{align*}
Z_t^{-1} - Z_t^{\prime -1} &= Z_t^{-1}(Z_t^{\prime} - Z_t)Z_t^{\prime-1}
\\
&= w_t^{-1}Z_t^{-1}(\psi ^{\prime} - \psi )Z_t^{\prime -1}
\end{align*}
and since $Z_t^{-1},Z_t^{\prime -1}$ and $\psi ^{\prime} - \psi$ are
completely positive it follows that $Z_t^{-1} - Z_t^{\prime -1}$ is
completely positive and $\omega \geq_q \omega ^{\prime}$.

Now suppose $\omega \geq_q \omega ^{\prime}$ and $\pi_t^\#$ and
$\pi_t^{\prime \# }$ are the generalized boundary representations of
$\omega$ and $\omega ^{\prime}$, respectively.  Then $w_t\pi _t^\#\Lambda
\geq w_t\pi_t^{\prime \#}\Lambda$ for all $t > 0$.  Then we have
\begin{align*}
w_t(\pi_t^\#\Lambda - \pi_t^{\prime \#}\Lambda) &= w_t(Z_t^{-1} -
Z_t^{\prime -1}) (\iota + w_t^{-1}(\zeta_t + \zeta_t^* + R_t))
\\
w_t(\pi_t^\#\Lambda - \pi_t^{\prime \#}\Lambda) &= Z_t^{-1}(\psi ^{\prime}
- \psi )Z_t^{\prime -1}(\iota + w_t^{-1}(\zeta_t + \zeta_t^* + R_t))
\end{align*}
and taking the limit as $t \rightarrow 0+$ we have
$$
Z_t \rightarrow \iota\qquad Z_t^{\prime} \rightarrow \iota\qquad \text{and
} \qquad (\iota + w_t^{-1}(\zeta_t + \zeta_t^* + R_t)) \rightarrow \iota
$$
and, hence,
$$
w_t(\pi_t^\#\Lambda - \pi_t^{\prime \#}\Lambda) \rightarrow \psi ^{\prime}
- \psi
$$
so $\psi ^{\prime} - \psi$ is the limit of completely positive maps so
$\psi ^{\prime} \geq \psi$.

Now we prove the first part of the theorem.  Assume then that $\omega$ is a
$q$-weight map over $\mathbb C^p$ of index zero and the range of $\omega$ is
$B(\mathbb C^p)$.  Then by Theorem 4.4 we can write $\omega = \psi^{-1}
\vartheta$ where $\psi$ is conditionally negative and invertible and its
inverse is completely positive and $\vartheta \in B(\mathbb C^p) \otimes
\mathfrak A (\mathbb C^p)_*$ is completely positive and
$$
\vartheta \vert_t \Lambda = v_t(\iota + \nu_t)
$$
where $1/v_t \rightarrow 0+$ and $\Vert\nu_t\Vert \rightarrow 0$ as $t
\rightarrow 0+$.  (Note $v_t$ and $w_t$ can be slightly different.)  From
Lemma 4.5 we know $\vartheta$ can be expressed in the form given in the
statement of the theorem except that we can not conclude at this point that
$\rho$ is bounded.  We do know that
$$
\vartheta \vert_t \Lambda = w_t\iota + \zeta_t + \zeta_t^* + R_t
$$
and, therefore,
$$
\omega \vert_t \Lambda = \phi_t = \psi^{-1}(w_t\iota + \zeta_t + \zeta _t^*
+ R_t)
$$
and, hence,
$$
(\iota + \phi_s)^{-1}(\iota + \phi_t) = (\iota + \phi_s)^{-1}\psi^{-1}(\psi
+ w_t\iota + \zeta_t + \zeta_t^* + R_t)
$$
and for $0 < s < t$ by Theorem 4.1 the above expression is conditionally
negative for all $0 < s < t$ so we have
$$
Q(s,t) = tr(I + \phi_s(I))(\iota + \phi_s)^{-1}\psi^{-1}(\psi + w_t \iota +
\zeta_t + \zeta_t^* + R_t)
$$
is conditionally negative for $0 < s < t$.  Now as $s \rightarrow 0+$ by
Theorem 4.4 we have
$$
tr(I + \phi_s(I))(\iota + \phi_s)^{-1} \rightarrow \psi
$$
as $s \rightarrow 0+$.  Since $Q(s,t)$ is conditionally negative we have
taking the limit as $s \rightarrow 0+$ that
$$
Q(0,t) = (\psi + w_t\iota + \zeta_t + \zeta_t^* + R_t)
$$ is conditionally negative for all $t > 0$ and since $w_s\iota + \zeta_t +
\zeta _t^*$ is conditionally zero we have $\psi + R_t$ is conditionally
negative for all $t > 0$.  Now $R_t$ is non increasing in $t$ and we know
that
$$
R_t(A) = \sum_{i=1}^m \lambda_i(t)X_i(t)AX_i^*(t)
$$
where $\lambda_i > 0$ and $tr(X_i(t)) = 0$ and $tr(X_i^*(t)X_j(t)) =\delta_{ij}$
and since $\psi$ is conditionally negative we know from the previous
section that
$$
\psi (A) = rA + WA + AW^* - \sum_{i=1}^q \sigma_iS_iAS_i^*
$$
for $A \in B(\mathbb C^p)$ where $W$ and the $S_i$ are of trace zero and
$tr(S_i^*S_j) = \delta_{ij}$ and $\sigma _i > 0$ for $i,j = 1,\cdots ,q$.
Since $\psi + R_t$ is conditionally negative we have $\lambda_i(t) \leq
\sigma_{\max}$ where $\sigma_{\max}$ is the maximum of the $\sigma
_i^{\prime}s$.  Hence,
$$
\Vert R_t\Vert_{H.S.}^2 = \sum_{i=1}^m
\lambda_i^2(t)  \leq  m\sigma_{\max} ^2
$$
and
$$ \Vert R_t\Vert \leq p\Vert R_t\Vert_{H.S.}\leq p\sqrt{m} \sigma_{\max}
\leq p^2\sigma_{\max}
$$
so $R_t = \rho \vert_t (\Lambda )$ is uniformly bounded and since $\rho
\vert_t$  is non increasing we have $\rho \vert_t \Lambda $ converges to a
bounded $b$-weight map $\rho \Lambda $ as $t \rightarrow 0+$ and $\psi +
\rho \Lambda $ is conditionally negative.  Finally, we show $\rho$ is
bounded.  Note from Lemma 4.5 we have $tr(\rho (I-\Lambda )) \leq
tr(\vartheta (I-\Lambda ))$ so $\Vert\rho (I-\Lambda )\Vert \leq
p\Vert\vartheta (I-\Lambda )\Vert$ so
\begin{align*}
\Vert\rho\Vert = \Vert\rho (I)\Vert  &= \Vert\rho (I - \Lambda ) + \rho
(\Lambda )\Vert \leq \Vert\rho (I - \Lambda )\Vert + \Vert\rho (\Lambda )\Vert
\\
&\leq p\Vert\vartheta (I-\Lambda )\Vert + \Vert\rho (\Lambda )\Vert .
\end{align*}

Then all that remains is to show that $\psi (I) \geq \vartheta (I - \Lambda
)$.  From our previous computations we have
$$
I - \pi_t^\# (I) = Z_t^{-1}w_t^{-1}(\psi (I)-\vartheta \vert_t (I-\Lambda
))
$$
and $Z_t^{-1} \rightarrow \iota$ and $\vartheta \vert_t (I - \Lambda )
\rightarrow \vartheta (I - \Lambda )$ as $t \rightarrow 0+$ we have
$$
w_t(I - \pi_t^\# (I)) \rightarrow \psi (I) - \vartheta (I - \Lambda )
$$
and since the expression on the left is positive we have
$\psi (I) \geq \vartheta (I - \Lambda ). $ \end{proof}

Given a $q$-weight map of the form $\omega = \psi^{-1}\vartheta$ as describe
above we will call $\psi$ a coefficient map of $\omega$ and $\vartheta$ the
corresponding limiting $b$-weight map.  As mentioned before $\psi$ and
$\vartheta$ are unique up to multiplying each by a positive number $s$ so if
$\omega = \psi^{-1}\vartheta = \psi ^{\prime -1}\vartheta ^{\prime}$ then $\psi
^{\prime} = s\psi$ and $\vartheta ^{\prime} = s\vartheta$ with $s > 0$.

We are particularly interested is the case when $\omega$ is $q$-pure.  Note
from the previous theorem we can easily construct $q$-subordinates of a
$q$-weight map $\omega = \psi^{-1}\vartheta$ by simply finding $\psi ^{\prime}$
so that $\psi ^{\prime} \geq \psi$ and $\psi ^{\prime} + \rho
\Lambda $ is conditionally negative.  Note if the $\psi ^{\prime}$
satisfying these conditions are not totally ordered then $\omega$ will not be
$q$-pure.  From the discussion in the last section we see that the $\psi
^{\prime}$ satisfying these conditions are totally ordered if and only if
$\psi + \rho \Lambda$ is conditionally zero in which case the
$\psi ^{\prime}$ such that $\psi ^{\prime} + \rho \Lambda$ is
conditionally negative and $\psi ^{\prime} \geq \psi$ consists of those
$\psi ^{\prime}$ of the form $\psi ^{\prime} = \psi + s\iota$ with $s \geq
0$.  If $\omega = \psi^{-1}\vartheta$ is of this form so $\psi + \rho
 \Lambda$ is conditionally zero we say $\omega$ has a pure
skeleton.  Then $\omega = \psi^{-1}\vartheta$ has a pure skeleton if and
only if its trivial subordinates are totally ordered.

We summarize these observations in a definition.

\begin{defn}
If $\omega$ is a $q$-weight map over $\mathbb C^p$ of
index zero with range $B(\mathbb C^p)$ and $\omega = \psi^{-1}\vartheta$ where
$\psi$ and $\vartheta$ were defined in Theorem 4.7.  We call the pair
$(\psi ,\vartheta )$ a coefficient map of $\omega$ and $\vartheta$ the
corresponding limiting $b$-weight map of $\omega$.  The bounded map $\rho$ of
Theorem 4.7 is constructed from the $b$-weight map $\vartheta$ and we denote
this by writing $\rho _\vartheta$.  If $\omega = \psi^{-1}\vartheta$ as just
described and $\psi ^{\prime} \geq \rho _\vartheta \Lambda $ and $\psi
^{\prime} \geq \psi$ we call $\omega ^{\prime} = \psi ^{\prime
-1}\vartheta$ a trivial subordinate of $\omega $.  The only trivial
subordinates are of the form $\omega ^{\prime} = \psi ^{\prime -1}\omega$
with $\psi ^{\prime} = \psi + s\iota$ with $s \geq 0$ if and only if $\psi
+ \rho_\vartheta \Lambda $ is conditionally zero.  In this case we say
$\omega$ has a pure skeleton.
\end{defn}

Note that if $\omega$ is a $q$-weight map over $\mathbb C^p$ with index zero and
range $B(\mathbb C^p)$ and $\omega = \psi^{-1}\vartheta$ with $\psi$ a
coefficient map of $\omega$ and $\vartheta$ the corresponding limiting
$b$-weight map then from Lemma 3.2 there is a map $\psi ^{\prime}$ so that
$\psi ^{\prime} + \rho_\vartheta \Lambda $ is conditionally zero and $\psi
^{\prime}(I) = T$ where $T$ is any hermitian element of $B(\mathbb C^p)$ we
choose.  Then if we choose $T$ to be any positive operator with $T \geq
\vartheta (I - \Lambda )$ then $\omega ^{\prime} = \psi ^{\prime
-1}\vartheta$ has a pure skeleton.  If we choose $T = \vartheta (I -
\Lambda )$ then $\omega ^{\prime}$ is unital.  Simply put given an index
zero $q$-weight map $\omega$ over $\mathbb C^p$ with range $B(\mathbb C ^p)$
then one can easily modify it to produce a $q$-weight map with a pure skeleton.

The next theorem shows how to find all subordinates in the case at hand.

\begin{thm}
Suppose $\omega$ is a $q$-weight map over $\mathbb C^p$
of index zero and range $B(\mathbb C^p)$ and $\omega = \psi^{-1}\vartheta$
with $\psi$ a coefficient map of $\omega$ and $\vartheta$ the corresponding
limiting $b$-weight map.  Suppose there is a bounded completely positive
$B(\mathbb C^p)$  valued weight map $\eta \in B(\mathbb C^p) \otimes
\mathfrak A (\mathbb C^p)_*$ so that $\vartheta \geq \eta$.  Then if
$$
\psi ^{\prime} \geq \psi + \eta \Lambda \qquad \text{and} \qquad \psi
^{\prime} + \rho_\vartheta \Lambda  - \eta \Lambda
$$
is conditionally negative then $\omega ^{\prime} = \psi ^{\prime
-1}(\vartheta - \eta )$ is a $q$-subordinate of $\omega$.

Conversely, if $\omega ^{\prime}$ is $q$-subordinate of $\omega$ there is a
bounded completely positive $B(\mathbb C^p)$ valued $b$-weight map $\eta \in
B(\mathbb C^p) \otimes \mathfrak A (\mathbb C^p)_*$  with $\vartheta \geq \eta$
and $\psi ^{\prime}$ satisfying the conditions above so that $\omega ^{\prime}
= \psi ^{\prime -1}(\vartheta - \eta )$.
\end{thm}
\begin{proof}  Assume the hypothesis of first paragraph of the theorem.  We
assume the notation of Theorem 4.7 so that
$$
\vartheta \vert_t \Lambda = w_t\iota + \zeta_t + \zeta_t^* + R_t
$$
for $t > 0$.  Let $\vartheta ^{\prime} = \vartheta - \eta$ and let $T_t =
\eta \vert_t \Lambda$ so
$$
\vartheta ^{\prime} \vert_t \Lambda = w_t\iota + \zeta_t + \zeta_t^* + R_t
- T_t
$$
and let $\psi ^{\prime}$ be such that $\psi ^{\prime} \geq \psi + \eta
\Lambda  = \psi + T_o$ and $\psi ^{\prime} + \rho_\vartheta \Lambda  -
\eta \Lambda $ is conditionally negative.  We show $\psi ^{\prime} +
\rho_{\vartheta} \vert_t \Lambda  - \eta \vert_t \Lambda $ is
conditionally negative for $t > 0$.  To see this note that in constructing
$\rho_\vartheta$ from $\vartheta$ from the definition of the $g's$ and $h's$
we find that $\rho_\vartheta \vert_t \Lambda$ is the internal part of
$\vartheta \vert_t \Lambda $ where the internal part of a mapping of
$\mathbb C^p$ into itself was defined in equation 3.4 of the last section.
The reason we use the cut off at $t > 0$ is because $\vartheta$ is unbounded.
To be more specific for the case at hand suppose $t > 0$ and let
$L_t = \vartheta\vert_t\Lambda ,\medspace L _t^{\prime} =
\vartheta^{\prime}\vert_t\Lambda = (\vartheta - \eta )\vert_t \Lambda$ and
$L_t^{\prime\prime} = \eta\vert_t\Lambda$ and let $K_t,\medspace K_t^{\prime }$
and $K_t^{\prime\prime}$ be the internal parts of $L_t,\medspace L_t^{\prime }$
and $L_t^{\prime\prime}$, respectively, as defined in equation 3.4 of the last
section.  Note that $K_t = \rho_\vartheta\vert_t \Lambda,\medspace K_t^{\prime}
= \rho _{\vartheta^{\prime}}\vert_t \Lambda $ and $K_t- K_t^{\prime\prime} =
K_t^{\prime}.  $ Note that since $\rho_\vartheta ,\medspace
\rho_{\vartheta^{\prime}}$ and $\eta$ are bounded $K_t,K _t^{\prime}$ and
$K_t^{\prime\prime}$ converge in norm as $t \rightarrow 0$.  Note that
$K_t^{\prime}$ is non increasing in $t$ we have so $K_o^{\prime} \geq
K_t^{\prime}$ for $t > 0$.  Since the difference between a map and its internal
part is conditionally zero we have $\psi^{\prime} +
\rho_\vartheta\vert_t\Lambda - \eta\vert_t\Lambda$ is conditionally negative if
and only if $\psi^{\prime} + K_t^{\prime}$ is conditionally negative and since
we are given that $\psi^{\prime}+K_t^{\prime}$ is conditionally negative at
$t = 0$ we have $\psi^{\prime} + \rho_\vartheta \vert_t\Lambda -
\eta\vert_t\Lambda$ is conditionally negative for $t > 0$.

We have
$$
\psi ^{\prime}(I) \geq \psi (I) + \eta (\Lambda ) \geq \vartheta (I-\Lambda
) + \eta (\Lambda ) = \vartheta ^{\prime}(I-\Lambda ) + \eta (I)
\geq \vartheta ^{\prime}(I-\Lambda )
$$
so by Theorem 4.7 we have $\omega ^{\prime} = \psi ^{\prime -1}
\vartheta ^{\prime}$ is a $q$-weight map over $\mathbb C^p$.  Let $\pi_t^\#$ and
$\pi_t^{\prime\#}$ be the generalized boundary representations of $\omega$
and $\omega ^{\prime}$, respectively.  Then we have
$$
\pi_t^\# = (\iota + w_t^{-1}(\zeta_t + \zeta_t^* + R_t + \psi))
^{-1}w_t^{-1 }\vartheta \vert_t  = w_t^{-1}Z_t^{-1}\vartheta \vert_t
$$
and
$$
\pi_t^{\prime\#} = (\iota + w_t^{-1}(\zeta_t+\zeta_t^*+R_t-T_t+\psi
^{\prime})) ^{-1}w_t^{-1}(\vartheta-\eta ) \vert_t  = w_t^{-1}Z_t^{\prime
-1}( \vartheta-\eta ) \vert_t
$$
so
\begin{align*}
\pi_t^\# - \pi_t^{\prime\#} &= w_t^{-1}(Z_t^{-1}\vartheta \vert_t -
Z_t^{\prime -1}(\vartheta-\eta ) \vert_t )
\\
&= w_t^{-1}Z_t^{-1}(Z_t^{\prime}Z_t^{\prime -1}\vartheta \vert_t -
Z_tZ_t^{\prime -1}(\vartheta-\eta ) \vert_t )
\\
&= w_t^{-1}(Z_t^{-1}(Z_t^{\prime}-Z_t)Z_t^{\prime-1}\vartheta \vert_t +
Z_t^{\prime-1}\eta \vert_t )
\\
&= w_t^{-1}(Z_t^{-1}(\psi ^{\prime}-T_t-\psi )Z_t^{\prime-1}\vartheta
\vert_t  + Z_t^{\prime -1}\eta \vert_t ).
\end{align*}
Since $\psi ^{\prime} \geq \psi + \eta \Lambda  \geq \psi + \eta \vert_t
\Lambda  = \psi + T_t$ we have $\psi ^{\prime} - T_t - \psi \geq 0$.  Note
$Z_t^{-1}$ and $Z_t^{\prime -1}$ are completely positive from Lemma 3.3 and
since the other terms $\vartheta\vert_t$ and $\eta \vert_t$  are completely
positive we have $\pi_t^\# \geq \pi_t^{\prime\#}$ for $t > 0$ so
$\omega ^{\prime}$ is a $q$-subordinate of $\omega$.

Now suppose $\omega$ is a $q$-weight map over $\mathbb C^p$ of index zero and
the range of $\omega$ is $B(\mathbb C^p)$ and $\omega ^{\prime}$ is a non zero
$q$-subordinate of $\omega$.  Then from Theorem 4.3 we know that $\omega
^{\prime}$ is of index zero.  Suppose $\pi _t^\#$ and $\pi_t^{\prime\#}$
are the generalized boundary representations of $\omega$ and $\omega
^{\prime}$, respectively.  Since $\omega \geq_q \omega ^{\prime}$ we have
$\pi_t^\# \geq \pi_t^{\prime \#}$ for $t > 0$ and we know $\pi_t^\#\Lambda
\rightarrow \iota$ and $\pi_t^{\prime \#}\Lambda (A) \rightarrow A$ for all $A
\in Range(\omega ^{\prime})$.   Let $L^{\prime}$ be a limit point of
the $\pi_t^{\prime\#}\Lambda$ as $t \rightarrow 0+$.  Then we have $\iota
\geq L^{\prime} \geq 0$ and since $\iota$ is pure as a completely positive
map we have $L^{\prime} = s\iota$ for $s \in [0,1]$.  Since $L^{\prime}(A)
=$ A for $A$ in the range of $\omega ^{\prime}$ and since $\omega
^{\prime}$ is not zero there are non zero $A$ in the range of $\omega
^{\prime}$.  Hence $L^{\prime} = \iota$ and since this is true for any
limit point we have $\pi_t^{\prime\#}\Lambda \rightarrow \iota$ as $t
\rightarrow 0+$.  Hence, $\omega ^{\prime}$ has range $B(\mathbb C^p)$ so we
know that $\omega ^{\prime} = \psi ^{\prime -1} \vartheta ^{\prime}$ where
$\psi ^{\prime}$ is a coefficient map for $\omega ^{\prime}$ and $\vartheta
^{\prime}$ is the corresponding limiting $b$-weight map of $\omega ^{\prime}$.

In writing $\omega = \psi^{-1}\vartheta$ the maps $\psi$ and $\vartheta$
are not unique in that we can multiply both of them by the same positive
number and $\omega$ is unchanged.  Although the maps are not unique we can
make them unique by requiring $tr(\psi^{-1}(I)) = 1$.  We will assume that
$\psi$ has been normalized so that $tr(\psi^{-1}(I)) = 1$ and prove the
theorem under this additional hypothesis.  Then at the end of the proof we
will show that the truth of theorem under this hypothesis implies the
theorem is true in general.

Then with this hypothesis on $\psi$ we see from the argument in Theorem 4.3
that
$$
\psi_t = tr(I + \phi_t(I))(\iota + \phi_t)^{-1} \rightarrow \psi.
$$
Now for the $q$-weight map $\omega ^{\prime}$ we express this in the from
$\omega ^{\prime} = \psi ^{\prime\prime -1}\vartheta ^{\prime\prime}$ where
we require $tr(\psi ^{\prime\prime -1}(I)) = 1$.  The reason we use the
double prime is because $\psi ^{\prime\prime}$ will turn out to be a
multiple of the $\psi ^{\prime}$ in the statement of the theorem.  Then in
terms of $\psi ^{\prime\prime}$ we have
$$
\psi_t^{\prime\prime} = tr(I + \phi_t^{\prime}(I))(\iota + \phi_t^{\prime})
^{-1} \rightarrow \psi ^{\prime\prime}.
$$
Since $\omega \geq_q \omega ^{\prime}$ we have $\pi_t^\# \geq
\pi_t^{\prime\#}$ so
$$
tr(I + \phi_t(I))(\pi_t^\# - \pi_t^{\prime\#} ) = \psi_t\omega \vert_t -
\frac {tr(I+\phi_t (I))} {tr(I+\phi_t ^{\prime}(I))}
\psi_t^{\prime\prime}\omega^{\prime} \vert_t  \geq 0
$$
for $t > 0$.  Since for some number $r$ it is not true that $\psi\omega
\geq r\psi ^{\prime\prime}\omega ^{\prime}$ it follows that
$$
\kappa = \lim \sup_{t\rightarrow 0+} \frac {tr(I+\phi_t(I))} {tr(I+\phi_t
^{\prime}(I)) }  \leq r.
$$
Now let $t_k > 0$ be a decreasing sequence converging to zero and
$$
\lim_{k\rightarrow\infty} \frac {tr(I+\phi_{t_k}(I))} {tr(I+\phi_{t_k}
^{\prime} (I))}  = \kappa .
$$
Then we find
$$
\lim_{k\rightarrow\infty}  \psi_{t_k}\omega \vert_{t_k}  - \frac {tr(I +
\phi_{t_k}(I))} {tr(I+\phi_{t_k} ^{\prime}(I))} \psi_{t_k}^{\prime\prime}
\omega^{\prime}\vert_{t_k} = \vartheta - \kappa\vartheta ^{\prime\prime}
= \eta \geq 0.
$$
Where the last equality is the definition of $\eta$.

Now we compute
$$
tr(I+\phi_{t_k}(I))(\pi_{t_k}^\#\Lambda - \pi_{t_k}^{\prime\#}\Lambda ) =
\frac {tr(I+\phi_{t_k}(I))} {tr(I+\phi_{t_k} ^{\prime}(I))}
\psi_{t_k}^{\prime\prime} - \psi_{t_k} \rightarrow \kappa\psi
^{\prime\prime} - \psi.
$$
Now for any $s > 0$ we have
$$
tr(I+\phi_{t_k}(I))(\pi_{t_k}^\#- \pi_{t_k}^{\prime\#} )\Lambda \geq tr(I+
\phi_{t_k}(I))(\pi_{t_k}^\#- \pi_{t_k}^{\prime\#} )\vert_s \Lambda
$$
and as $k \rightarrow \infty$ the right hand side of the above inequality
converges to $(\vartheta-\kappa\vartheta ^{\prime\prime}) \vert_s \Lambda$
and the left hand side converges to $\kappa\psi ^{\prime\prime} - \psi$.
Hence, we have
$$
\kappa\psi ^{\prime\prime} - \psi \geq (\vartheta - \kappa\vartheta
^{\prime\prime})  \vert_s \Lambda = \eta \vert_s \Lambda
$$
for $s > 0$.  Hence, $\eta \vert_s \Lambda$ is uniformly bounded and since
$$
\eta \vert_t (I) = \eta \vert_t (\Lambda ) + \eta \vert_t (I - \Lambda )
\leq \kappa\psi ^{\prime\prime}(I) - \psi (I) + \vartheta (I - \Lambda )
$$
we have $\eta \vert_t (I)$ is uniformly bounded and, hence, $\eta$ is
bounded and $\eta \Lambda  \leq \kappa\psi ^{\prime\prime} - \psi$ and
$\vartheta ^{\prime\prime} = (1/\kappa )(\vartheta - \eta )$.  Now we
define $\psi ^{\prime} = \kappa\psi ^{\prime\prime}$ and $\vartheta
^{\prime} = \kappa\vartheta ^{\prime\prime}$ and in terms of $\psi
^{\prime}$ we have
$$
\omega ^{\prime} = \psi ^{\prime -1}(\vartheta - \eta )\qquad \text{and }
\qquad \psi ^{\prime} \geq \psi + \eta \Lambda.
$$

So all that remains is to show that $\psi ^{\prime} + \rho_\vartheta
\Lambda  - \eta \Lambda $ is conditionally negative.
Since $\omega ^{\prime}$ is a $q$-weight map we have from Theorem 4.7 that
$\psi ^{\prime} + \rho_{\vartheta^{\prime}}\Lambda $ is conditionally
negative.  But as we saw earlier in the first paragraphs of this proof
$\rho_{\vartheta}\Lambda - \rho_{\vartheta^{\prime}}\Lambda - \eta \Lambda$
is conditionally zero which yields the desired result.

Now we have proved the theorem with this particular normalization of $\psi$
and $\vartheta$.  What about other normalizations?  If you examine the
statement of the theorem and make the following replacement, $\psi
\rightarrow s\psi ,\medspace \vartheta  \rightarrow s\vartheta ,\medspace
\psi ^{\prime}\rightarrow s\psi ^{\prime},\medspace \eta \rightarrow s\eta$
and note with these replacements one has $\rho_\vartheta \Lambda
\rightarrow s\rho_\vartheta \Lambda $ you see all the statements in the
theorem are replaced by equivalent statements.  Therefore, it follows that
since the theorem is true for a specific normalization for $\psi$ and
$\vartheta$ it follows that the theorem is true in general.   \end{proof}

\begin{thm}
Suppose $\omega$ is a $q$-weight map over $\mathbb C^p$
of index zero and range $B(\mathbb C^p)$ and $\omega = \psi^{-1}\vartheta$
with $\psi$ a coefficient map of $\omega$ and $\vartheta$ the corresponding
limiting $b$-weight map.  Then $\omega$ is $q$-pure if and only if $\psi +
\rho_\vartheta \Lambda $ is conditionally zero and $\vartheta$ is
strictly infinite meaning that if $\eta$ is a completely positive bounded
$B(\mathbb C^p)$ valued normal weight map on $\mathfrak A (\mathbb C^p)$ and
$\eta$ is a subordinate of $\vartheta$ so $\vartheta \geq \eta$ then $\eta = 0$.
\end{thm}
\begin{proof}  Assume the hypothesis of the theorem.  From Theorem 4.7 we
know that if $\psi ^{\prime}$ is a hermitian map of $B(\mathbb C^p)$ into
itself and
$$
\psi ^{\prime} \geq \psi\qquad \text{and} \qquad \psi ^{\prime} +
\rho_\vartheta \Lambda
$$
is conditionally negative then $\omega ^{\prime} = \psi ^{\prime
-1}\vartheta$ is a $q$-subordinate of $\omega$ and if $\psi
^{\prime\prime}$ is a second such map satisfying these conditions and
$\omega ^{\prime\prime} = \psi ^{\prime\prime -1}\vartheta$ then $\omega
^{\prime} \geq_q \omega ^{\prime\prime}$ if and only if $\psi ^{\prime}
\leq \psi ^{\prime\prime}$.  Let $S$ be the set of $\psi ^{\prime}$ satisfying
the conditions above.  It follows that if $\omega$ is $q$-pure then $S$ must be
totally ordered.  Note if $\psi ^{\prime} = \psi + s\iota$ with $s \geq 0$ then
$\psi ^{\prime} \in S$ so if $S$ is totally ordered then $S$ simply consists of
$\psi ^{\prime} = \psi + s\iota$ with $s \geq 0$ and this is the case if
and only if $\psi + \rho_\vartheta \Lambda $ is conditionally zero.
Hence, if $\omega$ is $q$-pure we have $\psi + \rho _\vartheta \Lambda $
is conditionally zero.

Next suppose there is a bounded subordinate $\eta$ of $\vartheta$ and $\eta
\neq 0$.  Let $\psi ^{\prime} = \psi + \eta \Lambda$.  Then
$$
\psi ^{\prime} + \rho_\vartheta \Lambda  - \eta \Lambda = \psi +
\rho_\vartheta \Lambda
$$
is conditionally negative and by Theorem 4.9 it follows that $\omega
^{\prime} = \psi ^{\prime -1}(\vartheta - \eta )$ is a $q$-subordinate of
$\omega$.  To show that $\omega$ is not $q$-pure let $\psi ^{\prime\prime}
= \psi + s\iota$ where $s$ is a positive number of our choosing and let
$\pi_t^{\prime\#}$ and $\pi_t^{\prime\prime\#}$ be the generalized boundary
representation of $\omega ^{\prime}$ and $\omega ^{\prime\prime}$,
respectively.  Then recalling the computations of the previous theorem we
have
$$
\pi_t^{\prime\prime\#} = (\iota + w_t^{-1}(\zeta_t + \zeta_t^* + R_t + \psi
+ s\iota ))^{-1}w_t^{-1}\vartheta \vert_t  = w_t^{-1}Z_t^{\prime\prime
-1}\vartheta \vert_t
$$
and we have
$$
w_t(\pi_t^{\prime\prime\#} - \pi_t^{\prime\#} ) = Z_t^{-1}(T_o-T_t-s\iota
)Z_t^{\prime -1}\vartheta \vert_t  + Z_t^{\prime -1}\eta \vert_t
\rightarrow \eta - s\vartheta .
$$
We recall that $E(r,\infty )$ is the projection in $B(\mathbb C^p \otimes
L^2(0, \infty ))$ on functions with support in $[r,\infty )$.  Since $\eta
\neq 0$ we there is a $r > 0$ so that $\eta (E(r,\infty )) \neq 0$.  Now
choose $s > 0$ so small that $(s\vartheta - \eta )(E(r,\infty ))$ is not
positive.  For this choice of $s$ we have it is not true that $\omega
^{\prime\prime} \geq_q \omega ^{\prime}$ since $\eta - s\vartheta$ can not
be positive since $\eta$ is bounded and $\vartheta$ is unbounded.  It is
also not true that $\omega ^{\prime} \geq_q \omega ^{\prime\prime}$ since
$(s\vartheta - \eta )E(r,\infty )$ is not positive.  Hence, there are two
subordinates of $\omega$ that are not ordered so $\omega$ is not $q$-pure.

Hence, we have shown that if $\omega$ is $q$-pure $\psi + \rho_\vartheta
\Lambda $ is conditionally zero and $\vartheta$ is strictly infinite.
Now suppose $\omega$ satisfies the conditions of the theorem and $\omega
^{\prime}$ is a $q$-subordinate of $\omega$.  By the previous theorem we
have $\omega ^{\prime} = \psi ^{\prime -1}(\vartheta - \eta )$ where $\eta$
is a bounded subordinate of $\vartheta$ and
$$
\psi ^{\prime} \geq \psi + \eta \Lambda \qquad \text{and} \qquad \psi
^{\prime} + \rho_\vartheta \Lambda  - \eta \Lambda
$$
is conditionally negative.  Since $\vartheta$ is given to be strictly
infinite we have $\eta = 0$ so $\psi ^{\prime}$ satisfies the conditions
above with $\eta = 0$.  Since $\psi + \rho_\vartheta \Lambda $ is
conditionally zero we have $\psi ^{\prime} - \psi \geq 0$ and conditionally
negative so $\psi ^{\prime} = \psi + s\iota$ with $s \geq 0$.   Hence, the
$q$-subordinates of $\omega$ are totally ordered.  \end{proof}

In the next lemma we give conditions that $\vartheta$ be strictly infinite.

\begin{thm}
Suppose $\vartheta$ is a $B(\mathbb C^p)$ valued $b$-weight map
on $\mathfrak A (\mathbb C^p)$ of the form
$$
\vartheta_{ij}(A) = \sum_{k\in J} ((g_{ik}+h_{ik}),A(g_{jk}+h_{jk}))
$$
where the $g_{ik},h_{ik} \in \mathbb C^p\otimes L_+^2(0,\infty )$ and
$$
(g_{ik})_j(x) = \delta_{ij}g_k(x)\qquad \text{and} \qquad \sum_{k\in J}
\sum _{i=1}^p h_{ik}(x) = 0
$$
and
$$
\sum_{k\in J} \Vert h_{ik}\Vert^2 < \infty
$$
for $A \in \mathfrak A (\mathbb C^p),\medspace x \geq 0,\medspace i,j \in \{
1,\cdots ,p\}$ and $k \in J$ a countable index.  Let
$$
\mu (A) = \sum_{k\in J} (g_k,Ag_k)
$$
for $A \in \mathfrak A (\mathbb C )$.  Then $\vartheta$ is strictly infinite if
and only if $\mu$ is strictly infinite and the $h^{\prime}s$ are linearly
independent over the $g^{\prime}s$ by which we mean that if $c \in \ell^2(J)$
and
$$
\sum_{k\in J} c_kg_k = 0
$$
then
$$
\sum_{k\in J} c_kh_{ik}  = 0
$$
for each $i = 1,\cdots ,p$.
\end{thm}
\begin{proof}  Assume hypothesis and notation of the theorem.  Now
$\vartheta$ is strictly infinite if and only if for $c \in \ell^2(J)$ so
that
$$
F_i = \sum_{k\in J} c_k(g_{ik}+h_{ik}) \in \mathbb C^p\otimes L^2(0,\infty )
$$
for each $i = 1,\cdots ,p$ then $F_i = 0$ for each $i = 1,\cdots ,p$.
Suppose that $\mu$ is strictly infinite and the $h_{ik}$ are linearly
independent over the $g_k$.  Suppose $c \in \ell^2(J)$ and the $F_i$ above
are in $\mathbb C ^p\otimes L^2(0,\infty )$ for each $i = 1,\cdots ,p$. since
the sum of the $c_kh_{ik}$ is in $\mathbb C^p \otimes L^2(0,\infty )$ it
follows that the sum of the $c_kg_{ik}$ is in $\mathbb C^p\otimes L^2(0,\infty
)$ and, hence,
$$
\sum_{k\in J} c_kg_k \in L^2(0,\infty )
$$
but since $\mu$ is strictly infinite the above sum is zero and since the
$h_{ik}$ are linearly independent over the $g_k$ we have the sum of the
$c_kh _{ik}$ is zero so $F_i = 0$ for $i = 1,\cdots ,p$ and $\vartheta$ is
strictly infinite.

Now suppose $\mu$ is not strictly infinite.  Then there is a $c \in
\ell^2(J)$ so that
$$
g = \sum_{k\in J} c_kg_k  \in L^2(0,\infty )\qquad \text{and} \qquad
\sum_{k \in J} \vert c_k\vert^2 = 1
$$
and $g \neq 0$.  Let
$$
F_i = \sum_{k\in J} c_k(g_{ik}+h_{ik}).
$$
Since the sum of the $c_kh_{ik}$ is in $\mathbb C^p\otimes L^2(0,\infty )$
then $F _i \in \mathbb C^p\otimes L^2(0,\infty )$ for each $i = 1,\cdots ,p$.
Now from the condition on the $h_{ik}$ we have
$$
g(x) = \frac {1} {p} \sum_{i=1}^p (F_i)_i(x)
$$
so the $F_i$ can not all be zero so $\vartheta$ is not strictly infinite.

Now suppose the $h_{ik}$ are not linearly independent over the $g_k$.  Then
there is a $c \in \ell^2(J)$ so that
$$
\sum_{k\in J} c_kg_k = 0
$$
and
$$
\sum_{k\in J} c_kh_{ik}  \neq 0
$$
for some $i \in \{ 1,\cdots ,p\}$.  Then we have
$$
F_i = \sum_{k\in J} c_k(g_{ik} + h_{ik}) = \sum_{k\in J} c_kh_{ik} \in \mathbb
C ^p\otimes L^2(0,\infty )
$$
for each $i \in \{ 1,\cdots ,p\}$ and not all the $F_i$ are zero.  Hence,
$\vartheta$ is not strictly infinite.  \end{proof}

\section{Index zero with general range}

As mentioned in the last section Choi and Effros showed that if $L$ is a
completely positive contractive linear mapping of $B(\mathbb C^p)$ into itself
that is idempotent, (i.e. $L^2 = L)$ then $\mathcal{L}$ the range of $L$ is a
$*$-subalgebra of $B(\mathbb C^p)$ where the product in $\mathcal{L}$ is given
by $A \star B = L(AB)$.  Because of the importance of this results and to
establish notation we present a proof in the special simple case where
$B(H)=B(\mathbb C^p)$ is finite dimensional.  Note the element $I_o = L(I)$ is
the unit of $\mathcal{L}$ since $I_o \star A = A \star I_o = A$ for all
$A \in \mathcal{L}$.

\begin{thm}
Suppose $L$ is a completely positive contractive
linear mapping of $B(\mathbb C^p)$ into itself that is idempotent (i.e. $L^2 =
L)$.  Then there is a unique projection $F$ so that $L(A)=L(FAF)$ and $F$
is the smallest projection so that $L(F) = L(I)= I_o$.  If $\phi$ is the mapping
of $FB(\mathbb C^p)F$ into itself given by $\phi (A) = FL(A)F$ then the range
of $\phi$ is a $*$-algebra so
$$
\phi (\phi (A)\phi (B)) = \phi (A)\phi (B)\qquad \text{and} \qquad \phi (A^
*) = \phi (A)^*
$$
for $A,B \in B(\mathbb C^p)$.  The mapping $\phi$ is faithful in that if $A
\in FB(\mathbb C^p)F$ and $A \geq 0$ then if $\phi (A) = 0$ then $A = 0$.

Every operator in the range of $L$ commutes with $F$ so
$$
L(A) = \phi (A) + (I - F)L(\phi (A))(I - F)
$$
for all $A \in B(\mathbb C^p)$.  Both $L$ and $\phi$ have the Choi-Effros
property that
$$
L(L(A)B) = L(L(A)L(B)) = L(AL(B))
$$
and
$$
\phi (\phi (A)B) = \phi (\phi (A)\phi (B)) = \phi (A\phi (B))
$$
for $A,B \in B(\mathbb C^p)$.

Given the mapping $\phi$ with the properties listed above one can specify
any completely positive contractive map $\psi$ from the range of $\phi$
into $(I-F)B(\mathbb C^p)(I-F)$ and the mapping $L^{\prime}(A) = \phi (A) +
\psi (\phi (A))$ is a completely positive contractive linear mapping of
$B(\mathbb C^p)$ into itself that is idempotent.
\end{thm}
\begin{proof}  Suppose $L$ is a linear completely positive contractive
idempotent of $B(\mathbb C^p).  $ Let $L(I) = I_o$.  Since $L$ is completely
positive and contractive we have $0 \leq I_o \leq I$.  Since $L$ has norm one
we have $\Vert I_o\Vert = 1$ so  $I_o = E_o + B$ where $E_o$ is an hermitian
projection and $E_oB = BE_o = 0$ and $0 \leq B \leq sI$ where $s \in [0,1)$
(i.e. $s$ is strictly less than one).

Since $L$ is completely positive we have
$$
L(A) = \sum_{i=1}^m S_iAS_i^*
$$
for $A \in B(\mathbb C^p)$ where the $S_i$ are linearly independent elements of
$B(\mathbb C^p)$.  Since $L(I_o) = L(L(I)) = L(I) = I_o$ we have $L(I-I_o)=0$ so
$$
L(I-I_o) = L((I-E_o)(I-B)(I-E_o) = \sum_{i=1}^m S_i(I-E_o)(I-B)(I-E_o)S_i^* = 0
$$
and since $(I-E_o)(I-B)(I-E_o) \geq (1-s)(I-E_o)$ it follows that
$S_i(I-E_o) = 0$ which yields
$$
S_i = S_iE_o\qquad \text{and} \qquad S_i^* = E_oS_i^*
$$
for $i = 1,\cdots ,m$.  Recall that $I_o = E_o + B$ so
$0 \leq E_o \leq I_o \leq I$ and we have
$$
I_o = L(I) = \sum_{i=1}^m S_iS_i^*  = \sum_{i=1}^m S_iE_oS_i^* = L(E_o).
$$
Now let $\mathfrak M$ be the linear span of the ranges of the $S_i^*$ for $i =
1,\cdots ,m$ and let $F$ be the hermitian projection of $\mathbb C^p$ onto
$\mathfrak M$.  Note $F$ is the smallest projection so that $FS_i^* = S_i^*$
for $i = 1,\cdots ,m$.  Since for $i = 1,\cdots ,m$ we have $E_oS_i^* =
S_i^*$ it follows that $E_o \geq F$.  Note we have
$$
L(A) = \sum_{i=1}^m S_iAS_i^*  = \sum_{i=1}^m S_iFAFS_i^*  = L(FAF)
$$
for all $A \in B(\mathbb C^p)$.  Note that if $P$ is an hermitian projection
and $F \geq P \geq 0$ and $L(P) = 0$ then
$$
\sum_{i=1}^m S_iPPS_i^*  = 0
$$
so for $i = 1,\cdots ,m$ we have $PS_i^* = 0$ so $(F-P)S_i^* = FS_i^* = S_i^*$
so $F - P \geq F$ which yields $P = 0$.  Since any positive element of
$FB(\mathbb C^p)F$ is the sum projections in $FB(\mathbb C^p)F$ with positive
coefficients it follows that if $A \in FB(\mathbb C^p)F$ and $A \geq 0$ and
$L(A) = 0$ then $A = 0$.  It follows that $F$ is the smallest projection so
that $L(F) = L(I) = I_o$.

Now let $\phi (A) = FL(A)F$ for $A \in B(\mathbb C^p)$.  Note $\phi$ is
completely positive and idempotent.  And note that
$$
L(A) = L(L(A)) = L(FL(A)F) = L(\phi (A))
$$
for all $A \in B(\mathbb C^p)$.  This is an important equality that we will
often use in the remainder of this proof.

We show the range of $\phi$ is an algebra.  Since $\phi$ is completely
positive $\phi_2 = \iota_2 \otimes \phi$ is positive so we have for
$A \in B(\mathbb C^p)$ that
$$
\phi_2 ( \left[\begin{matrix} F&\phi (A)
\\
\phi (A)^*&\phi (A)^*\phi (A)
\end{matrix} \right])=
\left[\begin{matrix} F&\phi (A)
\\
\phi (A)^*&\phi (\phi (A)^*\phi (A))
\end{matrix} \right]  \geq 0.
$$
So we have
$$
B = \phi (\phi (A)^*\phi (A)) - \phi (A)^*\phi (A) \geq 0.
$$
But $\phi (B) = 0$ and since $FBF = B \geq 0$ we have $B = 0$.  Thus, $\phi
(\phi (A)^*\phi (A)) = \phi (A)^*\phi (A)$ for all $A \in B(\mathbb C^p).  $
It follows from complex linearity that
$$
\phi (\phi (A)\phi (B)) = \phi (A)\phi (B)
$$
for all $A \in B(\mathbb C^p)$.  Hence, the range of $\phi$ is a
$*$-subalgebra in $FB(\mathbb C^p)F$ with unit $F$.  Suppose $P$ is a
hermitian projection in the range of $\phi$.  Then we have
$$
P = \phi (P) = \sum_{i=1}^m FS_iPS_i^*F
$$
and
$$
(I - P)P(I - P) = \sum_{i=1}^m ((I-P)FS_iP)((I-P)FS_iP)^* = 0
$$
and so we conclude $(I-P)FS_iP = 0$ for $i = 1,\cdots ,m$.  Applying the
same argument to $F - P$ we find $(I-F+P)FS_i(F-P) = PFS_i(I-P) = 0$ for $i
= 1,\cdots ,m$.  Then from these two relations we find $FS_i$ maps the
range of $P$ into itself and the range of $I - P$ into itself so $FS_i$
commutes with $P$ for $i = 1,\cdots ,m$.  Since the range of $\phi$ is a
finite dimensional $C^*$-algebra each $A$ in the range of $\phi$ can be
written as a complex linear combinations of projections in the range of
$\phi$ and since each of those projections commute with the $FS_i$ for $i =
1,\cdots ,m$ we see that $\phi (A)$ commutes with the $FS_i$ and since the
range of $\phi$ is a $*$-algebra we have $\phi (A)$ commutes with both
$FS_i$ and $S_i^*F$ for each $i = 1,\cdots ,m$.  Hence, we see that
\begin{align*}
\phi (\phi (A)B) &= \phi (F\phi (A)BF) = \sum_{i=1}^m FS_i\phi (A)BS_i^*F
\\
&= \sum_{i=1}^m \phi (A)FS_iBS_i^*F = \phi (A)\phi (B) = \phi (\phi (A)\phi
(B))
\end{align*}
and
\begin{align*}
\phi (A\phi (B)) &= \phi (FA\phi (B)F) = \sum_{i=1}^m FS_iA\phi (B)S_i^*F
\\
&= \sum_{i=1}^m FS_iAS_i^*F\phi (B) = \phi (A)\phi (B) = \phi (\phi (A)\phi
(B))
\end{align*}
and, hence, we have
$$
\phi (\phi (A)B) = \phi (\phi (A)\phi (B)) = \phi (A\phi (B))
$$
for all $A,B
\in B(\mathbb C^p)$.

We prove that $L(A)$ commutes with $F$ for all $A \in B(\mathbb C^p)$.  We
have
\begin{align*}
FL(A) &= FL(\phi (A)) = \sum_{i=1}^m FS_i\phi (A)S_i^* = \sum_{i=1}^m \phi
(A)FS_iS_i^*
\\
&= \phi (A)FL(I) = \phi (A)FI_o = \phi (A)F = \phi (A)
\end{align*}
and
\begin{align*}
L(A)F &= L(\phi (A))F = \sum_{i=1}^m S_i\phi (A)S_i^*F = \sum_{i=1}^m
S_iS_i ^*F\phi (A) = L(I)F\phi (A)
\\
&= I_oF\phi (A) = F\phi (A) = \phi (A).
\end{align*}
and, hence, $FL(A) = \phi (A) = L(A)F$ for all $A \in B(\mathbb C^p)$.  Hence,
we have
\begin{align*}
L(A) &= FL(A)F + (I-F)L(A)F + FL(A)(I-F) + (I-F)L(A)(I-F)
\\
&= \phi (A) + (I-F)L(A)(I-F) = \phi (A) + (I-F)L(\phi (A))(I-F).
\end{align*}

Now we can prove the Choi-Effros identity that for $A,B \in B(\mathbb C^p)$ we
have
\begin{align*}
L(L(A)B) &= L(FL(A)B) = L(\phi (A)B) = L(\phi (\phi (A)B))
\\
&= L(\phi (\phi (A)\phi (B))) = L(\phi (A)\phi (B))
\end{align*}
and
\begin{align*}
L(AL(B)) &= L(AL(B)F) = L(A\phi (B)) = L(\phi (A\phi (B)))
\\
&= L(\phi (\phi (A)\phi (B))) = L(\phi (A)\phi (B))
\end{align*}
and
$$
L(L(A)L(B)) = L(FL(A)L(B)F) = L(\phi (A)\phi (B))
$$
so we have
$$
L(L(A)B) = L(L(A)L(B)) = L(AL(B))
$$
for all $A,B \in B(\mathbb C^p)$.

Now suppose $\psi$ is a completely positive contractive linear mapping of
the range of $\phi$ into $(I-F)B(\mathbb C^p)(I-F)$.   Let
$$
L^{\prime}(A) = \phi (A) + \psi (\phi (A))
$$
for $A \in B(\mathbb C^p)$.  We note $L$ is completely positive and
contractive and
\begin{align*}
L^{\prime}(L^{\prime}(A)) &= L^{\prime}(\phi (A)+\psi (\phi (A))) = \phi
(\phi (A))+\psi (\phi (\phi (A)))
\\
&= \phi (A) + \psi (\phi (A)) = L^{\prime}(A)
\end{align*}
for $A \in B(\mathbb C^p).$ \end{proof}

Our situation is further complicated by the fact that the limit $L$ is not
necessarily unique.  In fact, by experimenting you can find examples where
there are two limits $L$ and $L^{\prime}$ with disjoint support projections
$F$ and $F^{\prime}$.  What is unique is the range of $L$ since this is equal
to the range of $\omega$.  From the work of Jankowski we know that in the case
of $\mathbb C^3$ the Choi-Effros algebra can be isomorphic to the $(2 \times
2) -$ matrices with elements $\{ f_{ij}:i,j = 1,2\}$ given in terms of the
$(3 \times 3)$ matrix units $\{ e_{ij}:i,j = 1,2,3\}$ as follows
$$
f_{11} = e_{11}\qquad f_{12} = e_{12}\qquad f_{21} = e_{21}\qquad f_{22} =
e_{22} + \lambda e_{33}
$$
with $\lambda \in [0,1]$.  In this case we can deduce that the support
projection $F$ is $F = e_{11} + e_{22}$.

Since the limit $L$ is not unique we are faced with the following
problem. Suppose there are two completely positive contractive
idempotent maps $L_1$ and $L_2$ of $B(\mathbb C^p)$ into itself
with the same range $\mathcal{L}$.  Then as we have seen $L_1$ and
$L_2$ give us a multiplication on $\mathcal{L} $.  The question is
are they the same.  In particular is
$$
L_1(AB) = L_2(AB)
$$
for $A,B \in \mathcal{L}$.  Choi and Effros in \cite{CE} develop
the theory of operator systems and we believe their results show
the answer is yes.  Unfortunately they do not specifically state
the result we are after. Therefore, we will give a brief outline a
proof in our finite dimensional case and apologize in advance for
giving longer and more complicated argument when they would most
probably give.  In our discussion we will often be showing that
$L_1$ and $L_2$ satisfy certain conditions such as $L_1(EFE) =
\lambda E$ and $L_2(EFE) =\lambda E$ for certain elements $E$ and
$F$.  Rather than write out two derivation we will simply write
out a derivation for $L$.  So we will use the following notation.
When we write $L$ we mean any completely positive contractive
idempotent of $B(\mathbb C^p)$ into $\mathcal{L}$ whose range is
$\mathcal{L}$.  Then any formula involving $L$ applies to both
$L_1$ or $L_2$.  There is one obvious word of caution, namely, if
we have a collection of formulae involving $L$ the formulae are
valid for both $L_1$ and $L_2$ but we can not mix the subscripts 1
and 2 (i.e. in all the formulae in a collection we must either use
all 1's or all 2's).

Since $\mathcal{L}$ is a finite dimensional $C^*$-algebra every
hermitian $A$ is the sum of minimal projections
$$
A = \sum_{i=1}^m \lambda_iE_i
$$
where the $\lambda_i$ are real and the $E_i$ are mutually
orthogonal minimal projections.  Note that such a decomposition
does not require the idempotent $L_1$ or $L_2$ because a minimal
projection $E$ can be characterized as follows.  The element $E
\in \mathcal{L}$ has the property that $E \geq 0,\medspace \Vert
E\Vert = 1$ and if $A \in \mathcal{L}$ and $0 \leq A \leq E$ then
$A = \lambda E$.  One checks that if an element $E \in
\mathcal{L}$ has these properties then $E$ is a minimal
projection.  Note the fact that $L(E^2) = E$ (where $L$ is any
completely positive contractive idempotent with range $\mathcal{L}
)$ is automatically satisfied since $0 \leq E^2 \leq E$ and,
therefore, $0 \leq L(E^2) \leq L(E)$ and, therefore, $L(E) =
\lambda E$.  And since $L$ is completely positive and the matrix
of elements in $B(\mathbb C^p)$
$$
\left[\begin{matrix} I&E
\\
E&E^2
\end{matrix} \right]
\geq 0
$$
the matrix obtained by replacing the above matrix elements of the above matrix
with $L$ applied to that matrix element yields a positive matrix so
$$
\left[\begin{matrix} L(I)&L(E)
\\
L(E)&L(E^2)
\end{matrix} \right] =
\left[\begin{matrix} L(I)&E
\\
E&\lambda E
\end{matrix} \right] \geq 0
$$
and since $I \geq L(I)$ this matrix can only be positive if $\lambda \geq 1$
and since $0 \leq \lambda \leq 1$ we have $\lambda = 1$ and $L(E^2) = E$.

Next if $E$ and $F$ are hermitian projections in $\mathcal{L}$
then $E$ and $F$ are orthogonal if and only if $\Vert E+F\Vert =
1$.  This means that in the expression for $A$ above the fact that
the $E_i$ are mutually orthogonal minimal projections can be
determined without the use of $L_1$ or $L_2$.  Now if $E$ and $F$
are orthogonal minimal hermitian projections in $\mathcal{L}$ then
$L(EFE) = \lambda E$ with $\lambda \geq 0$ so $L(E(E+F)E) =
(1+\lambda)E$ and since $\Vert E+F\Vert = 1$ it follows that
$\lambda = 0$ and since $L(EF(EF)^*) = 0$ it follows that $L(EF) =
L(FE) = 0$. Then for the $E_i$ orthogonal minimal projections the
product $L(E_iE_j) = \delta_{ij}E_i$ for $i,j = 1,\cdots ,m$
(where $\delta_{ij} = 1$ for $i = j$ and $\delta_{ij} = 0$ for $i
\neq j)$.  Then we have
$$
L(A^2)= L(\sum_{i,j=1}^m\lambda_i\lambda_jE_iE_j)=\sum_{i=1}^m\lambda_i^2L(E_i).
$$
Hence we see that $L(A^2)$ is the same element of $\mathcal{L}$
for any completely positive contractive idempotent map $L$ with
range $\mathcal{L}$.  Then, in particular, we have $L_1(A^2) =
L_2(A^2)$. Applying this to the sum of two hermitian $A,B \in
\mathcal{L}$ we see that the Jordan product is the same for $L_1$
and $L_2$, namely, $L_1(AB+BA) = L_2(AB+BA)$ for all hermitian
$A,B \in \mathcal{L} $.

Next we use the fact that $L_1$ and $L_2$ are completely positive.  Let
$H_2 = H \oplus H$ and $\mathcal{L}_2$ be the set of $(2 \times 2)$-matrices
with entries in $\mathcal{L}$ and let $L_{12} = \iota_2 \otimes L_1$ and
$L_{22} = \iota_2 \otimes L_2$ (i.e. we apply $L_1$ and $L_2$ to
$(2 \times 2)$-matrices with entries in $B(\mathbb C^p)$).
Repeating the argument that
showed $L_1(AB+BA) = L_2(AB+BA)$ for all hermitian $A,B \in \mathcal{L}$ for
$\mathcal{L}_2$ we see that $L_{12}(AB+BA) = L_{22}(AB+BA)$ for hermitian
$A,B \in \mathcal{L}_2$.  Now consider product of matrices in $B(H_2)$
$$
X = \left[\begin{matrix} A&B
\\
B&-A
\end{matrix} \right]
\left[\begin{matrix} A&B
\\
B&-A
\end{matrix} \right] =
\left[\begin{matrix} A^2+B^2&AB-BA
\\
BA-AB&A^2+B^2
\end{matrix} \right]
$$
for hermitian $A,B \in \mathcal{L}$.  Since $L_{12}(X) = L_{22}(X)$ we have
that $L_1(AB-BA) = L_2(AB-BA)$ and since $L_1(AB+BA) = L_2(AB+BA)$ we have
$L_1(AB) = L_2(AB)$ for all hermitian $A,B \in \mathcal{L}$ and by complex
linearity we have $L_1(AB) = L_2(AB)$ for all $A,B \in \mathcal{L}$.

In summary we see that if $\omega$ is a non zero $q$-weight map of
index zero and $\pi_t^\#$ is the generalized boundary representation of
$\omega$ and $L$ is a limit point of $L_t = \pi_t^\#\Lambda$ as $t
\rightarrow 0+$ then the range of $\omega$ which we denote by
$\mathcal{L} (\omega )$ is a $C^*$-algebra with multiplication
defined by $A \star B = L(AB)$ for $A,B \in \mathcal{L} (\omega
)$.  Note we have shown the multiplication defined by $L$ does not
depend on which limit point of $L_t$ we take.  Note the unit $I_o$
does not depend on which limit point of $L_t$ we take since the
unit $I_o$ can be characterized without reference to $L$ in that
$I_o$ is the unique positive norm one element of $\mathcal{L}
(\omega )$ with the property that $I_o \geq A$ for all positive
norm one element of $\mathcal{L} (\omega )$.

Armed with these results we find a projection $P$ that dominates all support
projections.

\begin{thm}
Suppose $L$ is a completely positive contractive
linear idempotent mapping of $B(\mathbb C^p)$ into itself and $F$ is the
support projection which is the smallest hermitian projection so that $L(F) =
L(I) = I_o$ and $\phi (A) = FL(A)F$ for $A \in B(\mathbb C^p)$.  Let
$\mathcal{L}$ denote the range of $L$.  Let $P$ be the central support of $F$
in the center of the commutant $\mathcal{L} ^{\prime}$ of $\mathcal{L}$ (i.e.
$P$ is the smallest projection in $\mathcal{L} ^{\prime} \cap \mathcal{L}
^{\prime\prime}$ with $P \geq F$).  Then $\psi (A) = PL(A)P$ is completely
positive contractive linear mapping of $B(\mathbb C^p)$ into itself that is
idempotent and the range of $\psi$ is a $*$-algebra so
$$
\psi (\psi (A)\psi (B)) = \psi (A)\psi (B)\qquad \text{and} \qquad \psi (A^
*) = \psi (A)^*
$$
for $A,B \in B(\mathbb C^p)$.  Furthermore, if $L^{\prime}$ is another
completely positive contractive idempotent map of $B(\mathbb C^p)$ into itself
with range $\mathcal{L}$ then
$$
L^{\prime}(A) = L^{\prime} (PAP)
$$
for $A \in B(\mathbb C^p)$ and if $F^{\prime}$ is the support
projection for $L^{\prime}$ then $F^{\prime} \leq P$ and $P$ is the smallest
projection in $\mathcal{L} ^{\prime} \cap \mathcal{L} ^{\prime\prime}$ with
$P \geq F^{\prime}$.  If $I_o = L(I)$ is the unit of $\mathcal{L}$ then
$PI_o = I_oP = PI_oP = P$.
\end{thm}
\begin{proof}  Assume the hypothesis and notation of the lemma.  Let $\mathfrak
M$ be the linear span of vectors of the form $Af$ with $A \in \mathcal{L}
^{\prime}$ and $f$ in the range of $F$.  Note a vector $f$ is in $\mathfrak M$
if and only if $f$ can be written in the form
$$
f = \sum_{i=1}^m A_iFg_i
$$
with $A_i \in \mathcal{L} ^{\prime}$ and $g_i \in \mathbb C^p$ for $i = 1,\cdots
,m.  $ Note from the previous theorem we have $F \in \mathcal{L}^{\prime}$.
Then if $A \in \mathcal{L}$ we have
$$
Af = \sum_{i=1}^m AA_iFg_i = \sum_{i=1}^m A_iAFg_i = \sum _{i=1}^m A_iFAg_i
$$
so $Af \in \mathfrak M$ and if $B \in \mathcal{L} ^{\prime}$ then since
$BA_i \in \mathcal{L} ^{\prime}$ we have $Bf \in \mathfrak M$.  Hence,
$A\mathfrak M \subset \mathfrak M ,\medspace A^*\mathfrak M \subset
\mathfrak M ,\medspace B \mathfrak M \subset \mathfrak M$ and $B^*\mathfrak
M \subset \mathfrak M$ for $A \in \mathcal{L}$ and $B \in \mathcal{L}
^{\prime}$.  It follows that if $P$ is the orthogonal projection of
$\mathbb C^p$ onto $\mathfrak M$ then $P \in \mathcal{L} ^{\prime}$ and $P
\in \mathcal{L} ^{\prime\prime}$ so $P$ is in the center of $\mathcal{L}
^{\prime}$ which equals the center of $\mathcal{L} ^{\prime\prime}$.  Note
$P$ is the central support of $F$ since $P$ is the smallest projection $P
\geq F$ which is contained in both $\mathcal{L} ^{\prime}$ and $\mathcal{L}
^{\prime\prime}$.

To prove the last statement of the theorem we use another characterization of
$P$ that $P$ is the smallest projection so that $P \geq U^*FU$ for all unitary
$U \in \mathcal{L}^{\prime}$.  Note from this definition of $P$ we have $P$
commutes with all unitary $U \in \mathcal{L}^{\prime}$ so
$P \in \mathcal{L}^{\prime\prime}$ and $P \in \mathcal{L}^{\prime}$ since $P$
is in the algebra generated by the $U^*FU \in \mathcal{L} ^{\prime}$ and clearly
$P$ is the smallest projection in
$\mathcal{L}^{\prime} \cap \mathcal{L} ^{\prime\prime}$ with $P \geq F$.
Note the $L(I) = L(I_o)$ so  $L(I - I_o) = 0$.  Since $F$ is the smallest
projection so that $L(F) = I_o$ it follows that $F(I-I_o) = (I-I_o)F = 0$.
Hence, $U^*FU(I-I_o) = U^*F(I-I_o)U = 0$ for all $U \in \mathcal{L} ^{\prime}$
so $P(I-I_o) = (I-I_o)P = 0$ and we have $P = PI = PI_o = I_oP = PI_oP$.

Next we note that $P\mathcal{L} P$ is a $*$-algebra.  It is clear that
$P\mathcal{L} P$ is invariant under the $*$-operation.  Note that if
$A,\medspace B \in \mathcal{L}$ then from Theorem 5.1 there is a
$C \in \mathcal{L}$ so that $ABF = CF$.  Then for $f \in \mathfrak M$ of the
above form we have
$$
ABf = \sum_{i=1}^m ABA_iFg_i = \sum_{i=1}
^m A_iABFg_i = \sum_{i=1}^m A_iCFg_i = \sum_{i=1}^m CA_iFg_i = Cf
$$
so $P\mathcal{L} P$ is a $*$-algebra.  Since the mapping $A \rightarrow FAF$ is
a completely positive one to one mapping from $\mathcal{L}$ to $F\mathcal{L} F$
it follows that the mapping A $\rightarrow$ PAP is a completely positive one
to one mapping of $\mathcal{L}$ onto $P\mathcal{L} P$.  From the previous lemma
we have for $A \in B(\mathbb C^p)$ that $L(A) = L(FAF)$ and $FP = F$ so if
$\psi(A) = PL(A)P$ we have
$$
\psi^2 (A) = PL(PL(A)P)P = PL(FPL(A)PF)P = PL(FL(A)F)P = PL(L(A))P = PL(A)P
= \psi (A)
$$
so we see that $\psi$ is a completely positive contractive idempotent map
of $B(\mathbb C^p)$ into itself and the range of $\psi$ is a $*$-algebra and
$\psi (A) = 0$ if and only if $L(A) = 0$.

Now suppose $L^{\prime}$ is another completely positive contractive
idempotent map of $B(\mathbb C^p)$ into itself with range $\mathcal{L}$ and
$F^{\prime}$ is the support projection for $L^{\prime}$ and it is not true that
$F^{\prime} \leq P$.  Then repeating the previous argument we find there is a
central projection $P^{\prime}$ so that $P^{\prime} \mathcal{L} P^{\prime}$ is
a $*$-algebra and the mapping $\psi ^{\prime}(A) =
P^{\prime}L^{\prime}(A)P^{\prime}$ is completely positive contractive linear
idempotent map of $B(\mathbb C^p)$ so that the range of $\psi ^{\prime}$ is a
$*$-algebra and $\psi ^{\prime}(A) = 0$ if and only if $L^{\prime}(A) = 0$.

Since it is not true that $F^{\prime} \leq P$ we have $Q = PP^{\prime}$
is central projection and $Q \neq P^{\prime}$.  We will prove this in not the
case because $Q = P = P'$.  Since every element in
$\mathcal{L} ^{\prime\prime}$ can be written as a polynomial of elements of
$\mathcal{L}$ and the unit $I$ of $B(\mathbb C^p)$ it follows that
$$
Q=\sum_{k=1}^m A_k\qquad \text{with} \qquad A_k = B_{k1}B_{k2}\cdots B_{kn_k}
$$
with the $B_{ki} \in \mathcal{L}$ or $B_{ki} = I$.  Now we claim that for each
$A_k$ above there is an element $C_k \in \mathcal{L}$ so that $PA_kP = PC_kP$
and $P^{\prime}A_kP^{\prime} = P^{\prime}C_kP^{\prime}$ for each $k = 1,\cdots
,m$.  This is seen as follows.  Since $P\mathcal{L} P$ is an algebra we have
$$
PA_kP = PB_{k1}PB_{k2}P \cdots PB_{kn_k}P = PC_kP
$$
where $C_k \in \mathcal{L}$ is unique.  (Note $PIP = PI_oP$ so any terms
involving $I$ can be replaced by $I_o$.) Likewise we have
$P^{\prime}A_kP^{\prime} = P^{\prime}C_k^{\prime}P^{\prime}$ where
$C_k^{\prime} \in \mathcal{L}$ is unique.   Since the multiplication defined by
$L$ and $L^{\prime}$ are the same it follows that $C_k = C_k^{\prime}$ and,
hence, if we define
$$
Q_o = \sum_{k=1}^m C_k
$$
we have $PQP = PQ_oP = PQ_o$ and $P^{\prime}QP^{\prime}=P^{\prime}Q_oP^{\prime}
=P^{\prime}Q_o$.  Next suppose $A \in \mathcal{L}$.  Then $QA = PPP^{\prime}A =
PQ_oA = PL(Q_oA)$ and $AQ = APPP^{\prime}=APQ_o = PAQ_o = PL(AQ_o)$ and since
$Q \in \mathcal{L}^{\prime}$ we have $L(Q_oA)=L(AQ_o)$ so $Q_o$ is in the
center of $\mathcal{L}$ where we view $\mathcal{L}$ as a $*-$algebra with the
Choi-Effros multiplication.

Next we show $Q_o = I_o$.  If this is not the case then $I_o-Q_o$ is a
central projection in $\mathcal{L}$.  Since each projection in the center
of $\mathcal{L}$ is the sum of minimal central projections there is a non
zero minimal central projection $B_o \in \mathcal{L}$ with $B_o \leq
I_o-Q_o$.  Note $B_oP$ and $B_oP^{\prime}$ are central projections in
$\mathcal{L}^{\prime}$.  Since the mappings $A \leftrightarrow$ PA and $A
\leftrightarrow P^{\prime} A$ are $*$-isomorphisms of $\mathcal{L}$ with
$P\mathcal{L}$ and of $\mathcal{L}$ with $P^{\prime}\mathcal{L}$ it follows
that $PB_o$ and $P^{\prime} B_o$ are both non zero.  Note the product of
these projections is zero since
$$
PB_oP^{\prime} B_o = PP^{\prime} B_oB_o = PQ_oB_oB_o = PL(Q_oB_oB_o) =
PL(L(Q_oB _o)B_o)
$$
and since $B_o \leq I_o-Q_o$ we have $L(Q_oB_o) = 0$ so $PB_oP^{\prime} B_o =0$.
Now $B_o\star\mathcal{L}$ (where $B_o\star A = L(B_oA)$ for $A \in
\mathcal{L} )$ is a Choi-Effros factor of type~I$_q.  $ Let $E_{ij} \in
B_o\star\mathcal{L}$ be a complete set of matrix units so $E_{ij}\star E _{nm}
= \delta_{jn}E_{im}$ which span $B_o\star\mathcal{L}$.  Note there is a
natural $*$-isomorphism $\psi$ of $B( \mathbb C^q)$ into $B_o\star\mathcal{L}$
defined as follows.  If $A \in B(\mathbb C^q)$ corresponds to the matrix
$\{ a_{ij}\}$ then
$$
\psi (A) = \sum_{i,j=1}^q a_{ij}E_{ij}.
$$
Since the mappings $A \leftrightarrow$ PA and $A \leftrightarrow
P^{\prime} A$ are $*$-isomorphisms of $\mathcal{L}$ it follows
that $\pi_1(A) = P\psi (A)$ and $\pi_2(A) = P^{\prime}\psi (A)$
are $*$-representation of $B(\mathbb C^q)$ on $PB_o$ and
$P^{\prime} B_o$, respectively.  As is well know any two
$*$-representations of $B(\mathbb C^q)$ are quasi-equivalent so
there are intertwining operators.  Specifically let $f$ be a unit
vector so that $PE_{11}f = f$ and $g$ be a unit vector so that
$P^{\prime} E_{11}g = g$.  Let
$$
f_i = PE_{i1}f\qquad \text{and} \qquad g_i = P^{\prime} E_{i1}g
$$
for $i = 1,\cdots ,q$.  We define the operator $C$ as follows.  We define
$Cf_i = g_i$ for $i = 1,\cdots ,q$  and $Cf = 0$ if $f$ is orthogonal to
the $f _i$.  This defines an operator $C \in B(\mathbb C^p)$ of rank $q$.
A little computation show that $C\pi_1(A) = \pi_2(A)C$ for $A \in B(\mathbb
C^q)$.  Now if $A \in B_o\star\mathcal{L}$ then $A$ is a linear combination
of the $E_{ij}$ so $A = \psi (A_1)$ for a unique $A_1\in B(\mathbb C^q)$ so
we can write $A_1 = \psi^{-1}(A)$.  Then $\pi_1(A_1) = \pi_1(\psi^{-1}(A)) =
PB_oA$ and $\pi_2(A_1) = \pi_2(\psi^{-1}(A)) = P^{\prime} B_oA$ and the fact
that $C\pi_1(A) = \pi_2(A)C$ means that $CPB_oA = P^{\prime} B_oAC$ for $A
\in \mathcal{L}$.  Notice that the range of $C$ is contained the range of
$P^{\prime} B_o$ and the range of $C^*$ is contained in the range of $PB_o$
so we have
$$
C = CPB_o = P^{\prime} B_oC = P^{\prime} B_oCPB_o.
$$
It follows that $C \in \mathcal{L}^{\prime}$ since for $A \in \mathcal{L}$ we
have
$$
 CA = CPB_oA = P^{\prime} B_oAC = AP^{\prime} B_oC = AC.
$$

We have reached a contradiction.  Recall $PB_o$ is a central projection in
$\mathcal{L}^{\prime}$ and we see that $C \in \mathcal{L}^{\prime}$ does not
commute with it.  Hence, the assumption that $Q_o \neq I_o$ leads to a
contradiction.  Hence, $PP^{\prime} = PI_o = P^{\prime}I_o$.  But as we showed
earlier in this proof we have $PI_o = P$ and by the same argument
$P^{\prime}I_o = P^{\prime}$ so we have $PP^{\prime} = P = P^{\prime}$.  Since
the support projection $F^\prime$ for $L^{\prime}$ satisfies
$F^{\prime} \leq P^{\prime}$ we have $F^{\prime} \leq P$.

We complete the proof by showing
$$
L^{\prime}(A) = L^{\prime} (PAP)
$$
for $A \in B(\mathbb C^p)$.  Now we have
$$
L^{\prime}(PAP) = L^{\prime} (F^{\prime}PAPF^{\prime})
  =  L^{\prime} (F^{\prime}AF^{\prime}) = L^{\prime} (A).
$$
for $A \in B(\mathbb C^p).$

\end{proof}

\begin{defn}
Suppose $L$ is a linear completely positive contractive idempotent map of
$\mathbb C^p$ into itself and $\mathcal{L}$ is the range of $L$.  Suppose $F$
is the support projection for $L$ which is the smallest hermitian projection
$F \in B(\mathbb C^p)$ so that $L(F) = L(I) = I_o$.  The projection $P$ which
is the smallest projection in $\mathcal{L} ^{\prime}\cap\mathcal{L}
^{\prime\prime}$ with $P \geq F$ is called the maximal support projection for
$\mathcal{L}$.
\end{defn}

In the above definition it appears as if the maximal support projection
depends on $L$ but from the previous theorem we see that any linear
completely positive contractive idempotent map with the same range as $L$
yields the same projection $P$.

Next we consider the case where we have two completely positive contractive
linear idempotent mappings $L$ and $L_1$ of $B(\mathbb C^p)$ into itself and
$L - L_1$ is a completely positive map.

\begin{thm}
Suppose $L$ and $L_1$ are completely positive
contractive linear idempotent mappings of $B(\mathbb C^p)$ into itself and $L
- L_1$ is a completely positive map.  Suppose further that $L_1(I) = E$ and
$E$ is a projection.  Let $\mathcal{L}$ be the range of $L$.  Then $E$ commutes
with every element of $\mathcal{L}$ (i.e., $E \in \mathcal{L}^{\prime}$) and
$E \leq L(I)$ and $L _1(A) = EL(A) = EL(A)E$ for $A \in B(\mathbb C^p))$.
\end{thm}
\begin{proof}  Assume the hypothesis and notation of the theorem.  Suppose
$A \in B(\mathbb C^p)$ and $0 \leq A \leq I$ then since $L_1$ is completely
positive we have $0 \leq L_1(A) \leq E$ from which it follows that $L_1(A)
= EL_1(A)E$ and since $B(\mathbb C^p)$ is the complex linear span of its
positive elements we have $L_1(A) = EL_1(A)E$ for all $A \in B(\mathbb C^p)$.
Now consider the mapping
$$
\phi (A) = E(L(A) - L_1(A))E = EL(A)E - L_1(A)
$$
for $A \in B(\mathbb C^p)$.  Since $L \geq L_1$ this mapping is completely
positive.  Note $\phi (I) = EL(I)E - E$ and since $L$ is contractive we
have $\phi (I) \leq 0$ but since $\phi$ is completely positive we have
$\phi (I) = 0$ so $\phi = 0$ and
$$
L_1(A) = EL(A)E
$$
for all $A \in B(\mathbb C^p)$.  Now suppose $A \in \mathcal{L}$ and $A \geq 0$.
Then
$$
L(A) - L_1(A) = A - EAE \geq 0
$$
from which we conclude that $E$ commutes with $A$.  Since $\mathcal{L}$ is the
complex linear span of its positive elements we have $E \in \mathcal{L}
^{\prime}.$ \end{proof}

Notice that if $\omega$ is a $q$-weight map over $\mathbb C^p$ with range
$\mathcal{L} (\omega )$ when we compute the generalized boundary representation
of $\omega$ given by
$$
\pi_t^\# = (\iota + \omega \vert_t \Lambda )^{-1}\omega \vert_t
$$
in computing the inverse of $\iota + \omega \vert_t \Lambda$ we only need
compute $(\iota + \omega \vert_t \Lambda )$ on $\mathcal{L} (\omega )$.
We know from the general theory that the inverse exists but for
calculational purposes we only care about the inverse on $\mathcal{L}
(\omega )$.  In the case where $\omega$ is of index zero then we know that
the map $A \leftrightarrow$ PAP where $P$ is the maximal support projection
for $\mathcal{L} (\omega )$ is completely positive and one to one in both
directions.  So to parameterize $\mathcal{L} (\omega )$ we can take a
complete set of matrix units $e _{ij}^r$ for $\mathcal{L} (\omega )$ chosen
so that $Pe_{ij}^rP$ are a complete set of matrix units for $P\mathcal{L}
(\omega )P$.  Now when we analyze a $q$-weight map over $\mathbb C^p$ of
index zero when we speak of the map $\omega \vert_t \Lambda$ we will often
consider this to be a map of $\mathcal{L} (\omega) $ into itself rather
than a map of $B(\mathbb C^p)$ into itself.  This may seem an obvious
observation but it took us some time to realize this.  To give an example.
Note that limit points $L$ of $\pi_t^\#\Lambda$ as $t \rightarrow 0+$ are
not unique but if we restrict our attention to $\mathcal{L} (\omega )$ the
limit is unique and the limit is the identity map.

Next we show that if $\omega$ is $q$-pure $q$-weight map over $\mathbb C^p$ of
index zero then range of $\omega$ is a factor with the Choi-Effros product.
First we prove a routine lemma.

\begin{lem}
Suppose $B(X)$ is the Banach space of all linear maps of a finite dimensional
Banach space $X$ into itself and $A_n \in B(X)$ for $n = 1,2,\cdots $ is a
sequence of invertible elements and $A_n \rightarrow A$ as $n\rightarrow\infty$
then $A$ is invertible and $A _n^{-1}\rightarrow A^{-1}$ as $n \rightarrow
\infty$ if and only if there is a constant $K$ so that $\Vert A_n^{-1}\Vert
\leq K$ for all $n$.
\end{lem}
\begin{proof}  Assume $A_n$ is a sequence as stated above.  If $A$ is
invertible and $A_n^{-1} \rightarrow A^{-1}$ as $n \rightarrow \infty$ then
$\Vert A_n ^{-1}\Vert$ is uniformly bounded.  Now suppose there is a
constant $K$ so that $\Vert A_n^{-1}\Vert \leq K$ for all $n.  $ Then we
have
\begin{align*}
\Vert A_n^{-1} - A_m^{-1}\Vert  &= \Vert A_n^{-1}(A_n - A_m)A_m^{-1}\Vert
\leq \Vert A_n^{-1}\Vert\medspace\Vert (A_n - A_m)\Vert\medspace\Vert
A_m^{-1}\Vert
\\
&\leq K^2\Vert A_n - A_m\Vert \rightarrow 0
\end{align*}
as $n,m \rightarrow \infty$.  Since $B(X)$ is complete there is a $B
\in B(X)$ so that $A_n^{-1}\rightarrow B$ as $n \rightarrow
\infty$.  We show $A$ is invertible and $B = A^{-1}$.  Since $A_n
\rightarrow A$ there is a constant $K^{\prime}$ so that $\Vert A_n\Vert
\leq K^{\prime}$ for all $n$.  Now we have
\begin{align*}
\Vert AB - I\Vert &= \Vert (A - A_n)B + A_n(B - A_n^{-1})\Vert
\\
&\leq \Vert A - A_n\Vert\medspace\Vert B\Vert
+ K^{\prime}\Vert B - A_n^{-1}\Vert \rightarrow 0
\end{align*}
as $n \rightarrow \infty$.  Hence $AB = I$ and $B = A^{-1}.$ \end{proof}

\begin{thm}
Suppose $\omega$ is a $q$-weight map over $\mathbb C^p$ and
$\pi _t^\#$ is the generalized boundary representation of $\omega$.
Suppose that $\psi_t^\#$ is a completely positive map of $\mathfrak A (\mathbb
C^p)$ into $B(\mathbb C^p)$ which is subordinate to $\pi_t^\#$ so $\pi_t^\#
\geq \psi_t^\#$ for each $t > 0$.  Let
$$
\eta_t = (\iota - \psi_t^\#\Lambda )^{-1}\psi_t^\#
$$
for $t > 0$.  Then $\omega \vert_t  \geq_q \eta_t$ for $t > 0$ and if $\eta$
is a weak limit point in $B(\mathbb C^p) \otimes \mathfrak A (\mathbb C^p)_*$ of
$\eta_t$ as $t \rightarrow 0$ then $\eta$ is a $q$-subordinate of
$\omega$ so $\omega \geq_q \eta$.
\end{thm}
\begin{proof}  Assume the hypothesis and notation of the theorem.  Since the
generalized boundary representation $\phi_t^\#$ of $\eta_s$ is $\psi_s^\#$
for $t \leq s$ and the boundary representation $\pi_t^\#$ of $\omega
\vert_s$  is $\pi_s^\#$ for $t \leq s$ it follows that $\omega \vert_s
\geq_q \eta_s$ for all $s > 0$.  Now suppose that $\eta$ is a weak limit
point of $\eta_t$ as $t \rightarrow 0+$.  Since $B(\mathbb C^p)$ is finite
dimensional there is a decreasing sequence of $t_k > 0$ so that
$\eta_{t_k}(A) \rightarrow \eta(A)$ as $k \rightarrow \infty$ for each $A
\in \mathfrak A (\mathbb C^p)$.  To simplify notation we let $\eta_k = \eta
_{t_k}$ and $\psi_k^\# = \psi_{t_k}^\#$, $\pi_k^\# = \pi_{t_k}^\#$ and
$\omega_k = \omega \vert_{t_k}$   for $k = 1,2,\cdots$.  Suppose $s > 0$.
Since $\pi_k^\# \geq \psi_k^\#$ we have for $k$ large enough so that $t_k <
s$ that
$$
(\iota + \eta_k \vert_s \Lambda )^{-1}\eta_k \vert_s  \leq (\iota + \omega
\vert_s \Lambda )^{-1}\omega \vert_s.
$$
Now we have $\eta_k \vert_s \Lambda \rightarrow \eta \vert_s
\Lambda$ as $k \rightarrow \infty$ and from lemma 5.5 we have $(\iota +
\eta \vert_s \Lambda )$ is invertible and
$$
(\iota + \eta_k \vert_s \Lambda )^{-1}\rightarrow (\iota + \eta\vert_s
\Lambda )^{-1}
$$
as $k \rightarrow \infty$ if and only if there is a constant $K$ so that
$$
\Vert (\iota + \eta_k \vert_s \Lambda )^{-1}\Vert \leq K
$$
for all $k = 1,2,\cdots$.  Now we have
$$
\eta_k \vert_s  = (\iota -  \psi_k^\#\Lambda )^{-1}\psi_k^\# \vert_s
$$
so
$$
\iota + \eta_k \vert_s \Lambda = (\iota - \psi_k^\#\Lambda )^{-1}(\iota -
(\psi_k^\# - \psi_k^\# \vert_s)\Lambda )
$$
which yields
\begin{align*}
(\iota + \eta_k \vert_s \Lambda )^{-1} &= (\iota - (\psi_k^\# - \psi_k^\#
\vert_s)\Lambda )^{-1}(\iota - \psi_k^\#\Lambda )
\\
&= \iota - (\iota - (\psi_k^\# - \psi_k^\# \vert_s)\Lambda )^{-1}\psi_k^\#
\vert_s \Lambda
\end{align*}
so if we let $T_k = \iota - (\iota + \eta_k \vert_s \Lambda )^{-1}$ we have
$$
T_k = (\iota - (\psi_k^\# - \psi_k^\# \vert_s)\Lambda )^{-1}\psi_k^\#
\vert_s \Lambda
$$
and using the geometric series expansion $(1-x)^{-1}= 1 + x + x^2 + \cdots
$ which converges since $\Vert (\psi_k^\#-\psi_k^\# \vert_s) \Lambda\Vert
\leq \Vert\pi_k^\#\Lambda\Vert < 1$ we have
$$
T_k = (\iota+(\psi_k^\#-\psi_k^\# \vert_s )\Lambda+((\psi_k^\#-\psi_{k
\vert s}^\#\vert )\Lambda )^2+\cdots )\psi_k^\# \vert_s \Lambda
$$
so $T_k$ is the sum of completely positive terms so $T_k$ is completely
positive and, hence, $\Vert T_k\Vert = \Vert T_k(I)\Vert$.  Note that if
$\phi _i$ and $\phi_i^{\prime}$ are completely positive maps and $\phi_i
\geq \phi_i^{\prime}$ for $i = 1,\cdots ,n$ then
$$
\phi_1\phi_2\cdots \phi_n \geq \phi_1^{\prime}\phi_2^{\prime}\cdots
\phi_n^{\prime}
$$
so we have
\begin{align*}
T_k &\leq (\iota + \psi_k^\#\Lambda + (\psi_k^\#\Lambda )^2 \cdots )\psi_k
^\# \vert_s \Lambda
\\
&= (\iota - \psi_k^\#\Lambda )^{-1}\psi_k^\# \vert_s \Lambda = \eta_k
\vert_s \Lambda  \leq \omega \vert_s \Lambda
\end{align*}
and, hence,
$$
\Vert T_k\Vert \leq \Vert\omega \vert_s (\Lambda )\Vert
$$
for all $k$ with $t_k \leq s$.  Hence, we have
$$
\Vert (\iota + \eta_k \vert_s \Lambda )^{-1}\Vert = \Vert\iota - T_k\Vert
\leq \Vert\iota\Vert + \Vert T_k\Vert \leq 1 + \Vert\omega \vert_s (
\Lambda )\Vert
$$
for $t_k \leq s$.  Hence, from Lemma 5.5 we have $\iota + \eta \vert_s
\Lambda$ is invertible and
$$
(\iota + \eta_k \vert_s \Lambda )^{-1} \rightarrow  (\iota + \eta \vert_s
\Lambda )^{-1}
$$
as $k \rightarrow \infty$ and, hence,
$$
(\iota + \eta_k \vert_s \Lambda )^{-1}\eta_k \vert_s \rightarrow (\iota +
\eta \vert_s \Lambda )^{-1}\eta\vert_s
$$
as $k \rightarrow \infty$.  Since
$$
(\iota + \eta_k \vert_s \Lambda )^{-1}\eta_k\vert_s  \leq (\iota + \omega
\vert_s \Lambda )^{-1}\omega \vert_s
$$
for all $k$ we have
$$
(\iota + \eta \vert_s \Lambda )^{-1}\eta \vert_s  \leq (\iota + \omega
\vert_s \Lambda )^{-1}\omega \vert_s
$$
and, hence, $\eta$ is a $q$-subordinate of $\omega .$ \end{proof}

\begin{thm}
Suppose $\omega$ is a $q$-weight map over $\mathbb C^p$
of index zero and $\mathcal{L}$ is the range of $\omega$ which is a Choi-Effros
algebra with the multiplication $A\star B = L(AB)$ where $L$ is a limit
point of $\pi_t^\#\Lambda$ as $t \rightarrow 0+$.  Let $Q$ be a non zero
central projection in $\mathcal{L}$.  Then there is a $q$-subordinate
$\eta$ of $\omega$ so that the range of $\eta$ is $Q\star\mathcal{L} $.  In
particular, if $\omega$ is a $q$-pure $q$-weight map over $\mathbb C^p$ of
index zero then the range $\mathcal{L}$ of $\omega$ is a Choi-Effros factor.
\end{thm}
\begin{proof}  Assume the hypothesis and notation of the theorem.  If
$\mathcal{L}$ is a factor with the Choi-Effros product then $\eta = \omega$ and
the proof is complete.  Suppose then that $\mathcal{L}$ is not a factor with
the Choi-Effros product and $Q \in \mathcal{L}$ is a central projection in
$\mathcal{L}$.  Let $I_o$ be the unit of $\mathcal{L}$ and note $I_o-Q$ is not
zero.  Let $\pi_t^\#$ be the generalized boundary representation of $\omega$
and let $\psi_t^\# = Q\star \pi_t^\#$ and note that $\pi_t^\# - \psi_t^\# =
(I_o - Q)\star \pi_t^\# \geq 0$ for $t > 0$ and let
$$
\eta_t = (\iota - \psi_t^\#\Lambda )^{-1}\psi_t^\# = \psi_t^\# +
\psi_t^\#\Lambda \psi_t^\# + \cdots .
$$
Note the range of $\eta_t$ is contained in $Q\star\mathcal{L}$.  Then from
Theorem 5.6 we know that any weak limit $\eta$ of $\eta_t$ is a $q$-subordinate
of $\omega $.

Next we show that any limit point $\eta$ is not zero and the range of
$\eta$ is $Q\star\mathcal{L}$.  Now we have
$$
\eta_t = (\iota - \psi_t^\#\Lambda )^{-1}\psi_t^\#\qquad \text{and} \qquad
\psi_t^\# = Q\star\pi_t^\# = Q\star (\iota + \omega \vert_t \Lambda
)^{-1}\omega \vert_t
$$
so
$$
\eta_t = (\iota - Q\star\pi_t^\#\Lambda )^{-1}Q\star(\iota + \omega \vert_t
\Lambda )^{-1}\omega \vert_t
$$
and since
$$
(\iota + \omega \vert_t \Lambda )^{-1} = \iota - \pi_t^\#\Lambda
$$
we have
$$
\eta_t = (\iota - Q\star\pi_t^\#\Lambda )^{-1}Q\star(\iota -
\pi_t^\#\Lambda ) \omega \vert_t = R_t\omega \vert_t
$$
and
\begin{align*}
R_t &= (\iota + Q\star\pi_t^\#\Lambda + (Q\star\pi_t^\#\Lambda )^2 + \cdots
)Q\star (\iota - \pi_t^\#\Lambda )
\\
&= (\iota + Q\star\pi_t^\#\Lambda + (Q\star\pi_t^\#\Lambda )^2 + \cdots
)Q\star -(Q\star\pi_t^\#\Lambda + (Q\star\pi_t^\#\Lambda )^2 + \cdots )
\\
&= Q\star - Q\star\pi_t^\#\Lambda (I_o-Q)\star - (Q\star\pi_t^\#\Lambda)^2
(I_o-Q)\star - (Q\star\pi_t^\#\Lambda)^3(I_o-Q)\star - \cdots
\\
&= Q\star - Q\star (\iota - Q\star\pi_t^\#\Lambda )^{-1}(I_o - Q)\star
\end{align*}
so we have
\begin{align*}
\eta_t &= Q\star\omega \vert_t  - Q\star (\iota - Q\star\pi_t^\#\Lambda
)^{-1}(\omega \vert_t  - Q\star\omega \vert_t )
\\
&= Q\star (\iota - (\iota - Q\star\pi_t^\#\Lambda )^{-1}(I_o - Q)\star
)\omega  \vert_t
\end{align*}
for $t > 0$.  Let $\eta$ be a limit point of $\eta_t$ as $t \rightarrow 0
+$.  Since the range of $\omega$ is $\mathcal{L}$ for any $A \in \mathcal{L}$
there is a $t_o > 0$ and a $B \in \mathfrak A (\mathbb C^p)$ so that $E(t_o,
\infty ) BE(t_o,\infty ) = B$ (where $E(s,\infty )$ is the projection in
$\mathfrak A (\mathbb C^p)$ onto function with support in $[s,\infty ))$ and
$\omega (B) = A$.  Now let $A$ be any element of $\mathcal{L}$ so that $Q\star
A = A$.  Then $(I_o - Q)\star A = 0$ and we see that for $t \in (0,t_o)$ we have
$\eta_t(B) = A$.  This then shows that every limit point $\eta$ is not
zero and the range of $\eta$ is $Q\star\mathcal{L}$.  Then $\eta$ is a
$q$-subordinate of $\omega$ with range $Q\star\mathcal{L}$.

Now we prove the last sentences of the theorem.  Suppose then that $\omega$
is a $q$-pure $q$-weight map over $\mathbb C^p$ with range $\mathcal{L}$ which
is an algebra with the Choi-Effros multiplication.  Suppose $\mathcal{L}$ is not
a Choi-Effros factor.  Then there are at least two minimal central projections
$Q_1$ and $Q_2$ so that $Q_1\star Q_2 = 0$.  By what we have just proved there
are $q$-subordinates $\eta_1$ and $\eta_2$ of $\omega$ with ranges $Q_1
\mathcal{L}$ and $Q_2\mathcal{L}$, respectively.  Since $Q_1$ and $Q_2$ are
disjoint central projections it is immediately clear that it is not true that
$\eta_1 \geq_q \eta_2$ or $\eta_2 \geq_q \eta _1$.  Hence, $\omega$ is not
$q$-pure.
\end{proof}

Our ultimate goal is to understand $q$-pure $q$-weight maps
$\omega$ of index zero and from the last theorem we know that for
such maps the range of $\omega$ is a Choi-Effros factor (i.e. $\mathcal{L}$
the range of $\omega$ is a factor with the Choi-Effros product).
The remainder of this section we will focus on this
situation.  The main results of the next few theorems and lemmas
is to carry over the results of the last section to this new
setting.  What is surprising is that the arguments of the last
section carry over to our new setting with virtually no change
except for a change in notation.  What we will do is to replace
$B(\mathbb C^p)$ of the last section by $\mathcal{L}$ the range of
$\omega$ so instead of considering a mapping for $\mathfrak{A}
(\mathbb C ^p)$ into $B(\mathbb C^p)$ we will consider the same
mapping from $\tilde{\mathfrak{A}} (\mathbb C^p)$ to $\mathcal{L}$
where $\tilde{\mathfrak{A}} (\mathbb C^p)$ is $\mathcal{L} \otimes
\mathfrak{A} (\mathbb C )$.  Note $\mathfrak{A} (\mathbb C ^p) =
B(\mathbb C^p)\otimes\mathfrak{A} (\mathbb C )$ so $\tilde
{\mathfrak{A}} (\mathbb C^p)$ is obtained by restriction.  We will
spell this out in detail but for the moment we what to emphasize
the main point that if you see a tilde on a mapping it means the
same mapping only restricted $\mathcal{L}$.  The mapping $\phi_t =
\omega\vert_t\Lambda$ is a mapping of $B(\mathbb C^p)$ into
itself.  The mapping $\tilde \phi_t$ is the same mapping
restricted to $\mathcal{L}$.  The mapping $L_t = \pi_t^\#\Lambda$
is a mapping of $B(\mathbb C^p)$ into itself and the mapping
$\tilde L_t$ is the same mapping restricted to $\mathcal{L}$.

The basic rule is that the tilde means restrict to $\mathcal{L}$
so $\tilde B(\mathbb C^p) = \mathcal{L}$.  The mapping $\Lambda$
is a mapping of $B(\mathbb C^p)$ into it $B(\mathbb C^p\otimes
L^2(0,\infty ))$.  The mapping $\tilde{\Lambda}$ is the same
mapping restricted to $\mathcal{L}$. The range of $\tilde{\Lambda}$
is in $B(\mathbb C^p\otimes L^2(0,\infty )) = B(\mathbb
C^p)\otimes B(L^2(0,\infty ))$ but it is also in $\tilde B(\mathbb
C^p\otimes L^2(0,\infty )) = \tilde B(\mathbb C^p)\otimes
B(L^2(0,\infty )) = \mathcal{L}\otimes B(L^2(0,\infty ))$.  Here
is the tricky part.  In order to specify an element of $A \in
B(\mathbb C^p)$ we specify a $(p\times p)$-matrix $a_{ij} \in
\mathbb C$ for $i,j = 1,\cdots ,p$.  Now $\mathcal{L}$ is a
Choi-Effros factor of type~I$_q$ which means there are matrix
units $E_{ij} \in \mathcal{L}$ which span $\mathcal{L}$ so to
specify an element $A \in \mathcal{L}$ we specify a $(q\times
q)$-matrix $a_{ij} \in \mathbb C$ for  $i,j = 1,\cdots ,q$.
Similarly each element of $B( \mathbb C^p\otimes L^2(0,\infty ))$
can be specified by a $(p\times p)$-matrix of elements of
$B(L^2(0,\infty ))$ and each $A \in \tilde B(\mathbb C^p\otimes
L^2(0,\infty )))$ can be specified by a $(q\times q)$-matrix of
elements of $A_{ij} \in B(L^2(0,\infty ))$ by the formula
\begin{equation*}
A = \sum_{i,j=1}^q E_{ij}A_{ij}.
\end{equation*}
Notice there is a bijection from $\tilde B(\mathbb C^p\otimes
L^2(0,\infty ))$ to $B(\mathbb C^q\otimes L^2(0,\infty )))$ and
this bijection is a completely positive contraction in both
directions.  In the language of Choi and Effros it is a complete
order isomorphism.  Now here is the tricky point.  In calculations
involving $\Lambda$ we used $(p\times p)$-matrices but in
calculations involving $\tilde{\Lambda}$ we will use $(q\times
q)$-matrices.  The tricky point is if you ask what a particular
matrix element $a_{ij} \in \mathbb C$ corresponds to the answer is
not so simple. For example it involves the choice of matrix unit
$E_{ij} \in \mathcal{L} $.  For example in a calculation involving
$\tilde{\Lambda}$ we might have a function $(g_{ik})_j(x) \in
\mathbb C$ for $k \in J$ some index set and $i,j = 1,\cdots ,q$.
But translating back to the mapping $\Lambda$ these functions with
values in $\mathbb C$ become functions with values in $\mathbb
C^p$ (actually is certain subspaces of $\mathbb C^p$ depending in
the values of $i$ and $j)$ and there is even the possibility that
the term corresponding to $(g_{ik})_j(x)$ would involve multiple
terms $(g_{ik})_j$ with different $k's$.  This is both the beauty
and draw back of the tilde notation.  The draw back is that
understanding formulas in terms of $\tilde{\Lambda}$ in terms of
$\Lambda$ can be quite complicated.  The beauty is that
calculations with $\tilde{\Lambda}$ can be extremely simple where
as the corresponding calculation with $\Lambda$ would be so
complicated that one would not have the courage to undertake them.

We formalize this with the following extended definition.

\begin{defn}\label{5.8}
A subset $\mathcal{L}$ of $B(\mathbb C^p)$ is called a Choi-Effros factor
if there is a completely positive contractive idempotent mapping $L$ of
$B(\mathbb C^p)$ into itself with range $\mathcal{L}$ and $\mathcal{L}$
equipped with the Choi-Effros multiplication $A\star B = L(AB)$ is a
$(q\times q)$-matrix algebra with unit $I _o$.  A complete set of matrix
units for $\mathcal{L}$ consist of elements $E_{ij}$ for $i,j = 1,\cdots ,q$
which span $\mathcal{L}$ and satisfying the relations
\begin{equation*}
E_{ij}^* = E_{ji},\qquad E_{ij}\star E_{nm} = \delta_{jn}E_{im}
\end{equation*}
and
\begin{equation*}
E_{11}+E_{22}+ \cdots +E_{qq} = I_o.
\end{equation*}
If $A \in \mathcal{L}$ the matrix entries of $A$ are complex numbers $a_{ij
}$ so that
\begin{equation*}
A = \sum_{i,j=1}^q a_{ij}E_{ij}.
\end{equation*}
If $\psi$ is a mapping of $B(\mathbb C^p)$ into another space we denote by
$\tilde \psi$ the same mapping restricted to $\mathcal{L}$.  We denote,
\begin{equation*}
\tilde B(\mathbb C^p\otimes L^2(0,\infty )) = \mathcal{L} \otimes
B(L^2(0,\infty ))
\end{equation*}
\begin{equation*}
\tilde{\mathfrak{A}} (\mathbb C^p) = \mathcal{L} \otimes \mathfrak{A}
(\mathbb C ) \subset \mathcal{L} \otimes B(L^2(0,\infty )) \subset \tilde
B(\mathbb C^p \otimes L^2(0,\infty )).
\end{equation*}
Each element $A \in \tilde B (\mathbb C^p\otimes L^2(0,\infty ))$ or $A \in
\tilde{\mathfrak{A}} (\mathbb C^p)$ can be written as a $(q\times
q)$-matrix of elements in $B(L^2(0,\infty ))$
\begin{equation*}
A = \sum_{i,j=1}^q E_{ij}\otimes A_{ij}
\end{equation*}
and the operators $A_{ij} \in B(L^2(0,\infty ))$ are called the
coefficients of $A$ in $\mathcal{L} \otimes B(L^2(0,\infty ))$.  The
mapping from $A$ to the matrix of coefficients of $A$ gives us a complete
order isomorphism of $\tilde B(\mathbb C^p\otimes L^2(0,\infty ))$ and
$\tilde{\mathfrak{A}} (\mathbb C^p)$ with $B(\mathbb C^q\otimes L^2(0,\infty
))$ and $\mathfrak{A} (\mathbb C^q)$, respectively.

Given an element $A \in \mathcal{L}$ we denote by $\tilde tr(A)$ the trace
of $A$ in $\mathcal{L}$ defined as follows.  If
\begin{equation*}
A = \sum_{i,j=1}^q a_{ij}E_{ij}\qquad \text{then} \qquad \tilde tr(A) =
\frac {1} {q} \sum_{i=1}^q a_{ii}.
\end{equation*}
\end{defn}

Technically when we refer to the matrix elements $a_{ij}$ of $A$
we should say with respect to the matrix units $E_{ij}$ but we
will forgo repeating this.  Note the trace $\tilde tr(A)$ does not
depend on the particular choice of matrix unit $E_{ij}$ for
$\mathcal{L}$ and $\tilde tr(I _o) = 1$.

In summary all the calculations involving $\tilde \phi_t,\medspace \tilde
\Lambda ,\medspace \tilde{\mathfrak{A}} (\mathbb C^p)$ and $\mathcal{L}$ are
the same as calculations involving $\phi_t,\medspace \Lambda ,\medspace
\mathfrak{A} (\mathbb C^p)$ and $B(\mathbb C^p)$ except $\mathbb C^p$ is
replaced by $\mathbb C^q$ and we think of $\mathcal{L}$ as $B(\mathbb C^q).
$ For us the fact that this works is a triumph of the work of Choi and
Effros without which we would never have come this far.

The first example which shows the advantage of the tilde notation
comes from the very basic property of the $q$-weight map
$\omega$.  Recall this is a completely positive mapping of
$\mathfrak{A} (\mathbb C^p)$ into $B(\mathbb C^p)$ which satisfies
inequality $\omega (I - \Lambda ) \leq I$ where $\Lambda$ means
$\Lambda (I)$ by which we mean
\begin{equation*}
\omega\vert_t(I - \Lambda (I)) \leq I
\end{equation*}
for all $t > 0$.  The corresponding inequality the tilde notation
is
\begin{equation*}
\omega\vert_t(I - \Lambda (I_o)) \leq I_o
\end{equation*}
for all $t > 0$ where $I_o$ is the identity $\mathcal{L}$ the range of
$\omega $.  The next theorem shows this is true and not only in the factor
case but for the general case as well.

\begin{thm}\label{5.9}
Suppose $\omega$ is a $q$-weight map over $\mathbb C^p$ of index zero and
$\mathcal{L}$ is the range of $\omega$ and $I_o \in \mathcal{L}$ is the
unit of $\mathcal{L}$ (i.e. $I_o\star A = A\star I_o = A$ for all $A \in
\mathcal{L}$).  Then $\omega (I-\Lambda (I_o)) \leq I_o$ by which we mean
$\omega\vert_t(I-\Lambda (I_o)) \leq I_o$ for all $t > 0$.
\end{thm}
\begin{proof}  Assume the hypothesis and notation of the theorem and let
$\pi_t^\#$ be the generalized boundary representation of $\omega$.  First we
claim that $\pi_t(I) \leq I_o$ for all $t$.  Let $P$ be the maximal support
projection for $\mathcal{L}$.  Since $\pi_t^\# (I) \leq I$ we clearly have
$P\pi_t(I) = P\pi _t(I)P \leq  P$ for $t > 0$ and since the mapping $A
\rightarrow$ PA $= PAP$ is an order isomorphism for $A \in \mathcal{L}$ and
$PI_o = P$ and the range of $\pi_t^\#$ is $\mathcal{L}$ it follows that
$\pi _t^\# (I) \leq I_o$ for all $t > 0$.  Now we have
\begin{align*}
\omega\vert_t(I - \Lambda (I_o)) &= (\iota - \pi_t^\#\Lambda )^{-1}\pi_t^\#
(I - \Lambda (I_o))
\\
&= (\iota - \pi_t^\#\Lambda )^{-1}(\iota - \pi_t^\#\Lambda )\pi_t^\# (I)
\\
&\qquad - (\iota - \pi_t^\#\Lambda )^{-1}(\pi_t^\#\Lambda (I_o) -
\pi_t^\#\Lambda \pi_t^\# (I))
\\
&= \pi_t^\# (I)-(\iota-\pi_t^\#\Lambda )^{-1}\pi_t^\#\Lambda (I_o-\pi_t^\# (I))
\end{align*}
and since $\pi_t^\# (I) \leq I_o$ and the mapping
$(\iota-\pi_t^\#\Lambda )^{-1}\pi_t^\#\Lambda$ is completely positive we have
\begin{equation*}
\omega\vert_t(I - \Lambda (I_o)) \leq I_o
\end{equation*}
for all $t > 0.$ \end{proof}

The next theorem is translates the results of the last section where
$\mathcal{L}$ was $B(\mathbb C^p)$ to our new setting where $\mathcal{L}$
is a Choi-Effros factor.  In an earlier version of this paper we
laboriously translated each of the theorems and lemmas.  What the next theorem
shows is we get it all for free.  The simple trick is defining $\tilde \omega$
as explained in the proof of the theorem.

\begin{thm}\label{5.10}
Suppose $\omega$ is a $q$-weight map over $\mathbb C^p$ of index zero over
$B(\mathbb C^p)$ and the range of $\omega$ denoted by $\mathcal{L}$ is a
Choi-Effros factor of type~I$_q$ (where $q < p)$.  For $t > 0$ let
$\phi_t = \omega\vert_t\Lambda$ and let $\pi_t^\#$ be the generalized boundary
representation of $\omega$.  We denote by $\tilde \Lambda$ and $\tilde \phi_t$
these mappings restricted to $\mathcal{L}$ as described above.  Let
\begin{equation*}
v_t = \tilde tr(I_o + \tilde \phi_t(I_o))\qquad \text{and} \qquad \Theta_t
= v_t^{-1}(\iota + \tilde \phi_t).
\end{equation*}
Then $\Theta_t$ is completely positive invertible mapping of $\mathcal{L}$
onto $\mathcal{L}$ and the inverse $\Theta_t^{-1}$ is conditionally
negative and $\Theta_t$ converges to a limit $\Theta$ as $t \rightarrow 0+$
and $\Theta$ is completely positive.  The limit $\Theta$ is invertible and
the inverse $\Theta^{-1} = \psi$ is conditionally negative and $\Theta_t
^{-1} \rightarrow \Theta^{-1} = \psi$ as $t \rightarrow 0+$.

We define $\vartheta = \Theta^{-1}\omega = \psi\omega$ where $\vartheta \in
B(\mathbb C^p) \otimes \mathfrak{A} (\mathbb C^p)_*$ is completely positive
and
\begin{equation*}
\vartheta (A) = \lim_{t\rightarrow 0+} v_t\pi_t^\# (A)
\end{equation*}
for $A \in \mathfrak{A} (\mathbb C^p)$.  We define $\vartheta_{ij}$ by the
formula
\begin{equation*}
\vartheta (A) = \sum_{i,j=1}^q \vartheta_{ij}(A)E_{ij}
\end{equation*}
and we have $\omega = \Theta\vartheta = \psi^{-1}\vartheta$.  We have
$\tilde{\vartheta}$ the restriction of $\vartheta$ to $\tilde{\mathfrak{A}}
(\mathbb C ^p)$ can be expressed in the form
\begin{equation*}
\tilde \vartheta_{ij}(A) = \sum_{k\in J} ((g_{ik}+h_{ik}),A(g_{jk}+h_{jk}))
\end{equation*}
where the $g_{ik},h_{ik} \in \mathbb C^q\otimes L_+^2(0,\infty )$ and
\begin{equation*}
(g_{ik})_j(x) = \delta_{ij}g_k(x)\qquad \text{and} \qquad \sum_{i=1}^q
(h_{ik })_i(x) = 0
\end{equation*}
for $A \in \tilde{\mathfrak{A}} (\mathbb C^p),\medspace x \geq 0,\medspace
i,j \in \{ 1,\cdots ,q\}$ and $k \in J$ a countable index set and the $h_{ik}
\in \mathbb C^q \otimes L^2(0,\infty )$ and if
\begin{equation*}
w_t = \sum_{k\in J} (g_k,\Lambda\vert_tg_k)\qquad \tilde \rho_{ij}(A) =
\sum _{k\in J} (h_{ik},Ah_{jk})
\end{equation*}
then $\tilde \rho$ is bounded so
\begin{equation*}
\sum_{i\in J} \Vert h_{ik}\Vert^2 < \infty\qquad \text{and} \qquad
\sum_{k\in J} (g_k,(I-\Lambda )g_k) < \infty
\end{equation*}
and $1/w_t \rightarrow 0$ as $t \rightarrow 0+$ and $\psi$ satisfies
the conditions
\begin{equation*}
\psi (I_o) \geq \vartheta (I - \Lambda (I_o))\qquad \text{and} \qquad \psi
+ \rho \tilde \Lambda
\end{equation*}
is conditionally negative.
\end{thm}
\begin{proof}  Assume the hypothesis and notation of the theorem.  Now we
define a $\tilde \omega$ which will turn out to be a $q$-weight map over
$\mathbb C ^q$.  Given $\omega$ as stated we define $\tilde \omega$ by
simply restricting $\omega$ to $\tilde{\mathfrak{A}} (\mathbb C^p)$ which we
identify with $\mathfrak{A} (\mathbb C^q)$.  Note each element in
$\tilde{\mathfrak{A}} (\mathbb C^p)$ can be uniquely expressed as a $(q\times
q)$-matrix with entries in $\mathfrak{A} (\mathbb C )$.  Specifically if $A
\in \tilde{\mathfrak{A}} (\mathbb C^p)$ it can be written in the form
\begin{equation*}
A = \sum_{i,j=1}^q E_{ij}A_{ij}
\end{equation*}
with the $A_{ij} \in \mathfrak{A} (\mathbb C )$.  In this was way we can
think of $\tilde \omega$ as an element of $B(\mathbb C^q) \otimes
\mathfrak{A} (\mathbb C )$.  Then one simply checks that $\tilde \omega$ is
a $q$-weight map of index zero over $\mathbb C^q$.  The reason this works
is because in the definition of the generalized boundary representation
\begin{equation*}
\pi_t^\#  = (\iota + \omega\vert_t\Lambda )^{-1}\omega\vert_t  = (\iota +
\phi _t)^{-1}\omega\vert_t = (\iota + \tilde \phi_t)^{-1}\omega\vert_t =
v_t^{-1 }\Theta_t^{-1}\omega\vert_t
\end{equation*}
one only has to compute the inverse above on the range of $\omega$
which allows us to replace $\phi_t$ by $\tilde \phi_t$ in the
above equation. Notice that in the last section $\mathcal{L}$ was
$B(\mathbb C^p)$ so the unit $I_o$ of $\mathcal{L}$ was the unit $I$ of
$B(\mathbb C^p)$.  Now in applying the arguments of the lemmas and theorems 4.5
to 4.7 to our present situation we make the following changes.  We replace the
unit $I$ of $B(\mathbb C^p)$ by the unit $I_o$ of $\mathcal{L}$.  Note the unit
$I$ of $B(\mathbb C^p \otimes L^2(0,\infty ))$ is not replaced so, for example,
the unit $I$ in $\pi^\#_t(I)$ is not replaced.  However, the expression
$\pi^\#_t(\Lambda)$ which is a shorten form of $\pi^\#_t(\Lambda(I))$ is
replaced by $\pi^\#_t(\Lambda(I_o))$.  Notice the inequality $\omega(I-\Lambda)
\leq I$ translates to the inequality $\omega(I-\Lambda(I_o)) \leq I_o$ in our
new setting and this inequality was proved in Theorem 5.9.  Notice that in an
expression like $\psi(I)$ the $I$ should be replaced by $I_o$ since $\psi$ is
a map of $\mathcal{L}$ into itself.  We see that the statement involving $\psi$
in the last sentence of this theorem is simply the translation of the
corresponding statement in the statement of Theorem 4.7.  Note that in the
statement
$$
\psi (I_o) \geq \vartheta (I - \Lambda (I_o))
$$
since $I_o \in \mathcal{L}$ we can replace $\vartheta$ by $\tilde{\vartheta}$
since $I_o \in \mathcal{L}$ so the inequality follows from the tilde
calculations discussed above.

Now we address the results involving $\vartheta$ (without the tilde).  First
$\vartheta$ is defined as $\vartheta = \Theta^{-1}\omega$ so $\vartheta \in
B(\mathbb C^p) \otimes \mathfrak{A} (\mathbb C^p)_*$.  We show $\vartheta$ is
completely positive.  Since $\pi_t^\#  = v_t^{-1}\Theta_t^{-1}\omega\vert_t$ we
have $\vartheta\vert _t = v_t\Theta^{-1}\Theta_t\pi_t^\#$ so for $0 < t \leq s$
we have
\begin{equation*}
\vartheta\vert_s = v_t\Theta^{-1}\Theta_t\pi_t^\#\vert_s.
\end{equation*}
Now taking the limit as $t \rightarrow 0+$ and using the fact that
$\Theta^{-1}\Theta_t \rightarrow \iota$ as $t \rightarrow 0+$ we have
\begin{equation*}
\vartheta\vert_s = \lim_{t\rightarrow 0+} v_t\pi_t^\#\vert_s
\end{equation*}
and since $\vartheta \in B(\mathbb C^p) \otimes \mathfrak{A} (\mathbb C^p)_
*$ so $\vartheta\vert_s(A) \rightarrow \vartheta (A)$ as $s \rightarrow 0+$
for $A \in \mathfrak{A} (\mathbb C^p)$ we have
\begin{equation*}
\vartheta (A) = \lim_{t\rightarrow 0+} v_t\pi_t^\# (A)
\end{equation*}
and since $\pi_t^\#$ is completely positive it follows that
$\vartheta$ is completely positive.  \end{proof}

Our goal is to show that if $\omega$ is a $q$-weight map of index zero then
$\omega$ has a $q$-subordinate where $\mathcal{L}$ the range of $\omega$ is
not only a Choi-Effros factor but a factor which means that if $P$ is the
maximal support projection then $P\mathcal{L} = \mathcal{L}$.  In order to
prove this we will need to show that $\omega (\Lambda (I-P))$ is finite.
Now if $E$ is a minimal $\mathcal{L}$-projection (i.e.  $E\star E = E$ and
$E$ is minimal) then $P(E-E^2) = 0$  so $E-E^2 \leq I-P$ so if $\omega
(\Lambda (I-P))$ is finite then $\omega (\Lambda (E-E^2))$ is finite for
all minimal $\mathcal{L}$-projections $E$.  We will prove this but first we
need a routine lemma for making norm estimates.

\begin{lem}\label{5.11}
Suppose $H$ is a finite dimensional Hilbert space and $\mathcal{L}$ is a
Choi-Effros algebra which is the range of a completely positive contractive
idempotent map $L$ of $B(H)$ into itself and for $A,B \in \mathcal{L}$ the
Choi-Effros product $A\star B = L(AB)$ and $I_o$ is the unit of $\mathcal{L} $.
Suppose $E \in \mathcal{L}$ is an hermitian projection by which we mean $E =
L(E^*E)$.  Suppose $T \in \mathcal{L}$ is positive.  Then
\begin{equation*}
\Vert T\Vert \leq \Vert E\star T\star E\Vert + \Vert (I_o-E)\star T\star (I
_o-E)\Vert .
\end{equation*}
\end{lem}
\begin{proof}  Assume the hypothesis and notation of the lemma.  From the
Choi-Effros theory we know that $\mathcal{L}$ equipped with the
$\star$-multiplication is a finite dimensional $C^*$-algebra and as such it
has a faithful $*$-representation $\pi$ on a finite dimensional Hilbert
space $K$.  Since $\pi$ is faithful $\pi$ preserves norms so $\Vert T\Vert =
\Vert\pi (T)\Vert$.  Then to prove the lemma we need only estimate the norm
of $\pi (T)$.  To simplify notation rather that writing $\pi (T)$ or $\pi (E)$
in our calculations we will simply write $T$ or $E$ (without the $\pi )$.
This means that in estimating the norm of $T$ we will simply consider $T$ to
be a positive operator and $E$ a hermitian projection acting on a finite
dimensional Hilbert space.  Then to prove the lemma we need to prove the
estimate
\begin{equation*}
\Vert T\Vert \leq \Vert ETE\Vert + \Vert (I-E)T(I-T)\Vert .
\end{equation*}
Let
\begin{equation*}
A = ETE,\qquad B = ET(I-E),\qquad B^* = (I-E)TE,\qquad C  = (I-E)T(I-E).
\end{equation*}
Note that $T = A + B + B^* + C$.  Now let $H_1 = Range(E)$ and $H_2 =
Range(I-E)$.  We can now write $T$ as a $(2\times 2)$-matrix
\begin{equation*}
T = \left[\begin{matrix} A&B
\\
B^*&C
\end{matrix} \right]
\end{equation*}
where the entries $T_{ij}$ of $T$ map $H_i$ into $H_j$.  To prove the lemma
we show $\Vert T\Vert \leq \Vert A\Vert + \Vert C\Vert$.

Since $T \geq 0$ we have $(F,TF) \geq 0$ where $F = \{ f,zg\}$ with
$f \in H_1, g \in H_2,$ and $z \in \mathbb C$ which yields
\begin{equation*}
(f,Af) + \vert z\vert^2(g,Cg) + 2Re(z(f,Bg)) \geq 0.
\end{equation*}
Note that if either $(g,Cg) = 0$ or $(f,Af) = 0$ we have $(f,Bg) = 0$
otherwise you can find a value of $z$ so that the expression is less than zero.
Now assuming $(g,Cg) \neq 0$ we find by setting $z = -(f,Bg)/(g,Cg)$ that
\begin{equation*}
(f,Af)(g,Cg) \geq \vert (f,Bg)\vert^2.
\end{equation*}
There is a sequence of unit vectors $f_n,g_n \in H_1 \oplus H_2$ so
that $(f_n,Bg_n) \rightarrow \Vert B\Vert$.  Then we have
\begin{equation*}
\Vert A\Vert\cdot\Vert C\Vert \geq (f_n,Af_n)(g_n,Cg_n) \geq \vert
(f_n,Bg_n) \vert^2 \rightarrow \Vert B\Vert^2
\end{equation*}
as $n \rightarrow \infty$ so $\Vert A\Vert\cdot\Vert C\Vert \geq \Vert
B\Vert ^2$.  Now if $F = \{ f,g\}$ we have
\begin{align*}
(F,TF) &= (f,Af) + (g,Cg) + 2Re((f,Bg))
\\
&\leq \Vert A\Vert\cdot\Vert f\Vert^2 + \Vert C\Vert\cdot\Vert g\Vert^2 + 2
\Vert B\Vert\cdot\Vert f\Vert\cdot\Vert g\Vert
\\
&\leq \Vert A\Vert\cdot\Vert f\Vert^2 + \Vert C\Vert\cdot\Vert g\Vert^2 + 2
\Vert A\Vert^{\tfrac{1}{2}}\Vert C\Vert^{\tfrac{1}{2}}\Vert f\Vert\cdot\Vert
g\Vert
\\
&= (\Vert A\Vert^{\tfrac{1}{2}}\Vert f\Vert + \Vert C\Vert^{\tfrac{1}{2}}\Vert
g\Vert )^2.
\end{align*}
Maximizing this expression subject to the constraint $\Vert f\Vert^2 +
\Vert g\Vert^2 = 1$ we find the maximum occurs when
\begin{equation*}
\Vert f\Vert = \Vert A\Vert^{\tfrac{1}{2}} (\Vert A\Vert+\Vert C\Vert
)^{-\tfrac{1}{2} }\qquad \text{and} \qquad \Vert g\Vert = \Vert
C\Vert^{\tfrac{1}{2}} (\Vert A\Vert +\Vert C\Vert )^{-\tfrac{1}{2}}
\end{equation*}
which gives $(F,TF) \leq \Vert A\Vert + \Vert C\Vert$ and since $\Vert
T\Vert$ is the supremum of $(F,TF)$ for all unit vectors $F \in H_1 \oplus H_2$
we have $\Vert T\Vert \leq \Vert A\Vert + \Vert C\Vert .$ \end{proof}

\begin{thm}\label{5.12}
Suppose $\omega$ is a $q$-weight map over $\mathbb C^p$ of index zero and
$\mathcal{L}$ the range $\omega$ is a Choi-Effros factor of type~I$_q$.
Suppose $E$ is a minimal $\mathcal{L}$-projection so $E \in \mathcal{L}$
and $E\star E = E$.  Let $F$ be the support projection for $E$ so $F$
is smallest projection in $B(\mathbb C^p)$ with $F \geq E$.  Then $\omega
(\Lambda (F-E))$ is finite meaning there is a constant $K$ so that
$\Vert\omega\vert_t (\Lambda (F-E))\Vert \leq K$ for all $t > 0$. Note that
$F-E \geq E - E^2$ so we have shown that $\omega(\Lambda(E-E^2))$ is finite.
\end{thm}
\begin{proof}  Assume the hypothesis and notation of the theorem.  Suppose $E$
is a minimal $\mathcal{L}$-projection.  Assume $\psi ,\medspace \vartheta$ and
the $g_{ik}$ and $h_{ik}$ are defined as in Theorem 5.10.  We will show
$\Vert\vartheta\vert_t(\Lambda (F-E))\Vert$ is bounded for $t > 0$ and since
$\omega =  \psi^{-1}\vartheta$ we have $\Vert\omega\vert_t(\Lambda (F-E))\Vert
\leq \Vert \psi^{-1}\Vert\medspace \Vert\vartheta\vert_t(\Lambda (F-E))\Vert$
so we will have proved the theorem.  In the remainder of this proof we will
assume $t > 0$.

Recall $F$ is the support projection for $E$ the smallest projection in
$B(\mathbb C^p)$ with $F \geq E$ so $F$ is not necessarily in $\mathcal{L}$.
We begin by obtaining estimates of $\Lambda (F - E)$ in terms of elements of
$\mathcal{L}$. Since $I \geq F$ we have
\begin{equation*}
\Lambda (F-E) \leq \Lambda (I-E) = \Lambda (I-I_o) + \Lambda (I_o-E) \leq I
- \Lambda (I_o) + \Lambda (I_o-E)
\end{equation*}
so we have
\begin{equation*}
\vartheta\vert_t\Lambda (F-E) \leq \vartheta\vert_t(I-\Lambda (I_o)) +
\vartheta \vert_t\Lambda (I_o-E)
\end{equation*}
and since $\vartheta (I - \Lambda (I_o)) \leq \psi (I_o)$ we have
\begin{equation*}
\vartheta\vert_t\Lambda (F-E) \leq \vartheta\vert_t \Lambda (I_o-E)+\psi (I_o).
\end{equation*}

Next we get an estimate in terms of $\Lambda (E)$.  Let $\lambda_i \in [0,1]$
for $i = 0,1,\cdots ,m$ be the spectrum of $E$ in increasing order so
$\lambda_o = 0$ and $\lambda_m = 1$ so $\lambda_1$ is the smallest positive
eigenvalue of $E$.  Note the spectrum of $F-E$ consists of the numbers
$\lambda_o = 0$ and $1 - \lambda_i$ for $i = 1,2,\cdots ,m-1$.  Note we do not
include $\lambda_m$ since $\lambda_m = 1$ so $1 - \lambda_m = 0$ and this has
already been listed in the spectrum of $F-E$.  Let $\kappa$ be a positive real
number which we will define shortly.  Then the spectrum of $\kappa E - (F-E) =
(\kappa+1)E -F$ consists of the numbers 0 and $(\kappa+1)\lambda_i-1$ for $i =
1,\cdots ,m$.  Now let
\begin{equation*}
\kappa = \lambda_1^{-1}-1
\end{equation*}
and we see the spectrum of $\kappa E - (F-E)$ consist of the numbers 0 and
$(\lambda_i/\lambda_1-1)$ for $i = 2,\cdots ,m$.  Note we do not include $i
= 1$ since $(\lambda_1/\lambda_1-1) = 0$ and this has already been listed.
Hence, the spectrum of $\kappa E - (F-E)$ is non negative so we have
\begin{equation*}
\kappa E \geq F-E
\end{equation*}
and, hence, we have
\begin{equation*}
\vartheta\vert_t\Lambda (F-E) \leq \kappa\vartheta\vert_t\Lambda (E).
\end{equation*}

Now that we have above two estimates for $\vartheta\vert_t\Lambda (F-E)$ in
terms of elements of $\mathcal{L}$ we can use the formula for
$\tilde{\vartheta}$ in terms of the $g's$ and $h's$ in Theorem 5.10.  First
note that the tilde mappings are mapping of $B(\mathbb C^q)$ into itself
and so in our calculation we will think of $B(\mathbb C^q)$ as
$(q\times q)$-matrices.  Recalling the formula for $\tilde{\vartheta}$ in
Theorem 5.10 we find that for $A \in B(\mathbb C^q)$ representing an element
of $\mathcal{L}$
\begin{equation*}
\tilde{\vartheta}\vert_t\Lambda (A) = w_tA + Y_tA + AY_t^* +
\tilde{\rho}\vert _t\Lambda (A)
\end{equation*}
where $\tilde{\rho}$ is completely positive and uniformly bounded and $1/w_t
\rightarrow 0$ as $t \rightarrow 0+$ and $Y_t \in B(\mathbb C^q)$ is given by
\begin{equation*}
(Y_t)_{ij} = \int_t^\infty e^{-x}  \overline {(h_{ik})_j(x)} g_k(x) \,dx
\end{equation*}
and from the condition on the $h's$ it follows that $Y_t$ is of trace zero.
There are further properties of $\tilde{\rho}\vert_t\Lambda$ but we only need
that it is bounded meaning there is a constant $K$ so that $\tilde{\rho}\vert_t
\Lambda (I_o) \leq KI_o$ for all $t > 0$.  Now since $\vartheta\vert_t\Lambda
(F-E) \leq \kappa\vartheta \vert_t\Lambda (E)$ we have
\begin{equation*}
\vartheta\vert_t\Lambda (F-E) \leq \kappa w_tE + \kappa Y_tE + \kappa EY_t^
* + \kappa\tilde{\rho}\vert_t\Lambda (E)
\end{equation*}
and since $\vartheta\vert_t\Lambda (F-E) \leq \vartheta\vert_t\Lambda
(I_o-E) + \psi (I_o)$ we have
\begin{equation*}
\vartheta\vert_t\Lambda (F-E) \leq w_t(I_o-E) + Y_t(I_o-E) +(I_o-E)Y_t^* +
\tilde{\rho}\vert_t\Lambda (I_o-E) + \psi (I_o).
\end{equation*}
Note $\vartheta\vert_t\Lambda (F-E) \in \mathcal{L}$ where $\mathcal{L}$ is
the Choi-Effros algebra.  We will estimate its norm using the previous lemma.
Then sandwiching the bottom inequality between $E$ on the right and left using
the Choi-Effros product $A\star B = L(AB)$ we have
\begin{equation*}
E\star\vartheta\vert_t\Lambda (F-E)\star E \leq E\star\tilde{\rho}\vert_t
\Lambda (I_o-E)\star E +
E\star \psi (I_o) \star E \leq (K+\Vert\psi (I_o)\Vert )E
\end{equation*}
and sandwiching the top inequality between $(I_o-E)$ on the right and left
we have
\begin{equation*}
(I_o-E)\star \vartheta\vert_t\Lambda (F-E) \star (I_o-E) \leq (I_o-E)\star
\kappa \tilde{\rho}\vert_t \Lambda (E) \star (I_o-E) \leq \kappa K(I_o-E).
\end{equation*}
Now if $0 \leq A \leq B$ we have $0 \leq \Vert A \Vert \leq \Vert B \Vert$
which gives us the estimates
\begin{equation*}
\Vert E\star\vartheta\vert_t\Lambda (F-E)\star E \Vert
\leq (K+\Vert\psi (I_o)\Vert ) \qquad \text{and} \qquad \Vert (I_o-E)\star
\vartheta\vert_t\Lambda (F-E) \star (I_o-E) \Vert \leq \kappa K
\end{equation*}
so by the previous lemma we have
\begin{equation*}
\Vert\vartheta\vert_t\Lambda (F-E)\Vert \leq (1+\kappa)K + \Vert\psi (I_o)\Vert
\end{equation*}
and this bound in independent of $t.$ \end{proof}

In the previous theorem we showed that $\omega (\Lambda (E-E^2))$ is finite
for $E$ a minimal $\mathcal{L}$-projection.  Next we will show that if
$\omega (\Lambda (E-E^2))$ is finite for all minimal
$\mathcal{L}$-projections then $\omega (\Lambda (I-P))$ is finite where $P$
is the maximal support projection for $\mathcal{L}$.  Our strategy is as
follows.  Note that $P(I-I_o) = 0$ and $P(E-E^2) = 0$ for all minimal
projections $E \in \mathcal{L}$.  We define a projection $Q$ as the largest
projection in $B(\mathbb C^p)$ with these properties and then show that $Q
= P$.  We will restrict our attention to the case where $\mathcal{L}$ is a
Choi-Effros factor but we believe the results are valid in the more general
case.

Here we introduce some notation that we will be using for calculations.

\begin{defn}\label{5.13}
If $A \in B(\mathbb C^p)$ is positive we denote by
\begin{equation*}
A^+ = \lim_{n\rightarrow\infty} A^{1/n}
\end{equation*}
the support projection of $A$.  We denote by $A^-$ the projection
on the subspace spanned by the eigenvectors of $A$ of eigenvalue
$\lambda \geq 1$.
\end{defn}

Here we present various properties of $A^+$ and $A^-$ that we will use in
our calculations.  We assume $A \in B(\mathbb C^p)$ and $A \geq 0$.  First
note that since $\mathbb C^p$ is finite dimensional $A^+$ and $A^-$ can be
expressed as polynomials in $A$.  Note that if $C$ commutes with $A$ then
$C$ commutes with $A^+$ and $A^-$.  Note that if $Q$ is a projection then
$QA = 0$ if and only if $QA^+ = 0$.  Note that if $A \geq 0$ and $\Vert
A\Vert \leq 1$ then
\begin{equation*}
A^- = \lim_{n\rightarrow\infty} A^n.
\end{equation*}

Note if $A_1,A_2,\cdots ,A_m$ are positive operators in $B(\mathbb C^p)$
then the projection onto the subspace spanned by the ranges of the $A_i$ is
$(A_1+A_2+\cdots +A_m)^+$.  If $A_i \in B(\mathbb C^p)$ and $0 \leq A_i
\leq I$ for $i = 1,\cdots ,m$ then the projection onto the intersection of
the ranges of the $A_i^-$ is ((1$/m)(A_1+A_2+\cdots +A_m))^-$.

\begin{lem}\label{5.14}
Suppose $Q \in B(\mathbb C^p)$ is an hermitian projection and $A
\in B(\mathbb C^p)$ is positive then $QA = 0$ if and only if $QAQ
= 0$.  Suppose further that $A \in B(\mathbb C^p)$ is positive and
of norm one, ($A \geq 0$ and $\Vert A\Vert = 1)$.  Then the
following statements are equivalent.
\begin{enumerate}[(i)]
\item $Q(A-A^2) = 0$ \item $Q(A-A^2)^+ = 0$ \item $Q(A^+-A^-) = 0$ \item
$Q(A-A^-) = 0$  \item $Q(A^+-A) = 0$
\end{enumerate}
\end{lem}
\begin{proof}  Assume the hypothesis and notation of the theorem. We prove
$QA=0$ if and only if $QAQ = 0$.  Now if $QA=0$ then multiplying by $Q$ on
the right we have $QAQ = 0$.  Now if $QAQ=0$ we have $XX^* = 0$ where
$X = QA^{\tfrac{1}{2}}$ from which we conclude that $X = 0$ from which we
conclude $XA^{\tfrac{1}{2}} = QA = 0$.

Now to prove the equivalence of the five conditions of the lemma we further
assume that $A$ is of norm one.  From the spectral decomposition of $A$ we know
\begin{equation*}
A = \sum_{i=1}^m \lambda_iF_i
\end{equation*}
with $0 < \lambda_1 < \lambda_2 < \cdots  < \lambda_{m-1} < \lambda_m = 1$
and the $F_i$ are mutually orthogonal projections.  Then
\begin{equation*}
A - A^2 = \sum_{i=1}^{m-1} (\lambda_i-\lambda_i^2)F_i\qquad (A^+-A^-) =
(A-A ^2)^+= \sum_{i=1}^{m-1} F_i
\end{equation*}
and
\begin{equation*}
A - A^- = \sum_{i=1}^{m-1} \lambda_iF_i \qquad
A^+ -A = \sum_{i=1}^{m-1} (1-\lambda_i)F_i.
\end{equation*}
Since in all cases the coefficients of the $F_i$ are positive we have $Q$
times any of these expressions is zero if and only if $QF_i = 0$ for $i =
1,\cdots ,m-1$.  Hence, the five conditions above are equivalent.
\end{proof}

\begin{lem}\label{5.15}
Suppose $E \in B(\mathbb C^p)$ is positive and of norm one, $(E \geq 0$ and
$\Vert E\Vert = 1)$ and $T \in B(\mathbb C^p)$ satisfies $0 \leq T \leq
E$.  Suppose $Q \in B(\mathbb C^p)$ is an hermitian projection.  If
$Q(E-E^-) = 0$ and $Q(E-T) = 0$ then $Q(E^+-T^-) = Q(E-T^-) =0 $.  Note if
$\Vert T\Vert < 1$ then $T^- = 0$ so $QE = QE^+ = 0$.
\end{lem}
\begin{proof}  Assume the hypothesis an notation of the lemma.
We show there is a constant $b \geq 1$ so that $b(E^+- T) \geq E^-- T^-$.
Note that the range of $E^-$ is the set of vectors $f \in \mathbb C^p$ so
that $(f,Ef) = (f,f)$ and the range of $T^-$ is the set of vectors $f \in
\mathbb C ^p$ so that $(f,Tf) = (f,f)$.  Since $E \geq T$ the range of
$T^-$ is contained in the range of $E^-$ so $T^- \leq E^-$ and $E^-T^- =
T^-$.  Since
\begin{equation*}
I \geq E^+ \geq E \geq E^- \geq T^-\qquad \text{and} \qquad
I \geq E^+ \geq E \geq T \geq T^-
\end{equation*}
and $T^-$ is a projection so its eigenvalues are zero and one it follows
that $T^-$ commutes with $E^+,\medspace E,\medspace E^-, \medspace T$ and
$T^-$ and the product of any of these operators with $T^-$ is $T^-$.  Note
that $E^+$ is a unit for all of these operators and $E^+ -T^-$ is a unit
for the difference of any two of these operators.

Now let $c = \Vert T-T^-\Vert$ and note $c$ is the largest eigenvalue of
$T$ that is less than one so $0 \leq c < 1$.  Now we have
\begin{equation*}
E^+-T = (E^+-T^-) - (T-T^-) \geq (E^+-T^-) - \Vert T-T^-\Vert (E^+-T^-) =
(1 -c)(E^+-T^-).
\end{equation*}
Now let $b = (1-c)^{-1}$ and we have $b(E^+-T) \geq E^+-T^-$.

In the statement of the lemma we are given $Q(E - E^-) = 0$ and from the
previous lemma this implies $Q(E^+-E) = 0$.  Also we are given in the
statement of the lemma that $Q(E-T) = 0$.  Adding these two equations we
have $Q(E^+-T) = 0$ and since $b(E^+-T) \geq E^+-T^-$ we have $Q(E^+-T^-) = 0$.
Since $E^+-T^-\geq E-T^-\geq 0$ we have $Q(E-T^-) = 0.$
\end{proof}

\begin{lem}\label{5.16}
Suppose $\mathcal{L} \subset B(\mathbb C^p)$ is a Choi-Effros factor of
type~I$_q$ and $E_{ij}$ are a complete set of matrix units for $\mathcal{L}$.
Then if $i \neq j$
\begin{equation*}
E_{ii}^-E_{jk} = 0\qquad \text{and} \qquad E_{kj}E_{ii}^- = 0.
\end{equation*}
\end{lem}
\begin{proof}  Assume the hypothesis and notation of the lemma.  First we
prove $E_{11}^-E_{22} = 0$.  Since $\Vert E_{11}+E_{22}\Vert = 1$
if $f \in \mathbb C^p$ is a unit vector so that $E_{11}^-f = f$ we
have $(f,(E_{11}+E_{22})f) = 1 + (f,E_{22}f) \leq 1$ so
$(f,E_{22}f) = 0$ and $f$ is orthogonal to the range of $E_{22}$.
Hence $E_{22}E_{11}^- = 0$ so $E_{11}^-E_{22}=0$.  Let $L$ be a
completely positive contractive idempotent with range
$\mathcal{L}$ then $E_{22} = L(E_{22}) = L(E_{2i}E_{i2})$ and from
the Schwarz inequality for completely positive maps we have (note
$\Vert L \Vert = 1$)
\begin{equation*} E_{2k}E_{k2} = L(E_{k2}^*)L(E_{k2})
\leq L(E_{2k}E_{k2}) = E_{22}
\end{equation*}
so $E_{11}^-E_{2k}E_{k2} = 0$.  By the polar decomposition in a finite
dimensional Hilbert space we have $E_{2k} = (E_{2k}E_{k2})^{\frac {1} {2}} U$
where $U$ is a unitary so the range of $E_{2k}$ is the range of
$(E_{2k}E_{k2})^{\frac {1} {2}}$ and since $\mathbb C^p$ is finite dimensional
the range of $(E_{2k}E_{k2})^{\frac {1} {2}}$ is the range of $E_{2k}E_{k2}$.
Since $E_{11}^-f=0$ for $f$ in the range of $E_{2k}E_{k2} = 0$ it follows that
$E_{11}^-f = 0$ for $f$ in the range of $E_{2k}$.  Hence, $E_{11}^-E_{2k} =
0$. The proof for general $i,j,k$ with $i \neq j$ follows from replacing 1
and 2 by $i$ and $j$.  Taking adjoints we obtain the second equalities.
\end{proof}

\begin{lem}\label{5.17}
Suppose $\mathcal{L} \subset B(\mathbb C^p)$ is a Choi-Effros factor of
type~I$_q$.  Let $Q$ be the largest projection so that $Q(I-I_o) = 0$ and
$Q(E-E^2) = 0$ for all minimal $\mathcal{L}$-projections $E \in \mathcal{L}$.
Then if $E_{ij}$ are a complete set of matrix units for $\mathcal{L}$ then
$Q(E_{ii}-E_{ij}E_{ji}) = 0$ and $Q(E_{ii}^+-(E_{ij}E _{ji})^-) = 0$ for
$i,j = 1,\cdots ,q$
\end{lem}
\begin{proof}  Assume the hypothesis and notation of the theorem.  We prove
the lemma for $i,j = 1,2$.  Suppose $x$ and $\theta$ are real numbers with
$x^2 \leq 1$ (where we think of $x$ as a real variable and $\theta$ as a
constant) and
\begin{equation*}
E_x = s_x^2E_{11}+xe^{i\theta}s_xE_{12}+xe^{-i\theta}s_xE_{21}+x^2E_{22}\qquad
\text{with} \qquad s_x=\sqrt{1-x^2}.
\end{equation*}
Since $E_x$ is a minimal $\mathcal{L}$-projection we have $Q(E_x-E_x^2)=0.
$ Making a power series expansion in $x$ we have
\begin{equation*}
Q(E_x-E_x^2) = Q(A_o+xA_1+x^2A_2+\cdots) = 0.
\end{equation*}
Since this expression is real analytic in $x$ (meaning it has a convergent
power series expansion for $-1<x<1$) it follows that each of the
coefficients $QA_n$ must vanish so $QA_2 = 0$ and calculating $A_2$ we find
\begin{equation*}
QA_2 = Q(2E_{11}^2-E_{11}-E_{12}E_{21}+E_{22}-E_{21}E_{12}-E_{11}E_{22}-
E_{22}E_{11}-e^{2i\theta}E_{12}E_{12}-e^{-2i\theta}E_{21}E_{21}) = 0.
\end{equation*}
Averaging over $\theta$ from $\theta=-\pi$ to $\theta=+\pi$ the last two terms
average to zero.  From the pervious lemma we have $Q(E_{11}E_{22}) =
Q(E_{22}E_{11})= 0$ so we have
\begin{equation*}
Q(2E_{11}^2-E_{11}-E_{12}E_{21}+E_{22}-E_{21}E_{12}) = 0.
\end{equation*}
Since $E_{11}$ is a minimal projection in $\mathcal{L}$ we also have $Q(E_{11}-
E_{11}^2) = 0$ and adding twice this equation to the above equation and
multiplying by $Q$ on the right we find
\begin{equation*}
Q(E_{11}-E_{12}E_{21})Q + Q(E_{22}-E_{21}E_{12})Q = 0.
\end{equation*}
As we saw in the previous lemma we have from the Schwarz inequality
$E_{22}\geq E_{21}E_{12}$ and $E_{11}\geq E_{12}E_{21}$ so the expression above
is the sum of two positive terms so both terms must be zero.  From Lemma 5.14
we know that if $A \in B(\mathbb C^p)$ is positive and $QAQ = 0$ then $QA=0$
from which we conclude
\begin{equation*}
Q(E_{11}-E_{12}E_{21}) = 0\qquad \text{and} \qquad Q(E_{22}-E_{21}E_{12}) = 0.
\end{equation*}
Since $Q(E_{11}-E_{11}^2) = Q(E_{22}-E_{22}^2) = 0$ we have proved the lemma
for $i,j = 1,2$.  The argument we have just given will apply to any two values
for $i$ and $j$ from 1 to $q$. Now if $E = E_{ii}$ and $T = E_{ij}E_{ji}$ then
we have shown $Q(E-E^2)=0$ and $Q(E-T)=0$.  Lemma 5.14 shows that $Q(E-E^2)=0$
implies $Q(E-E^-)=0$ so we now have shown that $Q(E-E^-)=0$ and $Q(E-T)=0$ and
Lemma 5.15 shows these two conditions imply $Q(E^+-T^-)=0$.  Hence, we have
shown $Q(E_{ii}^+-(E_{ij}E_{ji})^-) = 0.$
\end{proof}

In the next lemma we define a projection $Q_1$ which will turn out to be
the maximal support projection.  The lemma is also a definition.

\begin{lem}\label{5.18}
Suppose $\mathcal{L}$ is a Choi-Effros factor of type~I$_q$ and $E_{ij}$
are a complete set of matrix units for $\mathcal{L}$ and let $I_o$ be the unit
of $\mathcal{L} $.  Let $Q$ be the largest projection so that $Q(I-I_o) = 0$
and $Q(E-E^2) = 0$ for every minimal projection $E \in \mathcal{L}$.  For
each $i = 1,\cdots ,q$ let $T_i$ be the projection onto the intersection of
the ranges of $(E_{ij}E_{ji})^-$ for $j = 1,\cdots ,q$ and let $Q_1 =
T_1+T_2+\cdots +T_q$.  Then $Q_1 \geq Q$.
\end{lem}
\begin{proof}  Assume the hypothesis and notation of the lemma.  From the
previous lemma we have $Q(E_{11}-(E_{1i}E_{i1})^-) = 0$ for $i = 1,\cdots
,q$ and, hence,
\begin{equation*}
Q(E_{11}-R_1) = 0\qquad \text{where} \qquad R_1 = \frac {1} {q} \sum_{i=1}^q
(E_{1i}E_{i1})^-
\end{equation*}
and then we have from Lemma 5.15 that $Q(E_{11}-R_1^-) = 0$ and $R_1^-$ is
$T_1$ so $Q(E_{11}-T_1) = 0$.  Repeating this argument with 1 replaced by $i$
yields the result that $Q(E_{ii}-T_i) = 0$ for $i = 1,\cdots ,q$.  Since
the sum of the $E_{ii}$ is $I_o$ we have $Q(I_o-Q_1) = 0$ and since
$Q(I-I_o) = 0$ it follows that
$Q(I-Q_1) = 0$.  Hence $Q_1 \geq Q.$
\end{proof}

\begin{thm}\label{5.19}
Suppose $\mathcal{L} \subset B(\mathbb C^p)$ is a Choi-Effros factor of
type~I$_q$ with unit $I_o$ and $Q$ is the largest projection so that
$Q(I-I_o) = 0$ and $Q(E-E^2) = 0$ for all minimal projections $E \in
\mathcal{L}$.  Then $Q = P$ where $P$ is the maximal support projection for
$\mathcal{L}$.
\end{thm}
\begin{proof}  Assume the hypothesis and notation of the theorem.  Let
$E_{ij}$ be a complete set of matrix units for $\mathcal{L}$ and let
$T_i$ be the projection onto the intersection of the
ranges of $(E_{ij}E_{ji})^-$ for $j = 1,\cdots ,q$.  Let
$g_1,g_2,\cdots ,g_r$ be an orthonormal basis for the range of
$T_1$.  Let $f_{ik}= E_{i1}g_k$ for $i = 1,\cdots ,q$ and $k =
1,\cdots ,r$.
We show that if $i \neq j$ then $E_{mi}f_{jk} = 0$.
First note that since $T_1g_k = g_k$ we have
$(g_k,E_{1j}E_{j1}g_k) = 1$ so $E_{1j}E_{j1}g_k = g_k$ so
\begin{equation*}
1 = (E_{j1}g_k,E_{j1}E_{1j}E_{j1}g_k) \leq (E_{j1}g_k,E_{jj}E_{j1}g_k) \leq 1
\end{equation*}
so $(E_{j1}g_k,E_{jj}^-E_{j1}g_k) = 1$ so $E_{jj}^-E_{j1}g_k = E_{j1}g_k$.
Then we have
\begin{equation*}
E_{mi}f_{jk} = E_{mi}E_{j1}g_k = E_{mi}E_{jj}^-E_{j1}g_k
\end{equation*}
and since $i \neq j$ we have from Lemma 5.16 that $E_{mi}E_{jj}^- = 0$ and,
hence, $E_{mi}f_{jk} = 0$ for $i \neq j$.

Now since $f_{ik}= E_{i1}g_k$ we have for $i \neq j$ that
$$
(f_{ik},f_{jm}) = (E_{i1}g_k,f_{jm}) = (g_k,E_{1i}f_{jm}) = 0
$$
for $k,m=1,\cdots,r$ and for $i =j$ we have
$$
(f_{ik},f_{im}) = (E_{i1}g_k,E_{i1}g_m) = (g_k,E_{1i}E_{i1}g_m) = (g_k,g_m)
= \delta_{km}
$$
so the $f_{ik}$ form an orthonormal set of vectors.

Now let
\begin{equation*}
C = \sum_{i,j=1}^q E_{ij}.
\end{equation*}
Note $C$ is positive so if $f = z_1f_{1k} + z_2f_{2k} + \cdots  z_qf_{qk}$
for $z_i$ complex numbers then $(f,Cf) \geq 0$ and calculating this using
the above relations we find
\begin{equation*}
(f,Cf) = \sum_{i,j=1}^q \overline {z_i}z_j(f_{ik},E_{ij}f_{jk})
\end{equation*}
so $(f,Cf) \geq 0$ for all choices of the $z_i$ if and only if the matrix
$A$ with matrix elements
\begin{equation*}
a_{ij} = (f_{ik},E_{ij}f_{jk})
\end{equation*}
is positive.  Now $a_{1i} = a_{i1} = 1$ and $a_{ii} \leq 1$ for $i =
1,\cdots ,q$.  Looking at the $(2\times 2)$-matrix
\begin{equation*}
\left[\begin{matrix} a_{11}&a_{1i}
\\
a_{i1}&a_{ii}
\end{matrix} \right] =
\left[\begin{matrix} 1&1
\\
1&a_{ii}
\end{matrix} \right] \geq 0
\end{equation*}
and given $a_{ii} \leq 1$ we see the determinant is non negative if and
only if $a_{ii} = 1$ and looking at the matrix
\begin{equation*}
\left[\begin{matrix} a_{11}&a_{1i}&a_{1j}
\\
a_{i1}&a_{ii}&a_{ij}
\\
a_{j1}&a_{ji}&a_{jj}
\end{matrix} \right]=
\left[\begin{matrix} 1&1&1
\\
1&1&a_{ij}
\\
1&a_{ji}&1
\end{matrix} \right] \geq 0
\end{equation*}
we see this matrix has a non negative determinant if and only if $a_{ij} =
a_{ji} = 1$.  It then follows that since $A$ is positive we have $a_{ij} =
1$ for all $i$ and $j$.  Note this is true for all $k = 1,\cdots ,r$.
Hence, we have
\begin{equation*}
E_{ij}f_{mk} = \delta_{jm}f_{ik}
\end{equation*}
so if $Q_1$ is the projection onto the span of the $f_{ik}$ then $Q_1$
commutes with the $E_{ij}$ and $\mathcal{L}$ restricted to the range of
$Q_1$ is an actual type~I$_q$ factor.

Note we have $E_{2i}Q_1E_{i2} = Q_1E_{22}$ so $T_2$ which is the
largest projection so that $T_2 \leq (E_{2i}E_{i2})^-$ for $i =
1,\cdots ,q$ satisfies the relation $T_2 \geq Q_1E_{22}$.  Note
$\Vert E_{21}T_1f\Vert = \Vert T_1f\Vert$ so it follows that
$\dim(E_{21}T_1E_{12}) = \dim(T_1)$.  Since $T_2 \geq
E_{21}T_1E_{12}$ it follows that $\dim(T_2) \geq \dim(T_1)$.  But
we can repeat all of the above arguments with the indices 1 and 2
interchanged and reach the conclusion that $\dim(T_1) \geq
\dim(T_2)$ so we conclude that $\dim(T_1) = \dim(T_2)$.  Since
$T_2 \geq E_{21}T_1E_{12}$ and $\dim(E_{21}T_1E_{12}) = \dim(T_1)$
we have $T _2=E_{21}T_1E_{12}$. And replacing 1 and 2 by indices
$i$ and $j$ we have $T_i = E_{ij}T_jE_{ji}$.

Now recall that $Q_1$ is the projection onto the span of the $f_{ik}$ for
$i = 1,\cdots ,q$ and $k = 1,\cdots ,r$.  Now $T_i$ is the projection onto
the span of the $f_{ik}$ for $k = 1,\cdots ,r$ so
\begin{equation*}
T_1 + T_2 + \cdots  + T_q = Q_1
\end{equation*}
and $Q_1$ is the projection $Q_1$ of Lemma 5.18 so $Q_1 \geq Q$ and since
$P$ satisfies all the conditions defining $Q$ we have $Q \geq P$.  Now we
will produce a completely positive contractive idempotent map of $B(\mathbb
C^p)$ into itself with range $\mathcal{L}$ and support projection $Q_1$.
Let
\begin{equation*}
L(A) = \frac {1} {r}  \sum_{i,j=1}^q \sum_{k=1}^r (f_{ik},Af_{jk})E_{ij}
\end{equation*}
for $A \in B(\mathbb C^p)$.  Note $L(E_{ij}) = E_{ij}$ so it is clear that
the range of $L$ is $\mathcal{L}$
and $L(L(A)) = L(A)$, $L(I) = I_o$ and $Q_1$ is the smallest projection
so that $L(Q_1) = I_o$.  Now let $F_k$ be the projection onto the vectors
$f_{ik}$ for $i = 1,\cdots ,q$.  Now $F_kAF_k$ is a $(q\times q)$ matrix
with matrix elements $(f_{ik},Af_{jk })$.  Clearly the mapping $A
\rightarrow F_kAF_k$ is completely positive and, hence, the sum of the maps
\begin{equation*}
\phi (A) = \frac {1} {r} \sum_{k=1}^r F_kAF_k
\end{equation*}
is a completely positive map from $B(\mathbb C^p)$ to the $(q\times
q)$-matrices and since the map from $(q\times q)$ matrices $A = \{
a_{ij}\}$ given by
\begin{equation*}
\psi (A) = \sum_{i,j=1}^q a_{ij}E_{ij}
\end{equation*}
is completely positive it follows that the map $L = \psi\phi$ is completely
positive.  Hence, $Q_1$ is the support projection for a completely positive
contractive idempotent mapping of $B(\mathbb C^p)$ into itself with range
$\mathcal{L}$.  Hence, by Theorem 5.2 we have $P \geq Q_1$.  Hence, we
have $Q_1 \geq Q \geq P \geq Q_1$ so $P = Q.$ \end{proof}

\begin{thm}\label{5.20}
Suppose $\omega$ is a $q$-weight map over $\mathbb C^p$ of index zero and
$\mathcal{L}$ the range $\omega$ is a Choi-Effros factor of type~I$_q$ and
suppose $P$ is the maximal support projection for $\mathcal{L}$.  Then
$\omega (\Lambda (I-P))$ is finite meaning there is a constant $K$ so that
$\Vert\omega\vert_t(\Lambda (I-P))\Vert \leq K$ for all $t > 0$.
\end{thm}
\begin{proof}  Assume the hypothesis and notation of the theorem.  Then from
the previous theorem we know that $P$ is the largest projection so that
$P(I-I_o) = 0$ and $P(E-E^2) = 0$ for all minimal projections.  This means
that $I-P$ is the projection onto the space spanned by the ranges of
$I-I_o$ and $E-E^2$ for all minimal projections $E \in \mathcal{L}$.  If
$I-I_o \neq 0$ let $\lambda$ be the smallest positive eigenvalue of $I-I_o$
otherwise let $\lambda = 1$.  Then
$\lambda^{-1}(I-I_o) \geq (I -I_o)^+ = R_o$.  If $E - E^2 \leq R_o$ for all
minimal projections then $R_o = I-P$.  If this is not the case then there
is a minimal projection $E_1$ so that this is not true.  Then let $R_1 =
(R_o+(E_1-E_1^2)^+)^+$ and note $R_1 \leq \lambda^{-1
}(I-I_o)+(E_1-E_1^2)^+$.  If $E-E^2 \leq R_1$ for all minimal projections
$E$ then $R_1 = I-P$ and if not we can find a minimal projection $E_2$ and
form the bigger projection $R_2 = (R_1+(E_2-E_2^2)^+)^+$ and note $R_2 \leq
\lambda^{-1}(I-I_o) + (E_1 -E_2^2)^+ + (E_2-E_2^2)^+$.  We can continue
this process each time increasing the dimension of $R_i$ by at least one
and since $\mathbb C^p$ is finite dimensional this process must terminate
after a finite number of steps.  Then we have
\begin{equation*}
I - P \leq \lambda^{-1}(I-I_o) + (E_1-E_1^2)^+ + (E_2-E_2^2)^++ \cdots + (E
_m-E_m^2)^+.
\end{equation*}

For the rest of this proof when write $\Vert\omega (\Lambda (A))\Vert \leq
K$ we mean $\Vert\omega\vert_t(\Lambda (A))\Vert \leq K$ for all $t > 0$.
Now from Theorem 5.9 we know $\omega (I-\Lambda (I_o)) \leq I_o$ so we know
$\omega (I-\Lambda (I)+\Lambda (I-I_o)) \leq I_o$ and since $I \geq \Lambda
(I)$ we have $\Vert\omega (\Lambda (I-I_o))\Vert \leq 1$.  From Theorem
5.12 we know there are constants $K_i$ so that $\Vert\omega (\Lambda
(E_i^+-E_i))\Vert \leq K_i$ for $i = 1,\cdots ,m$.  Let $\lambda_i$ be the
largest eigenvalue
of $E_i$ that is less than one and let
$C_i = K_i/(1-\lambda_i)$.   Since $(E_i-E_i^2)^+ \geq (1-\lambda_i)^{-1}
(E_i^+-E_i)$ it follows that $\Vert\omega (\Lambda((E_i-E_i^2)^+))\Vert
\leq C_i$ and, hence, we have
\begin{equation*}
\Vert\omega (\Lambda (I-P))\Vert \leq \lambda^{-1} + C_1 + C_2+ \cdots  + C
_m.
\end{equation*} \end{proof}

We remark that if we were being graded on how good a bound we have obtained
we would get a pretty low grade.  Fortunately we only need a bound as the
next lemma shows.

\begin{lem}\label{5.21}
Suppose $\omega$ is a $q$-weight map over $\mathbb C^p$ of index zero and
$\mathcal{L}$ the range $\omega$ is a Choi-Effros factor of type~I$_q$ and
suppose $P$ is the maximal support projection for $\mathcal{L}$.  Then
there is a $q$-subordinate $\eta$ (i.e. $\omega \geq_q\eta )$ with range
$\mathcal{L}$ so that $\Vert\eta (\Lambda (I-P))\Vert \leq \tfrac{1}{2}$.
\end{lem}
\begin{proof}  Assume the hypothesis and notation of the lemma.  From the
previous theorem we there is a constant $K$ so that
$\Vert\omega\vert_t\Lambda (I-P)\Vert \leq K$ for all $t > 0$.  Assume $K$
is the least such constant.  If $K \leq \tfrac{1}{2}$ then $\eta = \omega$
satisfies the conclusion of the theorem so we assume $K > \tfrac{1}{2} $.
Note $\omega\vert_t\Lambda (I-P)$ is non increasing in $t$.  Hence, there
is a $t_o > 0$ so that $(\omega\vert_t\Lambda-\omega\vert_{t_o}\Lambda
)(I-P) \leq (1/4)I$ for $0 < t < t_o$.  Let $\pi_t^\#$ be the generalized
boundary representation of $\omega$.  For $t > 0$ let
\begin{equation*}
\eta_t = (\iota-\lambda_t\pi_t^\#\Lambda )^{-1}\lambda_t\pi_t^\#
\end{equation*}
where $\lambda_t$ for $0 < t < t_o$ is
the largest number in $(0,1]$ so that $\Vert\eta_t\vert_{t_o}\Lambda
(I-P)\Vert \leq 1/4$.  Let $\eta$ be a weak limit point of $\eta_t$ as $t
\rightarrow 0+$.  From Theorem 5.6 we know that $\eta$ is a $q$-subordinate
of $\omega$.  From the construction of $\eta$ we have
$\Vert\eta\vert_{t_o}\Lambda (I-P)\Vert \leq 1/4$.  Since $\eta \leq
\omega$ we have $(\eta\vert_t\Lambda-\eta\vert_{t_o}\Lambda )(I-P) \leq
(1/4)I$ for $0 < t < t_o$ so we have $\Vert\eta\vert_t\Lambda (I-P)\Vert
\leq \tfrac{1}{2}$ for all $t > 0$.

Note if $\eta = 0$ then $\Vert\eta\vert_t\Lambda (I-P)\Vert = 0$ for all $t
> 0$ and recalling how $\lambda_t$ were chosen this implies $K \leq
\tfrac{1}{2}$ and for this case the proof is trivial.  From the construction
of $\eta$ is it clear that the range of $\eta$ is contained in
$\mathcal{L}$.  Now let $\psi_t^\#$ be the generalized boundary
representation of $\eta$.  Since $L_t = \pi_t^\#\Lambda \geq
\psi_t^\#\Lambda = L_t^{\prime}$ for $t > 0$ so taking a sequence $t_k
\rightarrow 0+$ so that both $L_{t_k}$ and $L_{t_k}^{\prime}$ converge to
limits $L$ and $L^{\prime}$ we have $L \geq L^{\prime}$.  Now $L$
restricted to $\mathcal{L}$ is the identity map and since the identity map
is pure as a completely positive map we have $L^{\prime}(A) = \lambda L(A)$
for $A \in \mathcal{L}$ with $0 \leq \lambda \leq 1$ and since $L^{\prime}
\neq 0$ we have $\lambda >$ 0.  (Actually $\lambda = 1.)$.  Hence, the
range of $L^{\prime}$ is $\mathcal{L}$ so the range of $\eta$ contains
$\mathcal{L}$ so the range $\eta$ is $\mathcal{L}.$ \end{proof}

We remark in the previous lemma we did not specifically use the
fact that $P$ was the maximal support projection.  The properties of $P$
that we did use were the fact that $P \in \mathcal{L}^{\prime}$ and
$\Vert\omega (\Lambda (I-P))\Vert < \infty$.  This means the conclusions of
Lemma 5.21 remain valid if we only assume these weaker conditions on $P$.

\begin{thm}\label{5.22}
Suppose $\omega$ is a $q$-weight map of index zero over $\mathbb C^p$ with
range $\mathcal{L}$ which is a Choi-Effros factor of type~I$_q$ with $P$ the
maximal support projection for $\mathcal{L}$.  Then there is a $q$-subordinate
$\eta$ of $\omega$ so that the range of $\eta$ is $\mathcal{L}_\eta =
P\mathcal{L}$.  It follows that Choi-Effros product for $\mathcal{L}_\eta$
is the ordinary operator product.  Furthermore, it follows that if $\omega$
is a $q$-pure $q$-weight map of index zero over $\mathbb C^p$ then
$\mathcal{L}$ the range of $\omega$ is a factor and the unit $I_o$ of
$\mathcal{L}$ equals the maximal support projection of $\mathcal{L}$ (i.e.
the range of $\omega$ is a Choi-Effros factor and $A\star B = AB$ for $A,B
\in \mathcal{L}$ and $I_o = P)$.
\end{thm}
\begin{proof}  Assume the hypothesis and notation of the theorem.  We make
a change in notation.  We replace $\omega$ in the statement of the theorem
by $\omega^{\prime}$.  Then by the previous lemma we know that
$\omega^{\prime }$ has a $q$-subordinate $\omega$ with range $\mathcal{L}$
so that $\Vert\omega\Lambda (I-P)\Vert \leq \tfrac{1}{2} $.  Since any
$q$-subordinate of $\omega$ is a $q$-subordinate of $\omega^{\prime}$ we
only need to prove the theorem for $\omega$ where $\omega$ now satisfies the
additional the requirement that $\Vert\omega\Lambda(I-P)\Vert\leq\tfrac{1}{2}$.

We show how to construct $\eta$.  Suppose $\pi_t^\#$ is the generalized
boundary representation of $w$ and let $\psi_t^\# = P\pi_t^\#$ for $t > 0$
and let
\begin{equation*}
\eta_t = (\iota - \psi_t^\#\Lambda )^{-1}\psi_t^\#.
\end{equation*}
Then we have
\begin{align*}
\eta_t &= (\iota - P\pi_t^\#\Lambda )^{-1}P\pi_t^\#  = (\iota -
P\pi_t^\# \Lambda )^{-1}P(\iota + \omega_t\Lambda )^{-1}\omega\vert_t
\\
&= (\iota - P\pi_t^\#\Lambda )^{-1}P(\iota - \pi_t^\#\Lambda )
\omega\vert _t
\\
&= P(\iota - P\pi_t^\#\Lambda )^{-1}((\iota - P\pi_t^\#\Lambda ) - (I - P))
\omega\vert_t
\\
&= P\omega\vert_t - P(\iota - P\pi_t^\#\Lambda )^{-1}(I - P)
\omega\vert_t
\\
&= (P - P(\iota - P\pi_t^\#\Lambda )^{-1}(I - P))\omega\vert_t = (\xi -
\zeta _t)\omega\vert_t
\end{align*}
where for $A \in B(\mathbb C^p)$ the maps $\xi$ and $\zeta_t$ are given by
$\xi (A) = PAP$ and
\begin{align*}
\zeta_t(A) &= P(\iota - P\pi_t^\#\Lambda )^{-1}((I-P)A(I-P))
\\
&= P\pi_t^\#\Lambda (\iota - P\pi_t^\#\Lambda )^{-1}((I-P)A(I-P))
\\
&= (P\pi_t^\#\Lambda + (P\pi_t^\#\Lambda )^2 + \cdots )((I-P)A(I-P)).
\end{align*}
Note $\zeta_t$ is completely positive.  We estimate the norm
$\Vert\zeta_t\Vert = \Vert\zeta_t(I)\Vert$.  Since $\pi_t^\#\Lambda -
P\pi_t^\#\Lambda$ is a completely positive map if we replace $P\pi_t^\#$ by
$\pi_t^\#$ in the above formula for $\zeta_t(A)$ with $A = I$ we obtain the
estimate
\begin{align*}
\zeta_t(I) &= (P\pi_t^\#\Lambda +(P\pi_t^\#\Lambda )^2 +\cdots )((I-P)I(I-P))
\\
&\leq P(\pi_t^\#\Lambda + (\pi_t^\#\Lambda )^2 + \cdots )(I-P)
\\
&= P(\iota-\pi_t^\#\Lambda )^{-1}\pi_t^\#\Lambda (I-P) =
P\omega\vert_t\Lambda (I-P)
\end{align*}
so we have $\Vert\zeta_t(I)\Vert \leq \Vert P\omega\vert_t\Lambda
(I-P)\Vert \leq \Vert\omega\vert_t\Lambda (I-P)\Vert \leq
\tfrac{1}{2}$ for all $t > 0$.  Note that $\eta_t =
(\xi-\zeta_t)\omega\vert_t$ so the limit points of $\eta _t$ as $t
\rightarrow 0+$ are in one to one correspondence with limit points
of $\zeta_t$.  Then let $\zeta$ be a limit point of $\zeta_t$ as
$t \rightarrow 0+$.  Then from Theorem 5.6 we have $\eta =
(\xi-\zeta )\omega$ is a $q$-subordinate of $\omega$.  Since the
$\zeta_t$ are completely positive maps with $\Vert\zeta_t\Vert
\leq \tfrac{1}{2}$ it follows that $\zeta$ is completely positive
with $\Vert\zeta\Vert \leq \tfrac{1}{2}$.  We will show that the
range of $\eta$ is $P\mathcal{L}$.

Here is the tricky part.  We are given that the range of $\omega$ is
$\mathcal{L}$ so for each $A \in \mathcal{L}$ there is a $B \in \mathfrak{A}
(\mathbb C^p)$ so that $A = \omega (B)$.  Now $\eta (B) = (\xi-\zeta)\omega(B)
= (\xi-\zeta)(A)$ so $\eta (B)$ does not depend on the actual operator $B$ but
only on the fact that $\omega (B) = A$.  This means that $\zeta$ can be viewed
as a map from $\mathcal{L}$ to $P\mathcal{L}$ and since $A \in \mathcal{L}$ is
uniquely determined by $PA$ this means that $\zeta$ can be viewed as
a mapping of $P\mathcal{L}$ into itself.  Finally, since $\zeta$ is completely
positive we can view $\zeta$ as a completely positive mapping of
$P\mathcal{L}$ into itself.  The strange thing is that $\zeta (A)$ is computed
from $(I-P)A$ and $(I-P)A$ may be zero while $A$ is not zero.  This is not an
error since if $(I-P)A = 0$ then $\zeta (A) = 0$.  We have worked very hard to
arrange it so that $\Vert\zeta\Vert \leq \tfrac{1}{2}$ and without such an
estimate we could not conclude this proof.  Now as a mapping of $P\mathcal{L}$
into itself the mapping $\xi$ is just the unit mapping and the mapping under
consideration is $\xi - \zeta$.  Since even for $\zeta$ considered as a
mapping of $P\mathcal{L}$ into itself we still have $\Vert\zeta\Vert \leq
\tfrac{1}{2}$ the mapping $\xi - \zeta$ is invertible with
\begin{equation*}
(\xi - \zeta )^{-1} = \xi + \zeta + \zeta^2 + \zeta^3+ \cdots
\end{equation*}
where the series converges in norm.  Hence, for each $A \in \mathcal{L}$
there is a $B \in \mathcal{L}$ so that $(\xi - \zeta )(B) = PA$.  Since the
range of $\omega$ is $\mathcal{L}$ there is a $C \in \mathfrak{A} (\mathbb
C^p)$ so that $\omega (C) = B$ and then
\begin{equation*}
\eta (C) = (\xi-\zeta )(\omega (C)) = (\xi-\zeta )(B) = PA
\end{equation*}
so the range of $\eta$ is $P\mathcal{L}$.

Now we prove the last statement of the theorem.  Suppose $\omega$ is a
$q$-pure $q$-weight map over $\mathbb C^p$ of index zero.  From Theorem 5.5 we
know that the range $\mathcal{L}$ of $\omega$ is a factor with the
Choi-Effros multiplication.  Now suppose $\eta$ is the $q$-subordinate of
$\omega$ we just constructed.  From Theorem 2.2 we know that the range of
$\eta$ is contained in $\mathcal{L}$ and, therefore, $P\mathcal{L} \subset
\mathcal{L}$ and since the mapping $A \leftrightarrow$ PA is a
$*$-isomorphism in each direction it follows that $P\mathcal{L} =
\mathcal{L}$.  Hence, the range of $\omega$ is a factor and the unit $I_o$
of $\mathcal{L}$ is $PI_o = P.$ \end{proof}

We remark that if in the previous theorem all we wish to prove
is that there is a $q$-subordinate $\eta$ with range $\mathcal{L}_\eta =
P\mathcal{L}$ then all we need to assume about the projection $P$ is that
$P \in \mathcal{L}^{\prime}$ and $\Vert\omega (\Lambda (I-P))\Vert < \infty$.

\section{The factor case}

As we saw in the last section if $\omega$ is a $q$-pure $q$-weight map over
$\mathbb C ^p$ of index zero then the range of $\omega$ is a factor of
type~I$_q$ contained in $B(\mathbb C^p)$ with $q \leq p$.  For these range
algebras $\mathcal{L}$ the Choi-Effros product is just the ordinary product so
$A\star B = AB$ for $A,B \in \mathcal{L}$.  Also the maximal support projection
$P$ for $\mathcal{L}$ is the unit $I_o$ of $\mathcal{L} .  $ Since we are
primarily interested in finding the $q$-pure $q$-weight map over $\mathbb C^p$
we will restrict our attention to these $q$-weight maps.

As mentioned in the last section when computing the generalized boundary
representation $\pi_t^\# = (\iota+\omega \vert_t \Lambda )^{-1}\omega \vert_t$
we need only consider how $(\iota+\omega \vert_t \Lambda )^{-1}$ acts on
$\mathcal{L}$ the range of $\omega $.  Therefore, in our computation we only
need know the action of $\omega \vert_t \Lambda$ on $\mathcal{L}$ which can be
parameterized by matrix units $E_{ij}$ for $i,j =$ $1,\cdots ,q$.  Given
a mapping $\phi$ of $B(\mathbb C^p)$ into itself we denote by $\tilde \phi$
the restriction of $\phi$ to $\mathcal{L}$.  If $\phi$ maps $\mathcal{L}$ into
itself then we can think of $\tilde \phi$ as a mapping of the $(q \times
q)$-matrices into themselves.  So in what follows when we work with maps
$\tilde \phi$ we will parameterize them as mappings of the $(q \times
q)$-matrices into themselves.  Note the order relations on such maps is the
same as the order relations on the corresponding matrix maps.  Notice that
if $\phi_t = \omega \vert_t \Lambda$ then $\tilde \phi_t = \omega \vert_t
\tilde \Lambda$.

First we formalize these ideas with the following definition.

\begin{defn}
Suppose $\mathcal{L}$ is a type~I$_q$ factor with $q \leq p$ contained in
$B(\mathbb C^p)$ where the identity $I_o$ of $\mathcal{L}$
is a projection (not necessarily the unit of $B(\mathbb C^p))$.  We say
$\omega$ is a $q$-weight map over $\mathcal{L}$ if $\omega$ is a $q$-weight map
over $\mathbb C^p$ with values in $\mathcal{L}$.  We say $\omega$ is $q$-pure
over $\mathcal{L}$ if the $\mathcal{L}$ valued $q$-subordinates of $\omega$ are
totally ordered.
\end{defn}

The next theorem is Theorem 5.10 only instead of getting an expression for
$\tilde{\vartheta}$ we get an expression for $\vartheta$.  Note in our
decomposition of $\vartheta$ in terms of $g's$ and $h's$ we need the
$E_{ij}$ to be actual partial isometries which they are due to the fact that
$\mathcal{L}$ is an algebra with the ordinary operator product.

\begin{thm}
Suppose $\omega$ is a $q$-weight map over $\mathbb C^p$ of
index zero and the range of $\omega$ is $\mathcal{L}$ which is a factor
of type~I$_q$.  Then $\omega$ is of the form $\omega = \psi^{-1}\vartheta$
where $\psi$ is an invertible conditionally negative map of $\mathcal{L}$ into
itself with a completely positive inverse and $\vartheta$ is of the form
$$
\vartheta (A) = \sum_{i,j=1}^q E_{ij}\vartheta_{ij}(A)
$$
where the $E_{ij}$ are a complete set of matrix units for $\mathcal{L}$ and
$$
\vartheta_{ij}(A) = \sum_{k\in J} ((g_{ik}+h_{ik}),A(g_{jk}+h_{jk}))
$$
for $A \in \mathfrak A (\mathbb C^p)$ where the $g_{ik},\medspace h_{ik} \in
\mathbb C^p\otimes L_+^2(0,\infty )$ and
$$
g_{ik}(x) = E_{i1}g_k(x)\qquad \text{and} \qquad E_{11}g_k(x) = g_k(x)
$$
and
$$
\sum_{i=1}^q E_{1i}h_{ik}(x) = 0
$$
for $A \in \mathfrak A (\mathbb C^p),\medspace x \geq 0,\medspace i,j \in \{
1,\cdots ,q\}$ and $k \in J$ a countable index set and the $h_{ik} \in \mathbb
C^p \otimes L^2(0,\infty )$ and if
$$
w_t = \sum_{k\in J} (g_k,\Lambda \vert_t g_k)\qquad \rho_{ij}(A) =
\sum _{k\in J} (h_{ik},Ah_{jk})
$$
then $\rho$ is bounded so
$$
\sum_{k\in J} \Vert h_{ik}\Vert^2 < \infty\qquad \text{and} \qquad
\sum_{k\in J} (g_k,(I-\Lambda )g_k) < \infty
$$
and $1/w_t \rightarrow 0$ as $t \rightarrow 0+$ and $\psi$ satisfies the
conditions
$$
\psi (I_o) \geq \vartheta (I - \Lambda (I_o))\qquad \text{and} \qquad \psi
+ \rho \tilde \Lambda
$$
is conditionally negative. (Recall $\tilde \Lambda$ is $\Lambda$
restricted to $\mathcal{L} .)$

Conversely, if $\vartheta ,\medspace \rho$ and $\psi$ are as given above
then $\omega$ is a $q$-weight map over $\mathbb C^p$ of index zero and the
range of $\omega$ is $\mathcal{L}$.  Furthermore, if $\psi ^{\prime}$ is
a second map satisfying the conditions above and $\omega ^{\prime} = \psi
^{\prime -1}\vartheta$ then $\omega ^{\prime}$ is a $q$-subordinate of
$\omega$ (i.e. $\omega \geq_q \omega ^{\prime})$ if and only if $\psi \leq
\psi ^{\prime}$.
\end{thm}
\begin{proof}  Assume the hypothesis and notation of the theorem.  Since
$\omega$ satisfies the hypothesis of Theorem 5.10 we have $\omega
= \Theta \vartheta = \psi^{-1}\vartheta$ where $\Theta, \psi$ and
$\vartheta$ are defined in Theorem 5.10.  Note $\vartheta$ is a
completely positive $\mathcal{L}$ valued $b$-weight map on $\mathfrak{A}
(\mathbb C^p)$.  In Theorem 5.10 there is a decomposition of
$\tilde{\vartheta}$ in terms of $g's$ and $h's$. Since the
$E_{ij}$ are actual partial isometries we can make a finer
decomposition of $\vartheta$ in terms of different $g's$ and $h's$
as follows. From the general theory of completely positive maps we
know that $\vartheta$ can be written in the form
$$
\vartheta_{ij}(A) = \sum_{k\in J} (F_{ik},AF_{jk})
$$
with the $F_{ik} \in \mathbb C^p\otimes L_+^2(0,\infty )$ for $k \in J$ a
countable index set.  Furthermore, we know the $F_{ik}$ can be chosen so
they are linearly independent over $\ell^2(J)$.  Now we define
\begin{equation*}
g_k = \frac {1} {q} \sum_{i=1}^q E_{1i}F_{ik}
\end{equation*}
and
\begin{equation*}
g_{ik}(x) = E_{i1}g_k(x)\qquad \text{and} \qquad h_{ik} = F_{ik} - g_{ik}.
\end{equation*}
One checks from the definition of the $h's$ that
\begin{equation*}
\sum_{i=1}^q E_{1i}h_{ik}(x) = 0.
\end{equation*}

We define the completely positive $B(\mathbb C^p)$ valued $b$-weight map
$\rho \in B(\mathbb C^p) \otimes \mathfrak A (\mathbb C ^p)_*$
given by
$$
\rho_{ij}(A) = \sum_{k\in J} (h_{ik},Ah_{jk})
$$
for $A \in \mathfrak A (\mathbb C^p)$ and for $t > 0$ we define
$$
w_t = \sum_{k\in J} (g_k,\Lambda \vert_t g_k), \qquad (R_t)_{ij}(A) =
\rho_{ij}\vert_t (\Lambda (A)),\qquad (Y_t)_{ij} = \sum_{k\in J}
((h_{ik})_j,\Lambda \vert_t g_k)
$$
and
$$
\zeta_t(A) = Y_tA
$$
for $A \in B(\mathbb C^p)$.  Then we have
$$
\vartheta \vert_t \Lambda (A) = w_tA + Y_tA + AY_t^* + R_t(A)
$$
for $A \in B(\mathbb C^p)$ or
$$
\vartheta \vert_t \Lambda = w_t\iota + \zeta_t + \zeta_t^* + R_t
$$

There is no simple relation between the $g's$ and $h's$ above and in Theorem
5.10.  In fact in Theorem 5.10 they are $\mathbb C^q$ valued functions of $x$
and here they are $\mathbb C^p$ valued functions of $x$.  What is unique is the
functionals they generate when restricted to $\mathcal{L}$ (i.e.
the tilde maps).  We see this by noting that in the decomposition
$$
\tilde{\vartheta}\vert_t \Lambda = w_t\iota + \tilde{\zeta}_t +
\tilde{\zeta}_t^* + \tilde{R}_t
$$
that the completely positive mapping $\tilde{\rho}\vert_t\Lambda =\tilde{R}_t$
is the internal part of $\tilde{\vartheta}\vert_t\Lambda$
where the internal part of a mapping was defined equation 3.4 of section 3 of
this paper.  We recall that in defining the internal part of a mapping we
decomposed operators in $B(\mathbb C^q)$ to a multiple of the unit $I$ plus an
operator of trace zero.  The trace zero conditions for the $h's$ corresponds
to the condition
\begin{equation*}
\sum_{i=1}^q E_{1i}h_{ik}(x) = 0.
\end{equation*}
In Theorem 5.10 we saw that $\tilde{\rho}\vert_t\Lambda$ is the internal part of
$\tilde{\vartheta}\vert_t\Lambda$.  Recall in Theorem 5.10 the $h's$ satisfied
the condition
\begin{equation*}
\sum_{i=1}^q (h_{ik})_i(x) = 0
\end{equation*}
which corresponds to the trace zero condition in the setting of Theorem 5.10.
We see then in both cases we have $\tilde{\rho}\vert_t \Lambda$ is the internal
part of the mapping $\tilde{\vartheta}\vert_t\Lambda$ so the results of
Theorem 5.10 apply to our $\tilde{\rho}\vert_t\Lambda$ and, hence,
there is a constant $K$ such that
$$
\tilde{\rho}\vert_t \Lambda(I_o)  = \rho\vert_t \Lambda(I_o) \leq KI_o
$$
for $t > 0$.  Now we have $\rho\vert_t(I) = \rho\vert_t
(\Lambda (I_o)) + \rho\vert_t(I-\Lambda (I_o)) \leq KI_o + \rho\vert_t
(I-\Lambda (I_o))$ and
\begin{align*}
\sum_{i=1}^q \rho_{ii}\vert_t(I-\Lambda (I_o)) &\leq \sum_{i=1}^q \vartheta
_{ii}\vert_t(I-\Lambda (I_o)) = \sum_{i=1}^q (\psi\omega
)_{ii}\vert_t(I-\Lambda (I_o))
\\
&\leq \Vert\psi\Vert \sum_{i=1}^q \Vert\omega\vert_t(I-\Lambda (I_o)\Vert \leq
\Vert\psi\Vert \sum_{i=1}^q \Vert I_o\Vert = q\Vert\psi\Vert
\end{align*}
so
$$
\Vert\rho\vert_t\Vert = \Vert \rho\vert_t(I)\Vert \leq K + q\Vert\psi\Vert
$$
for $t > 0$ so $\rho$ is bounded.

The fact that
\begin{equation*}
\psi (I_o) \geq \vartheta (I - \Lambda (I_o))\qquad \text{and} \qquad \psi
+ \rho \tilde \Lambda
\end{equation*}
is conditionally negative was established in Theorem 5.10.

Now for the proof of the last paragraph.  Here we return to the first part of
the proof of Theorem 4.7.  The proof in our case follows line by line the proof
of Theorem 4.7 making the following changes.  We replace $Z(t)$ by the
corresponding $\tilde{Z}(t)$.  This replacement is automatic because
$Z(t)^{-1}$ acts on $\vartheta$ which has range $\mathcal{L}$.  The unit
$I \in B(\mathbb C^p)$ is replaced by $I_o$.  Note the unit $I$ of $B(\mathbb
C^p \otimes L^2(0,\infty ))$ is not replaced so, for example, the unit $I$ in
$\pi^\#_t(I)$ is not replaced.  Expressions like $\pi^\#_t(\Lambda)$ are
replaced by $\pi^\#_t(\Lambda(I_o))$.  In proving $\pi^\#_t(I) \leq I$ one
proves the stronger inequality $\pi^\#_t(I) \leq I_o \leq I$ using the
inequality $\psi (I_o) \geq \vartheta (I - \Lambda (I_o))$ given in the
statement of the present theorem.  Note this inequality is the translation
of the inequality $\psi (I) \geq \vartheta (I - \Lambda)$ that was used in the
proof of Theorem 4.7.  The proof of the last sentence of this theorem follows
from the Theorem 4.7 with the replacements described.
\end{proof}

The next theorem shows how to find all $q$-subordinates in the case at hand
that have the same range $\mathcal{L}$.

\begin{thm}
Suppose $\omega$ is a $q$-weight map over $\mathbb C^p$ of
index zero and range $\mathcal{L}$ which is a factor of type~I$_q$ and $\omega
= \psi^{-1}\vartheta$ with $\psi$ a coefficient map of $\omega$ and
$\vartheta$ the corresponding limiting $b$-weight map.  Suppose there is a
bounded completely positive $\mathcal{L}$ valued weight map $\eta \in B(\mathbb
C^p)\otimes \mathfrak A (\mathbb C^p)_*$ so that $\vartheta \geq \eta \geq 0$.
Then if
$$
\psi ^{\prime} \geq \psi + \eta \tilde \Lambda \qquad \text{and} \qquad
\psi ^{\prime} + \rho_\vartheta \tilde \Lambda  - \eta \tilde \Lambda
$$
is conditionally negative (where $\tilde \Lambda$ is $\Lambda$ restricted to
$\mathcal{L}$) then $\omega ^{\prime} = \psi^{\prime -1}(\vartheta - \eta )$
is a $q$-subordinate of $\omega$.

Conversely, if $\omega ^{\prime}$ is $q$-subordinate of $\omega$ whose
range is contained in $\mathcal{L}$ there is a bounded completely positive
$\mathcal{L}$ valued $b$-weight map $\eta \in B(\mathbb C^p) \otimes \mathfrak A
(\mathbb C^p)_*$  with $\vartheta \geq \eta \geq 0$ and $\psi ^{\prime}$
satisfying the conditions above so that $\omega ^{\prime} = \psi ^{\prime -1}
(\vartheta - \eta )$.
\end{thm}
\begin{proof}  The proof of this theorem follows from the proof of Theorem
4.9 where each of the equations and inequalities in 4.9 are
reinterpreted in our new setting.  Rather than write out the new
proof we will explain how to make the translation.  The first
lines of the proof of Theorem 4.9 produce formulae for
$\vartheta\vert_t\Lambda$ and $\vartheta^{\prime}\vert_t\Lambda$.
Instead those formulae should be replaced by the corresponding
formulae for $\vartheta\vert_t\tilde{\Lambda}$ and
$\vartheta^{\prime}\vert_t\tilde{\Lambda}$ and the reference to
Theorem 4.7 should be replaced by a reference to Theorem 6.2.  The
mapping $T_t = \eta\vert_t\Lambda$ should be define as $T_t =
\eta\vert_t\tilde{\Lambda}$.  This brings us to the first rule in
translating the proof of Theorem 4.9 to our new setting, namely
replace the mapping $\Lambda$ by $\tilde{\Lambda}$.  This
replacement is justified because in all the calculation involving
$\Lambda$ the calculated mappings act on the range of $\vartheta
,\medspace \vartheta^{\prime}$ or $\eta$. Note that since $\phi_t
= \omega\vert_t\Lambda$ and $\phi_t^{\prime} = \omega
^{\prime}\vert_t\Lambda$ the replacement $\Lambda$ by
$\tilde{\Lambda}$ means these mappings should be replace by
$\tilde{\phi}_t$ and $\tilde{\phi}_t^{\prime }$.

The next replacement is the unit $I$ of $B(\mathbb C^p)$ which
should be replaced by $I_o$ the unit of $\mathcal{L}$.  Note in
our case $I_o = P$ the maximal support projection.  To give an
example this means the expression $\psi (I)$ in the proof of
Theorem 4.9 should be replaced by $\psi (I_o)$.  Another example
is an expression like $\eta \Lambda$ in Theorem 4.9 where
$\Lambda$ is not the mapping $\Lambda$ but short for $\Lambda (I)$
which should be replaced by $\eta (\Lambda (I_o))$ or $\eta
(\tilde{\Lambda}(I_o))$.  Note $\Lambda (I_o)$ and
$\tilde{\Lambda}(I_o)$ are interchangeable since $I_o \in
\mathcal{L}$.

Next an important rule we noted earlier.  The unit $I$ of $B(\mathbb
C^p\otimes L^2(0,\infty ))$ is not changed so, for example, an expression
like $\pi_t^\# (I)$ in Theorem 4.9 remains $\pi_t^\# (I)$ in translation.
Note the inequality $\pi_t^\# (I) \leq I$ in Theorem 4.9 becomes $\pi_t^\#
(I) \leq I_o$ in our new setting.

To see how these rules apply consider the inequality $\omega (I-\Lambda )
\leq I$ in Theorem 4.9.  Applying the translation rules we have described
this expression becomes $\omega (I - \Lambda (I_o)) \leq I_o$ which was
proved in Theorem 5.9.  Notice then that an inequality like $\psi (I) \geq
\vartheta (I - \Lambda )$ in Theorem 4.9 becomes $\psi (I_o) \geq \vartheta
(I-\Lambda (I_o))$ in our new setting.  Also the references to previous
theorems in Theorem 4.9 should be updated to the corresponding theorems in
sections five and six in our new setting.  Then following these rules the
proof of Theorem 4.9 gives us a proof of the present theorem.
\end{proof}

\begin{thm}
Suppose $\omega$ is a $q$-weight map over $\mathbb C^p$ of
index zero and the range $\mathcal{L}$ the range of $\omega$ is a factor of
type~I$_q$ with $q \leq p$.  We say $\omega$ is $q$-pure over $\mathcal{L}$ if
the $\mathcal{L}$-valued $q$-subordinates are totally ordered.  Then $\omega$ is
$q$-pure over $\mathcal{L}$ if and only if $\omega$ can be written in the form
given in Theorem 6.2 so $\omega = \psi\vartheta$ where $\psi + \rho_\vartheta
\tilde \Lambda$ is conditionally zero and $\psi (I_o) \geq \vartheta
(I - \Lambda (I_o))$ and $\vartheta$ is strictly infinite $\mathcal{L}$-valued
$b$-weight map meaning that if $\eta$ is a finite completely positive
$\mathcal{L}$ valued $b$-weight map and $\vartheta \geq \eta \geq 0$
then $\eta = 0$.
\end{thm}

\begin{proof}  The proof of the theorem follows from the proof of Theorem 4.10
where we translate the proof of Theorem 4.10 following the procedure just
described.
\end{proof}

\begin{thm}
Suppose $\omega$ is a $q$-weight map over $\mathbb C^p$
of index zero and $\mathcal{L}$ is the range of $\omega$.  Then $\omega$ is
$q$-pure if and only if the following conditions are satisfied.
\begin{enumerate}[(i)]

\item $\mathcal{L}$ is a factor of type~I$_q$ with $1 \leq q \leq p$ and
$\omega$ is of the form $\omega = \psi\vartheta$ with $\psi ,\medspace
\vartheta$ and $\rho$ are as given in Theorem 6.2 and the map $\psi +
\rho\tilde \Lambda$ is conditionally zero.

\item The $b$-weight map $\vartheta$ is strictly infinite as an $\mathcal{L}$
valued $b$-weight map.

\item  If $e \in \mathcal{L}^{\prime}$ is a non zero hermitian projection and
$e\leq I_o$ the unit of $\mathcal{L}$ then $\Vert\omega (\Lambda (e))\Vert =
\infty$.
\end{enumerate}
\end{thm}
\begin{proof}  Suppose $\omega$ is a $q$-pure $q$-weight map over $\mathbb C^p$
of index zero. Then from Theorem 6.4 we see that $\omega$ satisfies conditions
(i) and (ii).

Now suppose $\omega$ fails to satisfy condition (iii) so there is a
projection $e \in \mathcal{L} ^{\prime}$ with $e \leq I_o$ and $\Vert\omega
(\Lambda (e))\Vert < \infty$.  Note since $\omega$ is of index zero
$\omega(\Lambda(I_o))$ is infinite so $e \neq I_o$.  Repeating the argument of
Lemma 5.21 (see the remark after Lemma 5.21) we find there is a
$q$-subordinate $\omega^{\prime}$ of $\omega$ so that
$\Vert\omega^{\prime}(\Lambda(e))\Vert \leq \frac {1} {2}$. Let $P = I - e$,
let $\pi_t^\#$ be the boundary representation of $\omega^{\prime}$ and let
$\psi_t^\# = P\pi_t^\#$ for $t > 0$.  Since $P \in \mathcal{L} ^{\prime}$ we
have $\psi_t^\# \leq \pi_t^\#$ for $t > 0.  $ Let
$$
\eta_t = (\iota - \psi_t^\#\Lambda )^{-1}\psi_t^\#
$$
for $t > 0$ and by Theorem 5.6 we know that any weak limit point of $\eta_t$
as $t \rightarrow 0+$ is a $q$-subordinate of $\omega^{\prime}$ and since
$\omega \geq_q \omega^{\prime}$ it is a $q$-subordinate of $\omega$.
Calculating $\eta_t$ as we did in Theorem 5.22 (see the remark after
Theorem 5.22) we find
$$
\eta_t = (P - P(\iota - P\pi_t^\#\Lambda )^{-1}(I-P))\omega \vert_t = (P -
\zeta_t)\omega \vert_t =  P(\iota - \zeta_t) \omega \vert_t
$$
where
$$
\zeta_t(A) = P(\iota - P\pi_t\Lambda )^{-1}((I-P)A(I-P))
$$
and repeating the argument in the proof of Theorem 5.22 we find that
$\zeta_t$ is a completely positive contraction and we see that the limit
points of $\eta_t$ as $t \rightarrow 0+$ are of the form $\eta = (P -
\zeta )\omega$ where $\zeta$ is limit point of $\zeta_t$ as $t\rightarrow 0+$
and what is more important is $\Vert\zeta_t(I)\Vert \leq \tfrac{1}{2}$.  As in
the proof of Theorem 5.22 this bound insures that inverse
$$
(\iota - \zeta)^{-1} = \iota + \zeta + \zeta^2 + \cdots
$$
exists since the series converges.  Since the range of $\omega$ is
$\mathcal{L}$ the range of $\eta$ is $P\mathcal{L}$.  But this is a
contradiction since from Theorem 2.2 we know that the range of any
$q$-subordinate of $\omega$ is contained in $\mathcal{L}$.  Hence, the
assumption that $\Vert\omega (\Lambda (e))\Vert < \infty$ is false.

Now suppose $\omega$ satisfies the three conditions of the theorem.  Since
$\omega$ satisfies conditions (i) and (ii) then by Theorem 6.4 we have
$\omega$ is $q$-pure over $\mathcal{L}$.  Then if $\omega$ is not $q$-pure it
has a $q$-subordinate $\tau$ whose range is not contained in $\mathcal{L}$ and
by Theorem 5.22 there is a $q$-subordinate of $\nu$ of $\tau$ so that the
range of $\nu$ is a factor.  Since $\omega$ is of index zero it follows
that $\nu$ is of index zero.  Let $\pi_t^\#$ and $\phi _t^\#$ be the
generalized boundary representations of $\omega$ and $\nu$, respectively.
By routine compactness arguments we can find a decreasing sequence $t_k
\rightarrow 0+$ so that $\pi_{t_k}^\#\Lambda \rightarrow L$ and
$\phi_{t_k}^\#\Lambda \rightarrow L_2$ as $k \rightarrow \infty$.  Since
$\pi_t^\#\Lambda \geq \phi_t^\#\Lambda$ for $t > 0$ we have $L \geq L_2$
and $L$ and $L_2$ are completely positive contractive idempotent maps.  Let
$\mathcal{L}_1$ be the range of $\nu$.  Since $\mathcal{L}_1$ is a factor
the unit $P$ of $\mathcal{L}_1$ is a projection so $L_2(I) = P$ and by
Theorem 5.4 we have $P \in \mathcal{L} ^{\prime},\medspace P \leq I_o =
L(I)$ and $L_2(A) = PL(A)P$ for $A \in B(\mathbb C^p)$ and the range of
$\nu$ is $P\mathcal{L} $.

Notice that the mapping $A \rightarrow$ PA $= PAP$ is a $*$-isomorphism of
$\mathcal{L}$ onto $\mathcal{L}_1$.  Since $\pi_t^\# \geq \phi_t^\#$ and
the range of $\phi_t^ \#$ is in $P\mathcal{L} = \mathcal{L}_1$ it follows
that $P\pi_t^\# \geq \phi_t^\#$ for all $t > 0$.  Now let $\psi_t^\# = P\pi
_t^\#$ and as before let
$$ \eta_t = (\iota - \psi_t^\#\Lambda
)^{-1}\psi_t^\#.
$$
Since $\pi_t^\# \geq \psi_t^\# \geq \phi_t^\#$ it follows that
$\eta_t \geq_q \nu\vert_t$ for all $t > 0$ and by Theorem 5.6 we see that
if $\eta$ is a limit point of $\eta_t$ as $t \rightarrow 0+$ then $\omega
\geq_q \eta \geq_q \nu$ and since $\eta \geq \nu$ we see any limit point of
$\eta_t$ is not zero.  Let $\eta$ be a limit point of $\eta_t$ as
$t \rightarrow 0+$.  Repeating our earlier argument we see $\eta$ is of the
form  $\eta = (P -\zeta )\omega$ where $\zeta$ is a completely positive
contractive map of $\mathcal{L}$ into $P\mathcal{L}$.  (This time we only have
the bound $\Vert\zeta_t(I)\Vert \leq 1$ but we do not need smaller bound than
one since we know that $\eta \geq \nu$ so $\eta$ is not zero).

Now we need to look more closely at the map $(P - \zeta )$ which is a map
of $\mathcal{L}$ into $P\mathcal{L}$.  Now $\zeta$ is a completely positive
contractive map of $\mathcal{L}$ into $P\mathcal{L}$ and since $P \in
\mathcal{L}^{\prime}$ the mapping $A \rightarrow$ PA is a $*$-isomorphism
of $\mathcal{L}$ into $P\mathcal{L}$ so we can consider $\zeta$ to be a
completely positive contractive map of $P\mathcal{L}$ into itself.  Then
$(P - \zeta )$ restricted to $P\mathcal{L}$ is the mapping $(\iota-\zeta
)$.  Since $\Vert\zeta\Vert \leq 1$ we can not immediately conclude that
this mapping has an inverse.  But we are in luck since $\eta = (P-\zeta
)\omega \geq_q \nu$ and $\nu$ has range $P\mathcal{L}$.  Now let
$\Theta_t^\#$ be the generalized boundary representation of $\eta$.  Note
that $\pi_t^\# \geq \Theta_t^\# \geq \phi_t^\#$ for $t > 0$.  Recall that
for the sequence $t_k \rightarrow 0$ as $k \rightarrow \infty$ we have
$\pi_{t_k}^\#\Lambda \rightarrow L$ and $\phi_{t_k}^\#\Lambda \rightarrow
L_2$ as $k \rightarrow \infty$.  By a routine compactness argument we can
pass to a subsequence of $t_k$ (which we also denote by $t_k$) so that
$\Theta_{t_k}^\# \Lambda$ converges to a limit $L_1$ as $k \rightarrow
\infty$ and since we have passed to a subsequence $\pi_{t_k}^\#\Lambda
\rightarrow L$ and $\phi_{t_k}^\#\Lambda \rightarrow L_2$ as before.  Since
$\pi_t^\# \geq \Theta_t^\# \geq \phi_t^\#$ for $t > 0$ we have $L \geq L_1
\geq L_2$  and by Theorem 5.4 we have $L_1(A) = P_1 L(A) P_1 = P_1 L(A)$
for $A \in B(\mathbb C^p)$ with $P_1 \in \mathcal{L} ^{\prime}$.  Since $L
\geq L_1 \geq L_2$ we have $I_o \geq P_1 \geq P$.  Recalling the
construction of $\eta$ we see the range of $\eta$ is contained in
$P\mathcal{L}$.  Hence, we have $P_1 \leq P$ so $P_1 = P$.  Hence, the
range of $\eta$ is $P\mathcal{L}$ and thus the range of $(\iota-\zeta )$ is
$P\mathcal{L}$ and since $P\mathcal{L}$ is finite dimensional we conclude
the inverse $(\iota-\zeta )^{-1}$ exists.  Now for $0 < \lambda < 1$ we have
\begin{equation*}
(\iota - \lambda\zeta )^{-1} = \iota + \lambda\zeta + \lambda^2\zeta^2 +
\cdots
\end{equation*}
where the series converges in norm and since $(\iota-\zeta )^{-1}$ exists
and the resolvent set is open and the resolvent is continuous it follows
that $(\iota - \lambda\zeta )^{-1}$ converges to $(\iota - \zeta )^{-1}$ as
$\lambda \rightarrow 1$ and since $(\iota - \lambda\zeta )^{-1}$ is
completely positive being the sum of completely positive terms it follows
that $(\iota - \zeta )^{-1}$ is completely positive.   Since $\eta =
(P-\zeta )\omega$ we have $P\omega = (\iota-\zeta )^{-1}\eta$.  Now since
$\eta$ is a $q$-weight map we have from Theorem 5.9 that $\eta (I - \Lambda
(P)) \leq P$.  Hence
\begin{equation*}
P\omega (I - \Lambda (P)) = (\iota-\zeta )^{-1}\eta (I-\Lambda (P)) \leq
(\iota -\zeta )^{-1}(P).
\end{equation*}
Since the mapping $A \rightarrow$ PA is a $*$-isomorphism of $\mathcal{L}$
with $P\mathcal{L}$ for $A \in \mathcal{L}$ we have $\Vert A\Vert = \Vert
PA\Vert$ so we have
\begin{equation*}
\Vert\omega (I-\Lambda (P))\Vert = \Vert P\omega (I-\Lambda (P))\Vert  \leq
\Vert (\iota-\zeta )^{-1}(P)\Vert = \Vert (\iota-\zeta )^{-1}\Vert .
\end{equation*}
Now let $e_o = I_o-P$ and we have $e_o$ is projection in
$\mathcal{L}^{\prime }$ and $e_o \leq I_o$ and
\begin{equation*}
\omega (I - \Lambda (P)) = \omega (I - \Lambda (I_o)) + \omega (\Lambda
(e_o)) \leq \Vert (\iota-\zeta )^{-1}\Vert I
\end{equation*}
and since $\omega (I-\Lambda (I_o)) \geq 0$ we have $0 \leq \omega (\Lambda
(e_o)) \leq \Vert (\iota-\zeta )^{-1}\Vert I$ so $\Vert\omega (\Lambda
(e_o))\Vert$ is finite which violates condition (iii) of the theorem.
Hence, if $\omega$ is not $q$-pure then conditions (i), (ii) and (iii) can
not be satisfied.  \end{proof}

\begin{thm}
Suppose $\vartheta$ is a $\mathcal{L}$ valued $b$-weight map on
$\mathfrak A (\mathbb C^p)$ where $\mathcal{L}$ is a factor of type~I$_q$ of
the form given in Theorem 6.2. and
$$
\mu (A) = \sum_{k\in J} (g_k,Ag_k)
$$
for $A \in \mathfrak A (\mathbb C^p)$ where the $g_k$ are as given in
Theorem 6.2.  Notice that $\mu$ is supported on $E_{11}\otimes I$.  Then
$\vartheta$ is strictly infinite if and only if $\mu$ is strictly infinite and
the $h^{\prime}s$ (as given in Theorem 6.2) are linearly independent over the
$g^{\prime}s$ by which we mean that if $c \in \ell ^2(I)$ and
$$
\sum_{k\in J} c_kg_k = 0
$$
then
$$
\sum_{k\in J} c_kh_{ik}  = 0
$$
for each $i = 1,\cdots ,p$.  Furthermore, the $q$-weight map $\omega$
constructed from $\vartheta$ by Theorem 6.2 will satisfy condition (iii) of
Theorem 6.5 if and only if $\mu (\Lambda (f)) = \infty$ for every non zero
projection $f \leq E_{11}$.
\end{thm}
\begin{proof}  The proof of the lemma excluding the last sentence is simply
a repeat of the proof of Theorem 4.11 in our new setting so all we need do
is prove the last sentence, the sentence beginning with the word
furthermore.  Suppose then that $\mu (\Lambda (f)) < \infty$ with $f$ non
zero and $f \leq E_{11}$.  Let
$$
e = \sum_{i=1}^q E_{i1}fE_{1i}.
$$
We have $e \in \mathcal{L} ^{\prime}$ and one calculates that
\begin{align*}
\vartheta_{ii}(\Lambda (e)) & = \sum_{k\in J} ((g_{ik}+h_{ik}),\Lambda
(e)(g_{ik}+h_{ik}))
\\
&= \sum_{k\in J} ((e\otimes I)(g_{ik}+h_{ik}),\Lambda (e)(e\otimes I)(g_{ik
}+h_{ik}))
\\
&\leq \sum_{k\in J} \Vert (e\otimes I)(g_{ik}+h_{ik})\Vert^2.
\end{align*}
Now since $\mu (\Lambda (f)) < \infty$ and $\mu (I-\Lambda ) < \infty$ we have
$\mu (f\otimes I) < \infty$ which means
$$
\sum_{k\in J} \Vert fg_k\Vert^2 < \infty
$$
and since $g_{ik} = (E_{i1}\otimes I)g_k$ we have
$$
\sum_{k\in J} \Vert (e\otimes I)g_{ik}\Vert^2 = \sum_{k\in J} \Vert
fg_k\Vert ^2 < \infty
$$
and since
$$
\sum_{k\in J} \Vert h_{ik}\Vert^2 < \infty
$$
we have
$$
\sum_{k\in J} \Vert (e\otimes I)h_{ik}\Vert^2 < \infty
$$
from which we conclude that $\vartheta_{ii}(\Lambda (e)) < \infty$ for each
$i = 1,\cdots ,q$.  Since $\omega = \psi\vartheta$ we conclude that $\omega
(\Lambda (e))$ is bounded.  So we have shown that if $\mu (\Lambda (f)) <
\infty$ for a non zero projection $f$ with $f \leq E_{11}$ then there is a
projection $e \in \mathcal{L} ^{\prime}$ with $\Vert\omega (\Lambda (e))\Vert
< \infty$.

Conversely, if there is a non zero projection $e \in \mathcal{L} ^{\prime}$
with $\Vert\omega (\Lambda (e))\Vert < \infty$ then if $f = E_{11}e$ then one
sees that $\mu (\Lambda (f)) < \infty .$ \end{proof}

We are now ready to prove the main result of this paper that every
unital $q$-pure $q$-weight map of index zero is cocycle conjugate
to a unital $q$-pure $q$-weight map of index zero of range rank
one.  We remark that the proof is mostly notation.  What is really going on
is that the real work in proving the theorem has already been done in
the characterization of $q$-pure $q$-weight maps.

\begin{thm}\label{6.7}
Suppose $\omega$ is a unital $q$-pure $q$-weight map over $\mathbb C^{qm}$
of index zero and the range $\mathcal{L}$ of $\omega$ is a factor of
type~I$_q.  $ Then $\omega$ is cocycle conjugate to a unital $q$-pure
$q$-weight map $\eta$ of range rank one over $\mathbb C^m$.
\end{thm}
\begin{proof}  Suppose $\omega$ is as stated in the theorem.  Then from Theorem
6.2 $\omega$ is of the form $\omega = \psi^{-1}\vartheta$ where $\psi$ is an
invertible conditionally negative map of $\mathcal{L}$ into itself with a
completely positive inverse and $\vartheta$ is of the form given in Theorem
6.2.  Since $\omega$ is unital $I_o = I$ and $\psi$ satisfies the additional
requirement that $\psi (I) = \vartheta (I-\Lambda )$.  Furthermore, we know
from Theorem 6.5 that $\psi + \rho (\tilde \Lambda )$ is conditionally zero
and $\vartheta$ is strictly infinite as an $\mathcal{L}$ valued $b$-weight map.
Let the $g_k$ be as given in Theorem 6.2 and let $\mu$ be the $b$-weight on
$\mathfrak{A} (\mathbb C^p)$ defined by
\begin{equation*}
\mu (A) = \sum_{k\in J} (g_k,Ag_k)
\end{equation*}
for $A \in \mathfrak{A} (\mathbb C^p)$.  Note $\mu$ is supported on $E_{11}
\otimes I$ so we can identify $\mu$ as a $b$-weight on $\mathfrak{A} (\mathbb
C^m)$ where $m$ is the rank of $E_{11}.  $ Let $\eta$ be the range rank one
$q$-weight map on $\mathfrak{A} (\mathbb C^m)$ given by $\eta (A) = s_o
^{-1}I_m\mu (A)$ for $A \in \mathfrak{A} (\mathbb C^m)$ where $s_o = \mu (I
- \Lambda )$ and $I_m$ is the unit of $B(\mathbb C^m).  $ Note we have
$\eta (I - \Lambda ) = I_m$ so $\eta$ is unital $q$-weight map over
$\mathbb C ^m$.  From Theorem 6.6 we know that $\mu$ is strictly infinite and
if $f$ is a non zero projection in $B(\mathbb C ^m)$ then $\mu (\Lambda (f))$
is infinite so $\eta$ satisfies the conditions of Theorem 2.5 so $\eta$ is a
unital $q$-pure range rank one $q$-weight map of index zero.

We introduce the notation we will be using in this proof.  The variable $q$
is fixed as $\mathcal{L}$ is a factor of type~I$_q$.  The variable $r =
q+1$ is fixed.  We think of $\mathcal{L}$ the factor of type~I$_q$ with matrix
units $E_{ij}$ as sitting in $\mathcal{L}_1$ the factor of type~I$_r$ with
matrix units $E _{ij}^{\prime}$ and the two sets of matrix units match up so
that $E_{ij}^{\prime} = E_{ij }$ for $i,j = 1,\cdots ,q$ (i.e. we think of
$\mathcal{L}$ as the top left corner of $\mathcal{L}_1$).

We introduce two important projections if $B(\mathbb C^{rm})$
\begin{equation*}
E = E_{11}^{\prime} + E_{22}^{\prime} + \cdots  + E_{qq}^{\prime}\qquad
\text{and } \qquad F = E_{rr}^{\prime}.
\end{equation*}
Note $E$ is the unit of $\mathcal{L}$ and $E+F$ is the unit of $\mathcal{L}
_1$.  We will often consider an operator $A \in B(\mathbb C^{rm})$ as a
$(2\times 2)$-matrix with entries
\begin{equation*}
A=\left[\begin{matrix} EAE&EAF
\\
FAE&FAF
\end{matrix} \right]
\end{equation*}
and we will call $EAE$ the top diagonal of $A$, we call $EAF$ the upper
right corner of $A$, we call $FAE$ the bottom left corner of $A$ and we
call $FAF$ the bottom diagonal of $A$.

We say a mapping $\phi$ of $B(\mathbb C^{rm})$ into itself is a Shur
mapping if $\phi$ preserves this structure meaning
\begin{gather*}
E\phi (A)E = \phi (EAE)\qquad E\phi (A)F = \phi (EAF)
\\
F\phi (A)E = \phi (FAE)\qquad F\phi (A)F = \phi (FAF)
\end{gather*}
for $A \in B(\mathbb C^{rm})$.  We say a mapping $\phi$ of $B(\mathbb C^{rm
}) \otimes \mathfrak{A} (\mathbb C )$ into $B(\mathbb C^{rm})$ is a Shur
mapping if $\phi$ preserves this structure meaning
\begin{gather*}
E\phi (A)E = \phi ((E\otimes I)A(E\otimes I))\qquad E\phi (A)F = \phi
((E\otimes I)A(F\otimes I))
\\
F\phi (A)E = \phi ((F\otimes I)A(E\otimes I))\qquad F\phi (A)F = \phi
((F\otimes I)A(F\otimes I))
\end{gather*}
for $A \in B(\mathbb C^{rm}) \otimes \mathfrak{A} (\mathbb C )$.

Now we will prove the theorem by constructing a unital $q$-pure $q$-weight map
$\omega^{\prime}$ of index zero over $\mathbb C^{rm}$ which is a Shur mapping
so that $E\omega^{\prime}E = \omega$ and $F\omega^{\prime}F = \eta$.  Note
since $\omega ^{\prime}(A^*) = (\omega^{\prime}(A))^*$ and the top and bottom
diagonals have been specified the mapping $\omega^{\prime}$ is determined
once we specify the upper right corner $E\omega^{\prime}F$.

We define $\omega^{\prime} = \psi^{\prime-1}\vartheta^{\prime}$ where
$\vartheta ^{\prime}$ is $b$-weight map over $\mathbb C^{rm}$ given by
\begin{equation*}
\vartheta^{\prime}(A) = \sum_{i,j=1}^r
E_{ij}^{\prime}\vartheta_{ij}^{\prime }(A)
\end{equation*}
for $A \in \mathfrak{A} (\mathbb C^{rm})$ and
\begin{equation*}
\vartheta_{ij}^{\prime}(A) = \sum_{k\in J} ((g_{ik}^{\prime}+h_{ik}^{\prime
}),A(g_{jk}^{\prime}+h_{jk}^{\prime}))
\end{equation*}
where $g_{ik}^{\prime} = g_{ik},\medspace h_{ik}^{\prime} = h_{ik}$ for $k
\in J$ and $i = 1,\cdots ,q$ where the $g_{ik}$ and $h_{ik}$ are from the
expression for $\vartheta$ in Theorem 6.2 and
\begin{equation*}
(g_{rk}^{\prime})_j(x) = E_{r1}^{\prime}g_k(x)\qquad \text{and} \qquad
h_{qk }^{\prime} = 0
\end{equation*}
for $x \geq 0$.  Note we have constructed $\vartheta^{\prime}$ so that
$\vartheta ^{\prime}$ is a Shur mapping so that the top left diagonal of
$\vartheta^{\prime}$ is $\vartheta$ and the bottom right diagonal of
$\vartheta^{\prime}$ is $\eta$.

Now we define $\psi^{\prime}$.  We have seen that $\psi (I) = \vartheta (I-
\Lambda )$ and $\psi+\rho_\vartheta (\tilde{\Lambda})$ is conditionally
zero.  This means that there is a $Q \in B(\mathbb C^{qm})$ so that
\begin{equation*}
\psi (A) = QA + AQ^* - \rho_\vartheta (\tilde{\Lambda}(A))
\end{equation*}
for $A \in \mathcal{L}$.  Calculating $\vartheta (I-\Lambda )$ we find
\begin{equation*}
\vartheta (I-\Lambda ) = I\mu (I-\Lambda ) + Y + Y^* + \rho_\vartheta
(I-\Lambda )  = s_oI + Y + Y^* + \rho_\vartheta (I-\Lambda )
\end{equation*}
where
\begin{equation*}
Y_{ij} = \sum_{q\in J} \int_0^\infty(1-e^{-x})(h_{ik}(x),E_{j1}g_k(x))\,dx
\end{equation*}
from which we find
\begin{equation*}
Q + Q^* - \rho_\vartheta (\Lambda ) = s_oI + Y + Y^* + \rho_\vartheta
(I-\Lambda )
\end{equation*}
so
\begin{equation*}
Q = \tfrac{1}{2} s_oI + B + iC\qquad \text{where} \qquad B = \tfrac{1}{2} (Y
+ Y^* + \rho_\vartheta (I))
\end{equation*}
and $C = C^*$.  Note the if you replace $C$ by $C + \lambda I$ for real
$\lambda$ the mapping $\psi$ unchanged so $C$ is only determined up to
adding a multiple real multiple of $I$.

Since the $h_{rk}$ are zero it follows that $\rho_{\vartheta^{\prime}}
(\tilde {\Lambda}(A))$ is a Shur mapping namely
\begin{equation*}
\rho_{\vartheta^{\prime}} (\tilde{\Lambda}(A)) = E\rho_{\vartheta^{\prime}} (
\tilde{\Lambda}(EAE))E.
\end{equation*}
Now it is clear how to define $\psi^{\prime}$.  We define
\begin{equation*}
\psi^{\prime}(A) = Q^{\prime}A + AQ^{\prime*} - \rho_{\vartheta^{\prime}}
(\tilde {\Lambda}(A))
\end{equation*}
and we define $Q^{\prime}$ so that
\begin{gather*}
EQ^{\prime}E = Q,\qquad EQ^{\prime}F = 0,
\\
FQ^{\prime}E = 0,\qquad FQ^{\prime}F = \tfrac{1}{2} s_oF
\end{gather*}
so
\begin{gather*}
\psi^{\prime}(I) = \vartheta^{\prime}(I-\Lambda (I))
\\
Q^{\prime} = Q + \tfrac{1}{2} s_oF = \tfrac{1}{2} s_oI + B  + iC
\end{gather*}
Note in the above equation $I = E + F$ the unit of $\mathcal{L}_1$
and $B = EBE$ and $C = ECE$.  We note that $\psi^{\prime}$ is a Shur
mapping meaning
\begin{gather*}
E\psi^{\prime}(A)E = \psi^{\prime}(EAE) = E\psi^{\prime}(EAE)E
\\
F\psi^{\prime}(A)F = \psi^{\prime}(FAF) = F\psi^{\prime}(FAF)F = s_oFAF
\\
E\psi^{\prime}(A)F = EQ^{\prime}AF + \tfrac{1}{2} s_oEAF = (s_oE+B+iC)AF  =
\psi^{\prime}(EAF) = E\psi^{\prime}(EAF)F
\\
F\psi^{\prime}(A)E = \tfrac{1}{2} s_oFAE + FAQ^{\prime*} E =
FA(s_oE+B-iC) = \psi^{\prime}(FAE) = F\psi^{\prime}(FAE)E
\end{gather*}
for $A \in \mathcal{L}_1$ so $\psi^{\prime}$ is a Shur mapping.
Now that we have specified $\omega^{\prime}$ we check from
Theorems 6.2 to 6.5 that $\omega^{\prime}$ is a unital $q$-pure
$q$-weight map over $\mathbb C^{rm}$.

To give it a name let $\gamma = E\omega^{\prime}F$ be the top right corner
of $\omega^{\prime}$.  We claim $\gamma$ is a hyper maximal $q$-corner from
$\omega$ to $\eta$.  Since $\omega ^{\prime}$ is Shur mapping and a
$q$-weight map $\gamma$ is a $q$-corner from $\omega$ to $\eta $.  To see
if $\gamma$ is hyper maximal we examine the $q$-subordinates of
$\omega^{\prime}$ whose top right corner equals $\gamma$.  But since
$\omega^{\prime}$ is $q$-pure the only $q$-subordinates of
$\omega^{\prime}$ are of the form $\omega_\lambda^{\prime\prime} =
\psi_\lambda^{\prime\prime -1}\vartheta^{\prime}$ where
$\psi_\lambda^{\prime\prime} = \psi^{\prime}+ \lambda\iota$ with $\lambda
\geq 0$.  Calculating $E\psi_\lambda^{\prime\prime}F$ we find
\begin{equation*}
E\psi_\lambda^{\prime\prime}(A)F = ((s_o+\lambda )E+B+iC)AF =
Z_\lambda AF
\end{equation*}
for $A \in \mathcal{L}_1$ so $E\psi_\lambda^{\prime\prime-1}F$ is given by
\begin{equation*}
E\psi_\lambda^{\prime\prime-1}(A)F = ((s_o+\lambda )E
+B+iC)^{-1}AF = Z_\lambda^{-1}AF
\end{equation*}
where by $Z_\lambda^{-1}$ we mean the inverse of this operator
$Z_\lambda$ considered as an operator from the range of $E$ to the
range of $E$. (The inverse on the larger space make no sense as
$Z_\lambda F = 0.)$  Let $z_i$ be the eigenvalues of $Z_o$.  Note
that Re(z$_i) > 0$ and the eigenvalues of $Z _\lambda$ are
$z_i+\lambda$ and the eigenvalues of $Z_\lambda^{-1}$ are
$(z_i+\lambda )^{-1}$.  In any event it is clear that
$E\psi_\lambda^{\prime\prime-1}F \neq E\psi^{\prime-1}F$ for
$\lambda > 0$ so the only $q$-subordinate of $\omega^{\prime}$
with corner $\gamma$ is $\omega^{\prime}$ and, hence, $\gamma$ is
a hyper maximal $q$-corner from $\omega$ to $\eta$.  Hence,
$\omega$ and $\eta$ are cocycle conjugate.  \end{proof}

In conclusion we have shown that every $q$-pure $q$-weight map over a
finite dimensional Hilbert space $K$ is cocycle conjugate to a $q$-pure
range rank one $q$-weight map and these are described in Theorem 2.5.
Theorem 2.6 gives necessary and sufficient conditions for two such
$q$-weight maps to be cocycle conjugate.  It follows that $q$-pure spatial
$E_0$-semigroups coming from $q$-weight maps over finite dimensional
Hilbert spaces are fairly well understood.  The next burning question is
whether this remains true in the infinite dimensional case.

\end{document}